\documentclass[a4paper,11pt,oneside]{amsart}

\usepackage{hyperref}
\hypersetup{colorlinks=true,allcolors=blue}
\usepackage{hypcap}

\usepackage[leqno]{amsmath}
\usepackage{amsfonts,amssymb,amsthm}
\usepackage{mathrsfs}
\usepackage{esint}
\usepackage{mathtools}
\usepackage{graphicx} 
\usepackage{galois}
\usepackage[utf8]{inputenc}
\usepackage{enumitem}
\usepackage[margin=1.21in]{geometry}

\theoremstyle{plain}
\newtheorem{theorem}{\protect\theoremname}[section]
\theoremstyle{plain}
\newtheorem{proposition}[theorem]{\protect\propositionname}
\theoremstyle{definition}
\newtheorem{definition}[theorem]{\protect\definitionname}
\theoremstyle{plain}
\newtheorem{corollary}[theorem]{\protect\corollaryname}
\theoremstyle{remark}
\newtheorem{remark}[theorem]{\protect\remarkname}
\theoremstyle{plain}
\newtheorem{lemma}[theorem]{\protect\lemmaname}

\providecommand{\corollaryname}{Corollary}
\providecommand{\definitionname}{Definition}
\providecommand{\lemmaname}{Lemma}
\providecommand{\propositionname}{Proposition}
\providecommand{\remarkname}{Remark}
\providecommand{\theoremname}{Theorem}

\numberwithin{equation}{section}
\numberwithin{figure}{section}
\numberwithin{table}{section}

\newcommand{\R}{\mathbb{R}}
\newcommand{\N}{\mathbb{N}}
\newcommand{\Z}{\mathbb{Z}}
\newcommand{\C}{\mathbb{C}}
\renewcommand{\S}{\mathbb{S}}

\newcommand{\Rd}{\R^d}

\newcommand{\Rdxi}{\Rd_\xi}
\newcommand{\Torus}{\mathbb{T}}
\newcommand{\Td}{\Torus^d}
\newcommand{\bydef}{\mathrel{\mathop:}=}
\newcommand{\K}{\mathcal{K}}
\renewcommand{\b}{b}

\newcommand{\demi}{\frac{1}{2}}
\newcommand{\norm}{\mathbf{n}}
\newcommand{\xy}{{x,y}}
\newcommand{\Pext}{P_{\mathrm{ext}}}
\newcommand{\phiw}{\varphi_\omega}
\newcommand{\supp}{\mathrm{supp\,}}

\newcommand{\udlu}{\underline{u}}

\newcommand{\jp}[1]{\langle#1\rangle}

\newcommand{\Holder}[1]{W^{#1,\infty}}
\newcommand{\Symbol}[2]{\dot{\Gamma}^{#1}_{#2}}
\newcommand{\Cinf}{C^\infty}
\newcommand{\M}[4]{M^{#1}_{#2,#3}(#4)}
\newcommand{\Mdot}[4]{\dot{M}^{#1}_{#2,#3}(#4)}
\renewcommand{\H}[1]{H^{#1}}
\newcommand{\Hdot}[1]{\dot{H}^{#1}}
\newcommand{\Linf}{L^\infty}
\newcommand{\Ccinf}{C_c^\infty}
\newcommand{\Lone}{L^1}
\newcommand{\Ltwo}{L^2}
\newcommand{\Ldottwo}{\dot{L}^2}

\newcommand{\Cls}{\mathscr{C}}

\renewcommand{\L}{\mathcal{L}}

\newcommand{\D}{\mathscr{D}}

\newcommand{\drv}{\mathrm{d}}
\renewcommand{\d}{\,\drv}
\newcommand{\dx}{\d x}

\newcommand{\dt}{\d t}
\newcommand{\ds}{\d s}

\newcommand{\pt}{\partial_t}
\newcommand{\ps}{\partial_s}
\newcommand{\pn}{\partial_\norm}
\newcommand{\px}{\partial_x}
\newcommand{\py}{\partial_y}
\newcommand{\pxi}{\partial_\xi}
\newcommand{\ddt}{\frac{\drv}{\drv t}}

\newcommand{\intTorus}{\int_{\Td}}

\newcommand{\T}[1]{T_{#1}}
\newcommand{\Fourier}{\mathcal{F}}

\newcommand{\Id}{\mathrm{Id}}
\newcommand{\op}{\mathrm{Op}}
\newcommand{\bop}{\overline{\op}}
\newcommand{\tr}{\mathrm{tr}}

\newcommand{\MeanCurv}[1]{\nabla_x \cdot \Big( \frac{\nabla_x #1}{\sqrt{1 + |\nabla_x #1|^2}} \Big)}
\newcommand{\SurfArea}[1]{1 + |\nabla_x #1|^2}
\newcommand{\VertVelo}{\nabla_x \eta \cdot \nabla_x \psi + G(\eta) \psi}

\newcommand{\Sol}{\mathcal{S}}
\newcommand{\Range}{\mathcal{R}}
\newcommand{\B}{\mathcal{B}}
\newcommand{\A}{\mathcal{A}}
\newcommand{\E}{\mathcal{E}}
\newcommand{\Zh}{\mathcal{Z}_h}

\renewcommand{\Re}{\mathrm{Re}}
\renewcommand{\Im}{\mathrm{Im}}

\begin{document}
	
\title{\textsc{Control of Three Dimensional Water Waves}}
\author{HUI ZHU} 
\thanks{The author is partially supported by the grant ``ANAÉ'' ANR-13-BS01-0010-03 of the Agence Nationale de la Recherche. This research is carried out during the author's PhD studies, financed by the Allocation Doctorale of the École Normale Supérieure.}

%\begin{center}
%\today
%\end{center}

\begin{abstract}

We study the exact controllability for spatially periodic water waves with surface tension, by localized exterior pressures applied to free surfaces. We prove that in any dimension, the exact controllability holds within arbitrarily short time, for sufficiently small and regular data, provided that the region of control satisfies the geometric control condition.
This result was previously obtained by Alazard, Baldi, and Han-Kwan~\cite{ABH-Control} for 2-D water waves. Our proof combines an iterative scheme, that reduces the controllability of the original quasi-linear equation to that of a sequence of linear equations, with a semi\-classical approach for the linear control problems.

\end{abstract}

\maketitle

%%\tableofcontents

\section{Introduction}

We consider the Zakharov~\cite{Zakharov} / Craig–Sulem~\cite{C-S} formulation of the gravity water wave system with surface tension. It is defined as follows on the torus $ \Td = \Rd / 2\pi\Z^d $,
\begin{equation}
\label{eq:equation-water-wave}
\begin{cases}
\pt \eta - G(\eta) \psi = 0, \\
\pt \psi + g \eta - H(\eta) + \frac{1}{2} |\nabla_x \psi|^2 - \frac{1}{2} \frac{( \VertVelo )^2}{\SurfArea{\eta}} = \Pext.
\end{cases}
\end{equation}
Here~$ g $ is the gravitational acceleration, $ H(\eta) = \nabla_x \cdot \big( \frac{\nabla_x \eta}{\sqrt{1 + |\nabla_x \eta|^2}} \big) $ is the mean curvature of the surface 
\begin{equation*}
\Sigma_t = \{(x,y) \in \Td \times \R : y = \eta(t,x)\},
\end{equation*}
and $ G(\eta) $ is the Dirichlet-Neumann operator, defined below by~\eqref{eq:def-G(eta)}, of the domain
\begin{equation*}
\Omega_t = \{(x,y) \in \Td \times \R : -b < y < \eta(t,x) \},
\end{equation*}
with depth $ b \in {} ]0,+\infty[ $. Our main theorem states that, any sufficiently small data can be generated by a suitable localized exterior pressure~$ \Pext $.

\begin{definition}
\label{def:GCC}
We say that an open set $ \omega \subset \Td $ satisfies the \textit{geometric control condition} if every geodesic of~$ \Td $ (which are straight lines, for we endow~$ \Td $ with the standard metric) eventually enters $ \omega $. More precisely, this means that for every $ (x,\xi) \in \Td \times \S^{d-1} $, there exists some~$ t \in {}]0,\infty[ $, such that $ x + t\xi \in \omega $.
\end{definition}

\begin{theorem}
\label{thm:main}
Suppose that $ d \ge 1 $, $ T > 0 $, $ s $ is sufficiently large, and $ \omega \subset \Td $ satisfies the geometric control condition, then for some $ \varepsilon_0 > 0 $ sufficiently small and for all $ (\eta_i,\psi_i) \in \H{s+1/2}(\Td) \times \H{s}(\Td) $ satisfying 
\begin{equation*}
\|(\nabla_x \eta_i,\nabla_x \psi_i)\|_{\H{s-1/2}\times\H{s-1}} < \varepsilon_0, \quad (i=0,1)
\end{equation*}
and $ \intTorus \eta_0 \dx = \intTorus \eta_1 \dx = 0 $, there exists $ \Pext \in C ([0,T],\H{s}(\Td)) $, such that
\begin{enumerate}[nosep]
\item $ \Pext $ is real valued, and for all $ t \in [0,T] $, $ \supp \Pext(t,\cdot) \subset \omega $;
\item there exists a unique solution to~\eqref{eq:equation-water-wave}, $ (\eta,\psi) \in C([0,T],\H{s+1/2}(\Td)\times\H{s}(\Td)) $, such that $ (\eta,\psi)|_{t=0} = (\eta_0,\psi_0) $, and $ (\eta,\psi)|_{t=T} = (\eta_1,\psi_1) $. 
\end{enumerate}
\end{theorem}

\begin{remark}
Theorem~\ref{thm:main} remains valid for infinite depth, that is $ b = +\infty $, with exactly the same proof. However, we restrict ourselves to finite depth for simplicity.
\end{remark}

\begin{remark}
The same result for 2-D water waves, that is $ d=1 $, was previously obtained by Alazard, Baldi, and Han-Kwan~\cite{ABH-Control}, where the geometric control condition is implicit, as it is always satisfied on~$ \Torus^1 $. We show in Appendix~\ref{app:necessity-of-GCC} that on~$ \Torus^2 $, for the exact controllability of the linearized equation around the flat surface (that is $ \eta = 0 $) of~\eqref{eq:equation-water-wave} with infinite depth, the geometric control condition is necessary, and consequently it is natural for the non-linear equation.
\end{remark}

\begin{remark}
The condition for the mean values of~$ \eta_0 $ and~$ \eta_1 $ is necessary since the zero frequency of~$ \eta $ is conserved in time.
\end{remark}

\begin{remark}
We shall only prove the existence of the solution $ (\eta,\psi) $, by an iterative construction performed in Section~\ref{sec:iteration}, because the uniqueness is a consequence of~\cite{ABZ-Surface-Tension}.
\end{remark}

This is a natural control result for a quasi-linear physical equation. Although many works have been done in the control theory of nonlinear equations, including equations describing water waves in some asymptotic regions, like the Benjamin--Ono equation, the Korteweg--de Vries equation, the nonlinear Schr\"{o}dinger equation, etc., the only work so far for the full water wave model is done by Alazard, Baldi and Han-Kwan~\cite{ABH-Control}, who proves the exact controllability for the system~\eqref{eq:equation-water-wave} on~$ \Torus^1 $. Our Theorem~\ref{thm:main} extends their result to higher dimensions.

A main ingredient in~\cite{ABH-Control} is Ingham's inequality, which is a tool specific to $ d=1 $. To tackle the lack of such estimate in higher dimensions, we have to distinguish the high frequency regime and the low frequency regime. The high frequency regime requires to implement in this quasi-linear framework the semi-classical approach (see Lebeau~\cite{Lebeau-Control-Schrodinger}, Burq--Zworski~\cite{BZ-Control-Schrodinger}), while the low frequency regime is studied by the uniqueness-compactness argument (see Bardos--Lebeau--Rauch~\cite{BLR}). It has to be noticed that the usual Carleman estimates do not seem to apply to the paradifferential context we are working with, and this uniqueness result has to be proven by a different method. 

\subsection{From Euler to Zakharov / Craig--Sulem}

We present here briefly the Eulerian formulation of the water wave system with surface tension to give some physical intuitions. Then we define the Dirichlet-Neumann operator, and derive the Zakharov / Craig--Sulem formulation~\eqref{eq:equation-water-wave}. 

Let~$ \Sigma_t $ and~$ \Omega_t $ be defined as above, and let $ \Gamma = \Td \times \{-b\} $ be the flat bottom. Denote by $ v : \Omega_t \to \Rd $ the Eulerian vector field, by $ P : \Omega_t \to \R $ the internal pressure of the fluid, and by $ \norm : \partial\Omega_t \to \S^d $ the exterior unit normal vector to the boundary $ \partial \Omega_t = \Sigma_t \cup \Gamma $. Then the Eulerian formulation of the water wave system is the following system of~$ (\eta,v) $,
\begin{equation}
\label{eq:equation-water-wave-euler}
\begin{cases}
\pt v + v \cdot \nabla_\xy v + \nabla_\xy P = -g \mathbf{e}_y & \mathrm{in\ } \Omega_t, \\
\nabla_\xy \cdot v = 0 & \mathrm{in\ } \Omega_t, \\
\pt \eta = \sqrt{\SurfArea{\eta}} \  v|_{\Sigma_t} \cdot \norm, & \mathrm{on\ } \Sigma_t, \\
v \cdot \norm = 0 & \mathrm{on\ } \Gamma, \\
-P|_{\Sigma_t} = \kappa H(\eta) + \Pext & \mathrm{on\ } \Sigma_t,
\end{cases}
\end{equation}
where~$ \mathbf{e}_y = (0,1) $ is the unit vector in the $ y $-direction, and~$ \kappa > 0 $ is the surface tension coefficient. The physical interpretation of~\eqref{eq:equation-water-wave-euler} is that
\begin{enumerate}[nosep]
\item $ v $ satisfies the incompressible Euler equation in the domain~$ \Omega_t $,
\item fluid particles which are initially on the surface will stay on the surface,
\item the bottom is impenetrable by fluid particles,
\item the internal pressure, the exterior pressure, and the surface tension balance out on the surface.
\end{enumerate}
We suppose furthermore that the vector field~$ v $ admits a scalar potential $ \phi : \Omega_t \to \R $, i.e., $ v = \nabla_\xy \phi $, which implies furthermore that $ v $ is irrotational, i.e., $ \nabla_\xy \times v = 0. $
Denote by $ \psi $ the trace of $ \phi $ to the free surface, in the sense that,
\begin{equation*}
\psi(t,x) = \phi(t,x,\eta(t,x)),
\end{equation*}
then $ \phi $ satisfies the Laplace equation, with a mixed boundary condition,
\begin{equation}
\label{eq:equation-laplace-DN}
\Delta_{x,y} \phi = 0, \quad \phi|_{\Sigma_t} = \psi, \quad \pn \phi |_\Gamma = 0.
\end{equation}
Define the Dirichlet-Neumann operator $ G(\eta) $ by
\begin{align}
\label{eq:def-G(eta)}
G(\eta) \psi(t,x)
& = \sqrt{1+|\nabla_x\eta(t,x)|^2} \ \pn \phi|_{y = \eta(t,x)} \\
& = (\py \phi)(t,x,\eta(x)) - \nabla_x \eta(t,x) \cdot (\nabla_x\phi)(t,x,\eta(x)). \nonumber
\end{align}
Combining~\eqref{eq:equation-water-wave-euler}, \eqref{eq:equation-laplace-DN} and~\eqref{eq:def-G(eta)}, and assume $ \kappa = 1 $, we obtain the Zakharov / Craig--Sulem formulation~\eqref{eq:equation-water-wave} in variables $ (\eta,\psi) $.

The study of the Cauchy problem of~\eqref{eq:equation-water-wave}, initiated by Kano--Nishida~\cite{K-N} and Yosihara~\cite{Yosihara-1,Yosihara-2}, has greatly progressed over decades. To name a few, the local well-posedness in the framework of Sobolev spaces, without smallness assumptions of initial data, has been established by Beyer--Günther~\cite{B-G} in the case with surface tension, and by Wu~\cite{Wu-2D,Wu-3D} in the case without surface tension. For recent progresses, we refer to Ifrim--Tataru~\cite{Ifrim-Tataru}, Ionescu--Pusateri~\cite{Ionescu-Pusateri}, Wang~\cite{Wang-3D-Global}, de~Poyferré--Nguyen~\cite{de Poyferre-Nguyen} and the references therein. Here we are influenced by the paradifferential approach developed by Alazard--Métivier~\cite{AM-Paralinearization}, Alazard--Burq--Zuily~\cite{ABZ-Surface-Tension, ABZ-Without-Surface-Tension, ABZ-non-local}, which allows them to prove the local well-posedness in low Sobolev regularity, without regularity assumptions for the bottom.

\subsection{Outline of Paper}

In Section~\ref{sec:Strategy-of-Proof}, we outline the method of our proof. In Section~\ref{sec:Paradifferentialization-and-Symmetrization}, we reformulate the problem by the para\-differential calculus. In Section~\ref{sec:Ltwo-linear-control}, we prove the null controllability in $ \Ltwo $ for linearized control problems of~\eqref{eq:equation-water-wave}. In Section~\ref{sec:H^s-linear-control}, we prove that for $ \H{s} $ initial data, the control obtained by Section~\ref{sec:Ltwo-linear-control} is also of regularity~$ \H{s} $. In Section~\ref{sec:iteration}, we prove with an iterative construction the null controllability in $ \H{s} $ of the quasilinear para\-differential equation obtained in Section~\ref{sec:Paradifferentialization-and-Symmetrization}. Finally in Section~\ref{sec:theorem-paradiff-implies-theorem}, we prove the exact controllability of~\eqref{eq:equation-water-wave} and conclude the proof of Theorem~\ref{thm:main}.

In Appendix~\ref{app:necessity-of-GCC}, we show that on $ \Torus^2 $, when $ b = \infty $, the geometric control condition is necessary for the controllability of the linearized equation of~\eqref{eq:equation-water-wave} around the flat surface. In Appendix~\ref{sec:paradifferential-calculus}, we recall the basics of the para\-differential calculus. In Appendix~\ref{app:linear-equations}, we prove the well-posedness of some linear evolution equations, which are used in the study of the linearized equations of~\eqref{eq:equation-water-wave}.

\section*{Acknowledgment}

The author would like to express his sincere gratitude to his Ph.D.\ advisors Thomas Alazard and Nicolas Burq for their continuous support and advices. He would like to thank Claude Zuily and Patrick Gérard for all their help and encouragement, and thank Huy Quang Nguyen for some useful discussions at the beginning of this project. He would also like to thank Jean-Marc Delort and Daniel Tataru for their careful reading of the manuscript.

\section{Strategy of Proof and Some Notations}

\label{sec:Strategy-of-Proof}

The general strategy of our proof of Theorem~\ref{thm:main} is to combine the iterative scheme of~\cite{ABH-Control} with Lebeau's method for the linear control problems, where we adapt the semi\-classical approach in the high frequency regime, and use a perturbative argument to prove the unique continuation property for the low frequency regime. We explain in the following some more details.

\subsection{Paralinearization and Reduction to Null Controllability}

By the time reversibility of~\eqref{eq:equation-water-wave}, the procedure of paralinearization and symmetrization systematically developed in~\cite{AM-Paralinearization, ABZ-Surface-Tension, ABZ-Without-Surface-Tension, ABZ-non-local} by Alazard, Burq, Métivier and Zuily, we can reduce the exact controllability of~\eqref{eq:equation-water-wave} to the \textit{null controllability} of a paralinearized equation. By the null controllability, we mean the exact controllability with null final data. For the para\-differential calculus, we refer to Bony~\cite{Bony-Paradiff}, Métivier~\cite{Metivier-Paradifferential-Calculus}, and H\"{o}rmander~\cite{Hormander-Hyperbolic-textbook}, see also Appendix~\ref{sec:paradifferential-calculus} where some basic results are presented.

Recall that the zero frequency of~$ \eta $ is preserved in time, and observe that the zero frequency of~$ \psi $ is of no physical importance and at the same time has no contribution in the nonlinear terms of the equation. Therefore it is natural to work in Sobolev spaces of functions with null zero frequencies,
\begin{equation}
\label{eq:Hdot-definition}
\Hdot{\sigma}(\Td) \bydef \{f\in\H{\sigma}(\Td) : \pi(D_x) f = f\}, \quad \sigma \in \R,
\end{equation}
where $ \pi(D_x) $ is a Fourier multiplier that projects to Fourier modes of nonzero frequencies (see Appendix~\ref{sec:paradifferential-calculus} for details). We equip $ \Hdot{\sigma}(\Td) $ with the usual Sobolev norm inherited from $ \H{s}(\Td) $. Observe that $ f \in \Hdot{\sigma}(\Td) $ means $ \int_{\Td} f \dx = 0 $. We also use the notation $ \Ldottwo(\Td) = \Hdot{0}(\Td) $. 

Following~\cite{ABZ-Surface-Tension,ABH-Control}, we paralinearize~\eqref{eq:equation-water-wave} to obtain a para\-differential equation for the complex variable $ u = u(\psi,\eta) $, 
\begin{equation}
\label{eq:def-u}
u = \T{q} \omega - i \T{p} \eta.
\end{equation}
Here $ \omega = \psi - T_B \eta $ is called \textit{the good unknown of Alinhac}\footnote{We use the same notation~$ \omega $ for the good unknown of Alinhac and the domain of control for it causes no ambiguity, and it is a standard notation in both cases in the literature.}, where~$ B $ is the trace to the free surface~$ \Sigma $ of the vertical velocity~$ \py\phi $ (see~\eqref{eq:def-B-V} and~\eqref{def:good-unknown-of-Alinhac}); while $ \T{q} $ and $ \T{p} $ are para\-differential operators depending solely on~$ \eta $, of orders~$ 0 $ and~$ 1/2 $ respectively, so that $ u \in \Hdot{s}(\Td) $ whenever $ (\eta,\psi) \in \H{s+1/2}(\Td) \times \H{s}(\Td) $. We show that this transform $ (\eta,\psi) \mapsto u $ is invertible except for the zero frequencies. So we first establish the null controllability for~$ u $, and then recover the zero frequencies to obtain the exact controllability for~$ (\eta,\psi) $. To do this, we seek a null control for~$ u $ of the following form
\begin{equation}
\label{eq:form-Pext}
\Pext(t,x) =  \chi_T^{}(t) \phiw(x) \Re\, F(t,x),
\end{equation}
where
\begin{enumerate}[nosep]
\item $ F \in C([0,T],\Hdot{s}(\Td)) $ is complex valued;
\item $ \chi_T^{}(\cdot) = \chi_1^{}(\cdot/T) \in \Cinf(\R) $ where $ \chi_1^{}(t) = 1 $ for $ t \le 1/2 $ and $ \chi_1^{}(t) = 0 $ for $ t \ge 3/4 $;
\item $ 0 \le \phiw \in \Cinf(\Td) $ satisfies $ 1_{\omega'} \le \phiw \le 1_\omega $ where $ \omega' $ satisfies the geometric control condition, and $ \overline{\omega'} \subset \omega $. Such~$ \omega' $ exists because~$ \Td $ is compact.
\end{enumerate} 
We fix $ \chi_T^{} $ and~$ \phiw $ and seek~$ F $, so that with $ \Pext $ of the form~\eqref{eq:form-Pext}, $ u $ satisfies the following nonlinear para\-differential equation,
\begin{equation}
\label{eq:intro-equation-paradiff-nonlinear}
(\pt + P(u) + R(u)) u = (\B(u) + \beta(u)) F.
\end{equation}
Here $ P(u) $ is a para\-differential operator of order~$ 3/2 $,
\begin{equation*}
P(u) = i \T{\gamma(u)} + \nabla_x \cdot \T{V(u)} + \mathrm{\ lower\ order\ terms},
\end{equation*} 
with $ \gamma(u) $ being a symbol of degree~$ 3/2 $ that depends on~$ \nabla_x\eta $ (hence depends on~$ u $) and $ V(u) $ being the trace to the free surface~$ \Sigma $ of the horizontal velocity~$ \nabla_x \phi $ (see~\eqref{eq:def-B-V}); while
\begin{equation*}
\B(u)F = \chi^{}_T \T{q}\phiw \Re F.
\end{equation*}
Under the smallness condition $ u = O( \varepsilon_0)_{\H{s}} $, we have 
\begin{equation*}
\T{\gamma(u)} - |D_x|^{3/2} = O(\varepsilon_0)_{\L(\Hdot{s},\Hdot{s-3/2})}, \quad
V(u) = O(\varepsilon_0)_{\H{s-1}};
\end{equation*}
while remainders $ R(u) $ and $ \beta(u) $ satisfy for some $ \vartheta > 0 $,
\begin{equation}
\label{eq:smallness-of-R}
\|R(u)\|_{\L(\Hdot{s}, \Hdot{s})} \lesssim \|u\|_{\H{s}}^\vartheta,
\quad  
\|\beta(u)\|_{\L(\Hdot{s}, \Hdot{s+1/2})} \lesssim \|u\|_{\H{s}}.
\end{equation}
Therefore perturbation arguments can be used to simplify the situation.

\begin{remark}
The reason to redo the paralinearization of~\eqref{eq:equation-water-wave} rather than borrowing directly the results from~\cite{ABZ-Surface-Tension} is due to two considerations. Firstly~\cite{ABZ-Surface-Tension} studies the Cauchy problem, to which only the regularity of the remainder $ R(u)u $ is important, while in this paper, the smallness is also required to treated it as a perturbation. Secondly, due to the existence of the exterior pressure $ \Pext $ (in~\cite{ABZ-Surface-Tension} $ \Pext = 0 $), the same estimates for $ \pt \psi $ no longer apply, as they now for $ \pt \psi - \Pext $, resulting in the appearance of the term~$ \beta(u) F $.
\end{remark}

We will first prove the null controllability of~\eqref{eq:intro-equation-paradiff-nonlinear} (that is, Theorem~\ref{thm:main-paradiff} below), then prove that it implies the exact controllability of~\eqref{eq:equation-water-wave} (that is, Theorem~\ref{thm:main}) in Section~\ref{sec:theorem-paradiff-implies-theorem}.

\begin{theorem}
\label{thm:main-paradiff}
Suppose that $ d \ge 1 $, $\omega \subset \Td $ satisfies the geometric control condition, $ T > 0 $, and~$ s $ sufficiently large, then for some $ \varepsilon_0 > 0 $ sufficiently small and for all initial data $ u_0 \in \Hdot{s}(\Td) $ with
$ \|u_0\|_{\H{s}} < \varepsilon_0 $,
there exists an $ F \in C([0,T],\Hdot{s}(\Td)) $, such that the unique solution $ u \in C([0,T],\Hdot{s}(\Td)) $ to~\eqref{eq:intro-equation-paradiff-nonlinear} with initial data $ u(0)=u_0 $ vanishes at time~$ T $, that is $ u(T) = 0 $.
Moreover $ F $ can be so chosen that
\begin{equation}
\label{eq:soft-estimate-u-F}
\|u\|_{C([0,T],\H{s}) \cap \Holder{1}([0,T],\H{s-3/2})}  \lesssim \varepsilon_0, \quad 
\|F\|_{C([0,T],\H{s})}  \lesssim \varepsilon_0.
\end{equation}
\end{theorem}

\begin{remark}
In the statement of Theorem~\ref{thm:main-paradiff}, and throughout this article, the relation $ X \lesssim Y $ is used to simplify the relation~$ X \le C(d,\omega,T,\b) Y $ for some constant~$ C $ which, whenever~$ \varepsilon_0 $ (and~$ h $ in the semi\-classical setting) is sufficiently small, depends only on~$ d $, $ \omega $, $ T $, $ \b $, and can be treated as a universal constant.
\end{remark}

\subsection{Iterative Scheme.}
To prove Theorem~\ref{thm:main-paradiff}, we adapt the iterative scheme of~\cite{ABH-Control} which reduces the study to the control problem of linear equations. To simplify the notation, we first introduce the following spaces.

\begin{definition}
For $ \sigma \in \R $, $ \varepsilon_0 > 0 $, $ T > 0 $, we say $ u \in \Cls^{1,\sigma}(T,\varepsilon_0) $, resp.\ $ F \in \Cls^{0,\sigma}(T,\varepsilon_0) $, if $ u \in C([0,T],\Hdot{\sigma}(\Td)) \cap \Holder{1}([0,T],\Hdot{\sigma-3/2}(\Td)) $, resp.\ $ F \in C([0,T],\Hdot{\sigma}(\Td)) $ and
\begin{equation*}
\|u\|_{C([0,T],\H{\sigma})} + \|\pt u\|_{\Linf([0,T],\H{\sigma-3/2})} < \varepsilon_0,
\quad \mathrm{resp.} \quad  \|F\|_{C([0,T],\H{\sigma})} < \varepsilon_0.
\end{equation*}
\end{definition}

For~$ s $ sufficiently large and $ \varepsilon_0 > 0 $, we fix $ \udlu \in \Cls^{1,s}(T,\varepsilon_0) $ and consider the null controllability of the linear equation
\begin{equation}
\label{eq:intro-linear-equation-perturbed}
(\pt + P(\udlu) + R(\udlu)) u = (\B(\udlu) + \beta(\udlu)) F.
\end{equation}
We show that for $ \varepsilon_0 $ sufficiently small, there exists a linear operator 
\begin{equation}
\label{eq:intro-def-Phi}
\Phi(\udlu) : \Hdot{s}(\Td) \to C([0,T],\Hdot{s}(\Td)),
\end{equation}
which \textit{null-controls}~\eqref{eq:intro-linear-equation-perturbed}, that is, for any $ u_0 \in \Hdot{s}(\Td) $,
\begin{equation*}
F = \Phi(\udlu) u_0 \in C([0,T],\Hdot{s}(\Td))
\end{equation*}
sends the initial data $ u(0) = u_0 $ at time $ t = 0 $ to final data~$ u(T) = 0 $ at time $ t=T $ by~\eqref{eq:intro-linear-equation-perturbed}. Then the iterative scheme proceeds as follows. Let~$ u_0 \in \Hdot{s}(\Td) $ such that $ \|u_0\|_{\H{s}} < \varepsilon_0 $. We set $ u^0 \equiv 0 $, $ F^0 \equiv 0 $, and for $ n \ge 0 $ set $ (u^{n+1},F^{n+1}) \in C([0,T],\Hdot{s}(\Td)) \times C([0,T],\Hdot{s}(\Td)) $ by letting
\begin{equation*}
F^{n+1} = \Phi(u^n) u_0,
\end{equation*}
and letting $ u^{n+1} $ be the solution to the equation
\begin{equation*}
(\pt + P(u^n) + R(u^n)) u^{n+1} = \B(u^n) F^{n+1}, \quad u^{n+1}(0) = u_0,\ u^{n+1}(T) = 0.
\end{equation*}
Be careful that for $ F^{n+1} $ to be well defined, we must check that for some constant $ C > 0 $, independent of~$ n $ and small~$ \varepsilon_0 $, $ u^n \in \Cls^{1,s}(T,C\varepsilon_0) $.
To prove the convergence of this scheme, we need the following contraction estimate for the control operator,
\begin{equation*}
\|\Phi(u^{n+1}) - \Phi(u^n)\|_{\L(\Hdot{s}, C([0,T],\Hdot{s-3/2}))} \lesssim \|u^{n+1}-u^n\|_{C([0,T],\H{s-3/2})},
\end{equation*}
which shows that $ \{(u^n,F^n)\}_{n \ge 0} $ is a Cauchy sequence in 
\begin{equation*}
C([0,T],\Hdot{s-3/2}(\Td)) \cap \Holder{1}([0,T],\Hdot{s-3}(\Td)) \times C([0,T],\Hdot{s-3/2}(\Td)),
\end{equation*}
and converges to some $ (u,F) $ that satisfies~\eqref{eq:intro-equation-paradiff-nonlinear} in the distributional sense. To recover the~$ \H{s} $-regularity of $ u $, we study~\eqref{eq:intro-equation-paradiff-nonlinear} by treating it as the linear equation~\eqref{eq:intro-equation-dual-non-perturbed} with $ \udlu = u $. The $ \H{s} $-regularity of~$ F $ comes from~\eqref{eq:intro-def-Phi} after proving that $ u \in \Cls^{1,s}(T,C \varepsilon_0) $, and that $ F = \Phi(u) u_0 $.

The construction of $ \Phi(\udlu) $ is the main effort of this paper. Treating $ R(\udlu) $ and $ \beta(\udlu) $ as perturbations, it suffices to study the following unperturbed equation,
\begin{equation}
\label{eq:intro-equation-paradiff-linear-non-perturbed}
(\pt + P(\udlu)) u = \B(\udlu) F.
\end{equation}
Indeed, we show that there exists a linear control operator,
\begin{equation*}
\Theta(\udlu) : \Ldottwo(\Td) \to C([0,T],\Ldottwo(\Td))
\end{equation*}
which null-controls~\eqref{eq:intro-equation-paradiff-linear-non-perturbed} for $ \Ltwo $ initial data, and moreover satisfies
\begin{equation*}
\Theta(\udlu)|_{\Hdot{s}} : \Hdot{s}(\Td) \to C([0,T],\Hdot{s}(\Td)),
\end{equation*}
Then for some $ \E(\udlu) = O(\varepsilon_0^\vartheta)_{\L(\Hdot{s},\Hdot{s})} $ with $ \vartheta > 0 $, 
\begin{equation*}
\Phi(\udlu) = \Theta(\udlu) (1+\E(\udlu))^{-1}.
\end{equation*} 
We first apply the \textit{Hilbert uniqueness method} to construct~$ \Theta(\udlu) $, and use a commutator estimate to prove its $ \H{s} $-regularity, with details explained below.

\subsection{Hilbert Uniqueness Method}

\label{sec:Ltwo-controllability-HUM}

The Hilbert uniqueness method is a purely functional analysis argument, due to Lions~\cite{Lions-HUM} which, when applied to our situation, establishes the equivalence between the $ \Ltwo $-null controllability of~\eqref{eq:intro-equation-paradiff-linear-non-perturbed} and the $ \Ltwo $-observability (see below) of its dual equation,
\begin{equation}
\label{eq:intro-equation-dual-non-perturbed}
(\pt - P(\udlu)^*) u = 0,
\end{equation}
where $ P(\udlu)^* $ is the formal adjoint operator of $ P(\udlu) $ with respect to the scalar product $ \Re(\cdot,\cdot)_{\Ltwo} $, such that for all $ f,g \in \Cinf(\Td) $,
\begin{equation*}
\Re(P(\udlu) f,g)_{\Ltwo} = \Re(f,P(\udlu)^* g)_{\Ltwo}.
\end{equation*}
The reason to use $ \Re (\cdot,\cdot)_{\Ltwo} $ instead of $ (\cdot,\cdot)_{\Ltwo} $, and thus consider the real Hilbert space $ (\Ldottwo(\Td),\Re(\cdot,\cdot)_{\Ltwo}) $ is that $ P(\udlu) $ is not $ \C $-linear, but only $ \R $-linear (see the exact expression of $ P(\udlu) $ in Proposition~\ref{prop:paralinearized-pb-of-control}), due to our pursuit of a real valued control, which results in the appearance of the $ \R $-linear operator $ \Re $ in~\eqref{eq:form-Pext}.

The (strong) $ \Ltwo $-observability of~\eqref{eq:intro-equation-dual-non-perturbed} is the following inequality, that for all its solution~$ u \in C([0,T],\Ldottwo(\Td)) $,
\begin{equation}
\label{eq:intro-observability-para}
\|u(0)\|_{\Ltwo}^2 \lesssim \int_0^T \|\B(\udlu)^* u(t)\|_{\Ltwo}^2 \dt,
\end{equation}
where $ \B(\udlu)^* $ is the adjoint of $ \B(\udlu) $ with respect to $ \Re(\cdot,\cdot)_{\Ltwo} $. To prove it, we consider~\eqref{eq:intro-equation-dual-non-perturbed} as a perturbation of a pseudo\-differential equation,
\begin{equation}
\label{eq:intro-Pseudo-Linear-Equation}
(D_t + Q(\udlu)) u = 0,
\end{equation}
where $ D_t = \frac{1}{i} \pt $, and $ Q(\udlu) $ is the pseudo\-differential operator with the same symbol as $ i P(\udlu)^* $, so that it is of principal symbol $ \sim |\xi|^{3/2}  $. Then we derive~\eqref{eq:intro-observability-para} from the following $ \Ltwo $-observability of~\eqref{eq:intro-Pseudo-Linear-Equation} by Duhamel's formula,
\begin{equation}
\label{eq:intro-observability-pseudo}
\|u(0)\|_{\Ltwo}^2 \lesssim \int_0^T  \|\phiw  \Re\, u\|_{\Ltwo}^2 \dt.
\end{equation}
To deal with this problem caused by the lack of $ \C $-linearity of the equation, which occurs especially in the proof of the unique continuation property that deals with the low frequency regime, we consider simul\-taneously~$ u $ and its complex conjugate~$ \bar{u} $. The pair $ \vec{w} = \binom{u}{\bar{u}} $ then satisfies the $ \C $-linear system of pseudo\-differential equation,
\begin{equation}
\label{eq:intro-equation-pseudo-diff-sys}
D_t \vec{w} + \A(\udlu) \vec{w} = 0,
\end{equation}
where $ \A $ is of principal symbol $ \sim \begin{pmatrix} |\xi|^{3/2} & 0 \\ 0 & -|\xi|^{3/2} \end{pmatrix} $.
Then~\eqref{eq:intro-observability-pseudo} is a consequence of the following $ \Ltwo $-observability of~\eqref{eq:intro-equation-pseudo-diff-sys},
\begin{equation}
\label{eq:intro-observability-sys}
\|\vec{w}(0)\|_{\Ltwo}^2 \lesssim \int_0^T  \|\phiw  \vec{e}\cdot\vec{w}\|_{\Ltwo}^2 \dt,
\end{equation}
where $ \vec{e} = \binom{1}{1} $, so that for $ \vec{w} = \binom{w^+}{w^-} $, $ \vec{e} \cdot \vec{w} = w^+ + w^- $.

\begin{remark}
Here we write $ \|\vec{w}(0)\|_{\Ltwo} = \|\vec{w}(0)\|_{\Ltwo \times \Ltwo} $ for simplicity. This will be our convention for the Sobolev norms of the pair~$ \vec{w} $. In this spirit, we will also write $ \H{s} = \H{s}\times\H{s} $ when there is no ambiguity.
\end{remark}

\subsection{Semiclassical Observability}

We adapt Lebeau's approach~\cite{Lebeau-Control-Schrodinger} to prove~\eqref{eq:intro-observability-sys}. This approach performs a dyadic decomposition in frequencies of~$ \vec{w} $, then proves a semi\-classical observability (see below) for each dyadic part by studying the propagation of mass ($ \Ltwo $-norm), and finally recovers the observability of~$ \vec{w} $ using Littlewood-Paley's theory. The idea behind this approach is that, the mass of solutions of a dispersive equation, is not well propagated unless for those localized in frequencies, because group velocities for different frequencies vary, and we are not expected to have a uniform propagation for all frequencies. 
To be more precise, for $ h = 2^{-j} $, with $ j \in \N $, we define
\begin{equation*}
\vec{w}_h = \Pi_h \Delta_h \vec{w},
\end{equation*}
where $ \Delta_h $ is a semi\-classical pseudo\-differential operator, and $ \Pi_h $ is defined by a semi\-classical functional calculus (see Section~\ref{sec:reduction-to-semiclassical-system} for the exact expression; to set $ \vec{w}_h $ in this way is purely out of a technical concern), so that the frequency of $ \vec{w}_h $ is localized in $ A h^{-1} \le |\xi| \le B h^{-1}$ for some $ 0 < A < B $. 
The semi\-classical observability states that, for $ h > 0 $ sufficiently small,
\begin{equation}
\label{eq:intro-observability-semiclassical}
\|\vec{w}_h(0)\|_{\Ltwo}^2 \lesssim h^{-1/2} \int_0^{h^{1/2}T}  \|\phiw \vec{e} \cdot \vec{w}_h\|_{\Ltwo}^2 \dt + \mathrm{remainder\ terms}.
\end{equation}
The reason why~\eqref{eq:intro-observability-semiclassical} holds on a semi\-classical time interval of length $ \sim h^{1/2} $ can be intuitively explained as follows. The dual equation~\eqref{eq:equation-dual-non-perturbed} with~$ \udlu = 0 $, $ g = 0 $ and $ \b = \infty $ (recall that~$ b $ is the depth of the domain of water) is 
\begin{equation}
\label{eq:intro-dual-udlu=0}
(D_t + |D_x|^{3/2}) u = 0,
\end{equation}
with its dispersion relation and group velocity being
\begin{equation}
\label{eq:intro-dispersion-relation}
w(k) = |k|^{3/2}, \quad v_g(k) = \frac{\partial w}{\partial k} = \frac{3}{2} \cdot \frac{k}{|k|^{1/2}}.
\end{equation} 
Due to the variation of frequencies on the interval $ [Ah^{-1},Bh^{-1}] $ of size $ \sim h^{-1} $, the variation of group velocity is then $ \sim h^{-1/2} $, hence the mass is only ``well propagated'' on a semi\-classical time interval of size $ \sim h^{1/2} $. That is why Lebeau introduced in~\cite{Lebeau-Control-Schrodinger} a semi\-classical time variable $ s $ by a scaling $ s = h^{-1/2} t $ (see also Staffilani-Tataru~\cite{Staffilani-Tataru} and Burq-Gérard-Tzvetkov~\cite{BGT-Strichartz} where this same scaling is used to prove the Strichartz estimates of Schr\"{o}dinger equations with variable coefficients), so that the semi\-classical time interval is of size $ \sim 1 $ in $ s $ variable.

As for a technical explanation, equation~\eqref{eq:intro-dual-udlu=0} writes in variables $ (s,x) $ as a zero order semi\-classical differential equation,
\begin{equation}
\label{eq:intro-equation-linearized-around-zero}
(hD_s + |hD_x|^{3/2}) u = 0.
\end{equation}
The vague word ``well propagated'' above can be quantitatively characterized by the semi\-classical defect measures of solutions to~\eqref{eq:intro-equation-linearized-around-zero}. For semi\-classical defect measures, we refer to Gérard~\cite{Gerard-Semiclassical-Measure}, Gérard-Leichtnam~\cite{G-L}, Lions--Paul~\cite{Lions-Wigner}, see also the survey by Burq~\cite{Burq-Bourbaki}. Then the proof of~\eqref{eq:intro-observability-semiclassical} argues by contradiction and studies the propagation of the semi\-classical defect measures, where the geometric control condition is required. This argument dates back to Lebeau~\cite{Lebeau-damped-wave}.

\subsection{Weak $ \Ltwo $-Observability}

By an energy estimate, we can show that~\eqref{eq:intro-observability-semiclassical} remains valid after replacing the interval of integration $ [0,h^{1/2}T] $ by $ I_k = [h^{1/2}kT,h^{1/2}(k+1)T] $ for $ k = 0,1,\ldots,h^{-1/2}-1 $ (we may of course assume that $ h = 2^{2j} $), with a slight change of the remainder term. Once it is proven, we patch $ \{I_k\}_{0 \le k < h^{-1/2}} $ up and use Littlewood-Paley's theory to obtain a weak observability, that for any $ N > 0 $,
\begin{equation}
\label{eq:intro-observability-weak}
\|\vec{w}(0)\|_{\Ltwo}^2 \lesssim \int_0^T  \|\phiw\vec{e}\cdot\vec{w}\|_{\Ltwo}^2 \dt + \|\vec{w}(0)\|_{\H{-N}}^2.
\end{equation}
The remainder term $ \|\vec{w}(0)\|_{\H{-N}}^2 $ comes from the fact that~\eqref{eq:intro-observability-semiclassical} holds only for small~$ h $, that is, for high frequencies. 

\subsection{Unique Continuation and Strong $ \Ltwo $-Observability}

We use the uniqueness-compactness argument of Bardos--Lebeau--Rauch~\cite{BLR} to eliminate the remainder term and obtain the strong $ \Ltwo $-observability. This argument derives the strong observability from the weak observability and the unique continuation property of the linear equation, which is, in our case equation~\eqref{eq:intro-equation-pseudo-diff-sys}. However, this equation has a fractional (hence non-analytical with respect to~$ \xi $) symbol, so Carleman's estimate does not seem to apply directly. We then use a perturbative argument (by contradiction), and reduce the problem to proving only the unique continuation property of the equation linearized around the equilibrium state, that is $ \eta = 0 $, which is a constant coefficient equation,
\begin{equation}
\label{eq:intro-equation-w-udlu=0}
D_t \vec{w} + \A(0) \vec{w} = 0.
\end{equation}
To be more precise, suppose that the strong observability does not hold for arbitrarily small~$ \varepsilon_0 $, then we can find a sequence of solutions~$ \vec{w}_n $ to~\eqref{eq:intro-equation-pseudo-diff-sys} with $ \udlu = \udlu_n \in \Cls^{1,s}(T,\varepsilon_n) $, $ \varepsilon_n \to 0 $ as $ n \to \infty $, such that
\begin{equation*}
\|\vec{w}_n(0)\|_{\Ltwo} = 1, \quad 
\int_0^T \|\phiw \vec{e}\cdot\vec{w}_n\|_{\Ltwo}^2 \dt = o(1).
\end{equation*}
The solutions $ \vec{w}_n $ converge in distributional sense to a solution $ \vec{w} $ to~\eqref{eq:intro-equation-w-udlu=0}, such that
\begin{equation}
\label{eq:intro-condition-unique-continuation}
\phiw \vec{e}\cdot\vec{w}|_{0 \le t \le T} = 0.
\end{equation}
Moreover, by the weak observability~\eqref{eq:intro-observability-weak} and Rellich--Kondrachov's compact injection theorem, $ \|\vec{w}(0)\|_{\H{-N}} > 0 $, hence $ \vec{w} \not \equiv 0 $. We will obtain a contradiction by the unique continuation property of~\eqref{eq:intro-equation-w-udlu=0}, that is, such a non-zero solution to~\eqref{eq:intro-equation-w-udlu=0} never exists.
To prove this, we consider the $ \C $-linear vector space~$ \mathscr{N} $ of initial data $ \vec{w}_0 \in \Ldottwo(\Td) $ whose corresponding solution satisfies~\eqref{eq:intro-condition-unique-continuation}. Then the weak observability~\eqref{eq:intro-observability-weak} implies that~$ \mathscr{N} $ is a compact metric space with respect to the $ \Ltwo $-norm, and hence of finite dimension. It is not difficult to see that $ \A(0) $ defines a $ \C $-linear operator on~$ \mathscr{N} $, and thus admits an eigenfunction $ \vec{w}_0 = \binom{w_0^+}{w_0^-} \in \Ldottwo(\Td) $ such that
\begin{equation}
\label{eq:intro-eigenfunction-A(0)}
\A(0) \vec{w}_0 = \lambda \vec{w}_0, \quad \lambda \in \C,
\end{equation}
and that
\begin{equation*}
(w^+_0 + w^-_0)|_{\omega} = \vec{e} \cdot \vec{w}_0|_{\omega} = 0.
\end{equation*}
We obtain from~\eqref{eq:intro-eigenfunction-A(0)} an elliptic equation for $ w^+_0 + w^-_0 $, and show that $ w^+_0 + w^-_0 $ has only a finite Fourier modes. Therefore it can never vanish on a non-empty open set unless $ w^+_0 + w^-_0 \equiv 0 $. Again easily from~\eqref{eq:intro-eigenfunction-A(0)} we also obtain $ w^+_0 - w^-_0 \equiv 0 $. This contradicts the fact that~$ \vec{w} $ is an eigenfunction and can not be identically zero.

\subsection{$ \H{s} $-Controllability}
Now we prove the $ \H{s} $-regularity of $ \Theta = \Theta(\udlu) $, that is, whenever the initial data $ u_0 \in \Hdot{s}(\Td) $, the control $ F = \Theta u_0 \in C([0,T],\Hdot{s}(\Td)) $. The Hilbert uniqueness method implies that, at the $ \Ltwo $ level,
\begin{equation*}
\Theta = -\B^*\Sol \K^{-1}, \quad \K = - \Range \B \B^* \Sol,
\end{equation*}
where the definitions of the solution operator~$ \Sol $ and the range operator~$ \Range $ will be made clear in Section~\ref{sec:HUM}. The main motivation to proving the strong $ \Ltwo $-observability is that, it is a rephrasing of the coercivity of the following $ \R $-bilinear form on $ \Ldottwo(\Td) $,
\begin{equation*}
\varpi(f,g) = \Re(\K f,g)_{\Ltwo}.
\end{equation*}
Therefore the strong $ \Ltwo $-observability implies, by Lax-Milgram's theorem, the invertibility of~$ \K $, and consequently the well-definedness of~$ \Theta $.

Coming back to the $ \H{s} $ level, observing that, by the definition of~$ \K $, it is easily verified that~$ \K $ sends $ \Hdot{s}(\Td) $ to itself. Therefore, to prove the $ \H{s} $-regularity of~$ \Theta $, we only need to show that $ \K|_{\Hdot{s}} : \Hdot{s}(\Td) \to \Hdot{s}(\Td) $ defines an isomorphism. We use again Lax-Milgram's theorem, and consider the following $ \R $-bilinear form on the semi\-classical Sobolev space $ \Hdot{s}_h(\Td) = \Hdot{s}(\Td) \cap \H{s}_h(\Td) $ which inherits the semi\-classical Sobolev norm from $ \H{s}_h(\Td) $,
\begin{equation*}
\varpi^s_h(f,g) = \Re(\Lambda^{s}_h|_{t=0} \K f,\Lambda^{s}_h|_{t=0} g)_{\Ltwo}.
\end{equation*}
where $ \Lambda^s_h = 1 + h^s \T{(\gamma^{(3/2)})^{2s/3}} $. In order to conclude, the key point is the following commutator estimate, that for~$ \varepsilon_0 $ and~$ h $ sufficiently small,
\begin{equation}
\label{eq:intro-commutator-estimate}
\big[\K,\Lambda^s_h|_{t=0}\big]\Lambda^{-s}_h|_{t=0} = O(\varepsilon_0 + h)_{\L(\Ldottwo,\Ldottwo)}.
\end{equation}
Indeed, if this is proven, we obtain the coercivity of~$ \varpi^s_h $ on $ \Hdot{s}_h(\Td) $,
\begin{equation*}
\varpi^s_h(f,f) 
= \varpi(\Lambda^{s}_h|_{t=0} f, \Lambda^{s}_h|_{t=0} f) + O(\varepsilon_0+h) \|\Lambda^{s}_h|_{t=0} f\|_{\Ltwo}^2
\gtrsim \|f\|_{\H{s}_h}^2.
\end{equation*}
The trick of introducing the semi\-classical parameter~$ h $ has already been used in~\cite{ABH-Control}, but our proof of the $ \H{s} $-regularity is different.

\section{Reformulation of Problem}
\label{sec:Paradifferentialization-and-Symmetrization}

This section performs the paralinearization of~\eqref{eq:equation-water-wave} as developed in~\cite{AM-Paralinearization,ABZ-Surface-Tension,ABZ-Without-Surface-Tension}, but with two differences: first a more careful treatment to the remainder terms, so that they are not only of sufficient regularity, but also ``super-linear'' (see Remark~2.6 of~\cite{ABH-Control}); and second, a modification concerning the existence of the exterior pressure~$ \Pext $. 

The estimates below are carried out for each fixed time $ t $, so the time variable will be temporarily omitted for simplicity.

\subsection{Paralinearization of Dirichlet-Neumann Operator}

Recall the following notations which are standard in literatures. Let 
\begin{equation*}
B(x) = (\partial_y \phi)(x,\eta(x)), \quad V(x) = (\nabla_x \phi)(x,\eta(x))
\end{equation*}
be the traces of the vertical and horizontal components of the Eulerian vector field to the free surface, respectively. They admit the following expressions which could be seen as linear operators depending on $ \eta $, applying on $ \psi $.
\begin{equation}
\label{eq:def-B-V}
B = B(\eta)\psi 
= \frac{\nabla_x \eta \cdot \nabla_x \psi + G(\eta)\psi}{1+|\nabla_x\eta|^2},  \quad
V = V(\eta)\psi 
= \nabla_x \psi -  B \nabla_x \eta,
\end{equation}
as well as the good unknown of Alinhac,
\begin{equation}
\label{def:good-unknown-of-Alinhac}
\omega = \omega(\eta)\psi = \psi - \T{B} \eta.
\end{equation}

\begin{remark}
When there is no ambiguity, we simply write $ B $, $ V $ and $ \omega $ for short. The notation $ B(\eta) $ will be used later in Lemma~\ref{lem:equation-good-unknown} to denote the operator $ B(\eta) = \frac{\nabla_x\eta\cdot\nabla_x + G(\eta)}{1+|\nabla_x\eta|^2} $.
\end{remark}

\begin{lemma}
\label{lemma:estimate-B,V}
Suppose $ s > 1/2 + d/2 $ and $ 1 + d/2 < \sigma \le s + 1/2 $, then
\begin{equation*}
\|(B,V)\|_{\H{\sigma-1}\times\H{\sigma-1}} \le C(\|\eta\|_{\H{s+1/2}}) \|\nabla_x \psi\|_{\H{\sigma-1}}.
\end{equation*}
\end{lemma}
\begin{proof}
By Theorem 3.12 of~\cite{ABZ-Without-Surface-Tension},
\begin{equation*}
\|G(\eta)\psi\|_{\H{\sigma-1}} \le C(\|\eta\|_{\H{s+1/2}}) \|\nabla_x\psi\|_{\H{\sigma-1}}.
\end{equation*}
We conclude by $ \|uv\|_{\H{\sigma-1}} \lesssim \|u\|_{\H{\sigma-1}} \|v\|_{\H{\sigma-1}} $ and~\eqref{eq:paradiff-estimate-composition-by-function}.
\end{proof}

The following proposition follows by a careful study of Proposition~3.14 of~\cite{ABZ-Surface-Tension}.

\begin{proposition}[Alazard--Burq--Zuily]
\label{prop:AM-Paradiff-of-DN}
Suppose $ s > 3 + d/2 $ and $ (\eta,\psi) \in \H{s+1/2}(\Td) \times \H{s}(\Td) $. Let the symbol
\begin{equation*}
\lambda = \lambda^{(1)} + \lambda^{(0)},
\end{equation*}
be defined by
\begin{align*}
\lambda^{(1)}(x,\xi) & = \sqrt{ (1+|\nabla_x \eta|^2) |\xi|^2 - (\nabla_x \eta \cdot \xi)^2 }, \\
\lambda^{(0)}(x,\xi) & = \frac{\SurfArea{\eta}}{2 \lambda^{(1)}} \big\{ \nabla_x \cdot \big( \alpha^{(1)} \nabla_x \eta \big) + i \pxi \lambda^{(1)} \cdot \nabla_x \alpha^{(1)}  \big\},
\end{align*}
where $	\alpha^{(1)}(x,\xi) = \frac{\lambda^{(1)} + i \nabla_x \eta \cdot \xi}{\SurfArea{\eta}} $. Then
\begin{equation*}
G(\eta) \psi = \T{\lambda}{\omega} - \nabla_x \cdot \T{V}{\eta} + f(\eta,\psi),
\end{equation*}
where for any $ 1/2 < \delta < 1 $,
\begin{equation}
\label{eq:estimate-f(eta,psi)}
\| f(\eta,\psi) \|_{\H{s+\delta}}
\le C(\|\eta\|_{\H{s+1/2}}) \|\nabla_x \psi\|_{\H{s-1}}.
\end{equation}
\end{proposition}

\begin{proof}
Now that we have a larger regularity, $ s>3+d/2 $, than that that in~\cite{ABZ-Surface-Tension}, which is $ s>2+d/2 $, we have no need to deal with the paraproduct $ \T{a} $ for some $ a \in \H{d/2-\varepsilon} $, which is of positive order. The proof in~\cite{ABZ-Surface-Tension} then implies that
\begin{equation*}
G(\eta) \psi = \T{\lambda} \omega - \T{V} \cdot \nabla_x\eta - \T{\nabla_x \cdot V} \eta + f(\eta,\psi),
\end{equation*}
where for any $ 1/2 < \delta < 1 $, $ \| f(\eta,\psi) \|_{\H{s+\delta}}
\le C(\|\eta\|_{\H{s+1/2}}) \|\nabla_x \psi\|_{\H{s-1}} $. Instead of considering $ - \T{\nabla_x \cdot V} \eta + f(\eta,\psi) $ together as the remainder term of regularity $ \H{s+1/2} $ (which already suffices in~\cite{ABZ-Surface-Tension} for the Cauchy problem), we observe here that
\begin{equation*}
\T{V} \cdot \nabla_x\eta + \T{\nabla_x \cdot V} \eta = \nabla_x \cdot T_V \eta,
\end{equation*}
and conclude.
\end{proof}

\begin{remark}
\label{rmk:DN-remainder}
By the definition of $ \lambda $, $ V $ and $ \omega $, the remainder $ f(\eta,\psi) $ is of the form,
\begin{equation*}
f(\eta,\psi) = R_G(\eta)\psi + M_{\b} \psi,
\end{equation*}
where $ R_G(\eta) $ is a linear operator depending on $ \eta $, and following~\eqref{eq:estimate-f(eta,psi)},
\begin{equation}
\label{eq:estimate-R-eta-AM}
\|R_G(\eta)\psi\|_{\H{s+\delta}}
\le C(\|\eta\|_{\H{s+1/2}}) \|\nabla_x \psi\|_{\H{s-1}},
\end{equation}
while $ M_{\b} = G(0) - |D_x| = m_{\b}(D_x) $ is a smoothing Fourier multiplier, with symbol 
\begin{equation*}
m_{\b}(\xi) = |\xi|(\tanh(\b|\xi|)-1).
\end{equation*}
Indeed, $ G(0) = |D_x| \tanh(\b D_x) $ is a Fourier multiplier which could be calculated directly.
\end{remark}

We aim to prove the following estimate.
\begin{proposition}
For $ s > 3 + d/2 $ and some $ \vartheta > 0 $,
\begin{equation*}
\|R_G(\eta)\psi\|_{\H{s+1/2}} \le C(\|\eta\|_{\H{s+1/2}}) \|\eta\|_{\H{s+1/2}}^\vartheta \|\nabla_x \psi\|_{\H{s-1}}.
\end{equation*}
\end{proposition}
\begin{proof}
This follows by interpolating between~\eqref{eq:estimate-R-eta-AM}, and the following Lemma~\ref{lem:estimate-R-eta-quadratic}.
\end{proof}

\begin{lemma}
\label{lem:estimate-R-eta-quadratic}
For $ s > 3 + d/2 $,
$\|R_G(\eta)\psi\|_{\H{s-2}} \le C(\|\eta\|_{\H{s+1/2}})\|\eta\|_{\H{s+1/2}}\|\nabla_x \psi\|_{\H{s-1}}.$
\end{lemma}
\begin{proof}
Write 
\begin{equation*}
R_G(\eta)\psi = (G(\eta)\psi - G(0)\psi) - (\T{\lambda}\omega-|D_x|\psi) + \nabla_x \cdot \T{V}\eta, 
\end{equation*}
and estimate the three terms separately. By Lemma~\ref{lemma:estimate-B,V} and Theorem~\ref{thm:Operator-Norm-Estimate-Paradiff},
\begin{equation*}
\|\nabla_x \cdot \T{V} \eta\|_{\H{s-1/2}} 
\lesssim \|V\|_{\Linf} \|\eta\|_{\H{s+1/2}} \le C(\|\eta\|_{\H{s+1/2}}) \|\eta\|_{\H{s+1/2}} \|\nabla_x\psi\|_{\H{s-1}}.
\end{equation*}
For the middle term, write $ \T{\lambda}{\omega} - |D_x| \psi = \T{\lambda - |\xi|}{\psi}	- \T{\lambda}{\T{B}{\eta}}$, and by Theorem~\ref{thm:Operator-Norm-Estimate-Paradiff},
\begin{align*}
\| \T{\lambda - |\xi|}{\psi} \|_{\H{s-1}}
& \lesssim \M{1}{0}{d/2+1}{\lambda-|\xi|} \|\nabla_x \psi\|_{\H{s-1}} 
\le C(\|\eta\|_{\H{s+1/2}}) \|\eta\|_{\H{s+1/2}} \|\nabla_x \psi\|_{\H{s-1}}, \\
\| \T{\lambda}{\T{B}{\eta}} \|_{\H{s-1/2}}
& \lesssim \M{1}{0}{d/2+1}{\lambda} \|B\|_{\Linf} \|\eta\|_{\H{s+1/2}}
\le C(\|\eta\|_{\H{s+1/2}}) \|\eta\|_{\H{s+1/2}} \|\nabla_x \psi\|_{\H{s-1}}.
\end{align*}
The first term is estimated by the following lemma.
\end{proof}

\begin{lemma}
$\| G(\eta) \psi - G(0) \psi \|_{\H{s-2}}
\le C(\|\eta\|_{\H{s+1/2}}) \|\eta\|_{\H{s+1/2}} \|\nabla_x \psi\|_{\H{s-1}}$
\end{lemma}
\begin{proof}
Recall the shape derivation of the Dirichlet-Neumann operator by~\cite{Lannes-Well-posedness}, see also~\cite{ABZ-Surface-Tension}.
\begin{proposition}
Let $ s > 2 + d/2$, and $ \eta \in \H{s+1/2}(\Td) $, $ \psi \in \H{\sigma}(\Td) $ for $ 1 \le \sigma \le s $. Then there exists a neighborhood of $ \eta $ in $ \H{s+1/2}(\Td) $, in which the mapping $\tilde{\eta} \mapsto G(\tilde{\eta}) \psi$ is differentiable. Moreover, for $ h \in \H{s+1/2}(\Td) $,
\begin{equation}
\label{eq:formula-shapde-derivative}
\drv_\eta G(\eta) \psi \cdot h 
\bydef \lim_{\varepsilon \to 0} \tfrac{1}{\varepsilon} \{ G(\eta + \varepsilon h) \psi - G(\eta) \psi \}
= -G(\eta) (Bh) - \nabla_x \cdot (Vh).
\end{equation}
\end{proposition}
	
Then by Newton-Leibniz formula,
\begin{equation*}
G(\eta) \psi - G(0) \psi
= \int_0^1 \drv G(t \eta) \psi \cdot \eta \dt 
 = - \int_0^1 \Big( G(t \eta) \big( B_t \eta \big) + \nabla_x \cdot \big( V_t \eta \big) \Big) \dt,
\end{equation*}
where $ B_t = B(t \eta)\psi,\ V_t = V(t \eta)\psi $. Therefore,
\begin{align*}
\| G(\eta) \psi - G(0)  \psi \|_{\H{s-2}}
& \le \int_0^1 \Big( C(\|t \eta\|_{\H{s+1/2}})\|B_t\|_{\H{s-1}} + \|V_t\|_{\H{s-1}} \Big) \|\eta\|_{\H{s-1}} \dt \\
& \le C(\| \eta \|_{\H{s+1/2}}) \|\eta \|_{\H{s+1/2}} \| \nabla_x \psi \|_{\H{s-1}}.
\end{align*}
\end{proof}

\subsection{Paralinearization of Surface Tension}

\begin{proposition}
\label{prop:Paralinearization-Surface-Tension}
Let $ \eta \in \H{s+1/2}(\Td) $ with $ s > 3/2 + d/2 $, then
\begin{equation}
H(\eta) = - \T{\ell}\eta + R_H(\eta)\eta,
\end{equation}
where the symbol 
\begin{equation*}
\ell = \ell^{(2)} + \ell^{(1)}
\end{equation*}
is explicitly defined by
\begin{equation*}
\begin{split}
\ell^{(2)} 
=\frac{1}{\sqrt{\SurfArea{\eta}}}\Big( |\xi|^2 - \frac{(\nabla_x\eta\cdot\xi)^2}{\SurfArea{\eta}} \Big), 
\qquad 
\ell^{(1)}
= \frac{1}{2} \pxi \cdot D_x \ell^{(2)},
\end{split}
\end{equation*}
and $ R_H(\eta) $ is a linear operator satisfying the estimate,
\begin{equation*}
\|R_H(\eta)\|_{\L(\H{s+1/2},\H{2s-2-d/2})} 
\le C(\|\eta\|_{\H{s+1/2}})\|\eta\|_{\H{s+1/2}}.
\end{equation*}
\end{proposition}
\begin{proof}
Letting $ L(v) = \frac{1}{\sqrt{1+|v|^2}} $ for $ v \in \Rd $, then $ L(0)-1 = 0 $, and  by~\eqref{eq:paralinearization},
\begin{align*}
\frac{1}{\sqrt{1+|\nabla_x \eta|^2}} = A(\nabla_x \eta)
& = 1 + (L(\nabla_x \eta) - 1) \\
& = 1 + \T{(\nabla L)(\nabla_x \eta)} \cdot \nabla_x\eta + R_{L-1}(\nabla_x\eta),
\end{align*}
where $ (\nabla L)(\nabla_x \eta) = - \frac{\nabla_x \eta}{\sqrt{\SurfArea{\eta}}^3} $, and the remainder satisfies the estimate
\begin{equation*}
\|R_{L-1}(\nabla_x\eta)\|_{\H{2s-1-d/2}}
\le C(\|\eta\|_{\H{s+1/2}}) \|\eta\|_{\H{s+1/2}}.
\end{equation*}
Therefore, using the fact that $ \T{\nabla_x \eta} 1 = 0 $, we have
\begin{align*}
\frac{\nabla_x \eta}{\sqrt{\SurfArea{\eta}}}
& = \nabla_x \eta \; L(\nabla_x \eta) \\
& = \T{\nabla_x \eta}{L(\nabla_x \eta)} + \T{L(\nabla_x \eta)}{\nabla_x \eta} + R\big(\nabla_x \eta,L(\nabla_x \eta)\big); \\
& = \T{\nabla_x \eta} \T{(\nabla L)(\nabla_x \eta)} \cdot \nabla_x \eta + \T{L(\nabla_x \eta)} \nabla_x \eta \\
& \qquad \qquad + R\big(\nabla_x \eta,L(\nabla_x \eta)\big) + \T{\nabla_x \eta}R_{L-1}(\nabla_x\eta) \\
& = \T{M} \nabla_x \eta + \tilde{R}(\eta) \eta,
\end{align*}
with the matrix valued symbol
\begin{equation*}
M 
= L(\nabla_x \eta) + \nabla_x \eta \otimes (\nabla L)(\nabla_x \eta)
= \frac{1}{\sqrt{\SurfArea{\eta}}} I_d - \frac{\nabla_x \eta \otimes \nabla_x \eta}{\sqrt{\SurfArea{\eta}}^3}.
\end{equation*}
As for the remainder, $ \tilde{R}(\eta) $ is a linear operator depending on $ \eta $ defined by
\begin{align*}
\tilde{R}(\eta) v
& = R\big(\nabla_x v,L(\nabla_x \eta)\big) + \T{\nabla_x v}R_{L-1}(\nabla_x\eta) \\
& \qquad \qquad + \big(\T{\nabla_x \eta} \T{(\nabla L)(\nabla_x \eta)} \cdot \nabla_x - \T{\nabla_x \eta \otimes (\nabla L)(\nabla_x \eta)} \nabla_x\big) v,
\end{align*}
satisfying by~\eqref{eq:composition-for-paraproduct} the estimate
\begin{equation*}
\|\tilde{R}(\eta) v\|_{\H{2s-1-d/2}} \le C(\|\eta\|_{\H{s+1/2}}) \|\eta\|_{\H{s+1/2}} \|v\|_{\H{s+1/2}}.
\end{equation*}
Indeed, the only problem is for the first term, for which we write
\begin{equation*}
R(\nabla_x v,L(\nabla_x\eta)) 
= R(\nabla_x v,1) + R(\nabla_x v,L(\nabla_x\eta)-1),
\end{equation*}
where
\begin{equation*}
R(\nabla_x v,1) = \nabla_x v - \T{\nabla_x v} 1 - \T{1}\nabla_x v = \nabla_x v - 0 - \nabla_x v = 0.
\end{equation*}
Consequently,
\begin{equation*}
H(\eta)
= \MeanCurv{\eta}
= \T{-M\xi\cdot\xi+i \nabla_x \cdot M\xi} \eta + R_H(\eta)\eta,
\end{equation*}
with $ \ell^{(2)} = M \xi\cdot\xi, \ell^{(1)} = \frac{1}{i}\nabla_x \cdot M\xi $, and $ R_H(\eta) $ defined by $ R_H(\eta) v = \nabla_x \cdot (\tilde{R}(\eta) v) $ which satisfies the desired estimate.
\end{proof}

\subsection{Equation for the Good Unknown}
\begin{proposition}
\label{prop:Paralinearization-Velocity}
Let $ \eta \in \H{s+1/2}(\Td) $ and $ \psi \in \H{s}(\Td) $ with $ s > 1 + d/2 $, then
\begin{equation}
\label{eq:paralinearization-velocity}
\begin{split}
\frac{1}{2} & \frac{(\VertVelo)^2}{\SurfArea{\eta}} - \frac{1}{2} |
\nabla_x \psi|^2 \\
& \qquad \qquad \qquad
= \T{B} G(\eta) \psi + \T{B}\T{V}\cdot \nabla_x\eta - \T{V}\cdot\nabla_x\psi + R_S(\eta,\psi) \psi,
\end{split}
\end{equation}
where $ R_S(\eta,\psi) $ is a linear operator satisfying the estimate
\begin{equation*}
\|R_S(\eta,\psi)\|_{\L(\H{s},\H{2s-2-d/2})} \le C(\|\eta\|_{\H{s+1/2}}) \|\nabla_x \psi\|_{\H{s-1}}.
\end{equation*}
\end{proposition}

\begin{proof}
By the definition of~$ B $, $ V $ and $ G(\eta) $, we easily verify the identity.
\begin{equation*}
G(\eta)\psi = B - \nabla_x \eta \cdot V.
\end{equation*}
Plug this into the left hand side of~\eqref{eq:paralinearization-velocity},
\begin{align*}
\frac{1}{2} & \frac{(\VertVelo)^2}{\SurfArea{\eta}} - \frac{1}{2} |
\nabla_x \psi|^2 \\
& \qquad \qquad \qquad
= \frac{1}{2} B (\nabla_x \eta \cdot \nabla_x \psi + G(\eta) \psi) - \nabla_x \psi \cdot (V + B\nabla_x \eta) \\
& \qquad \qquad \qquad
= B G(\eta) \psi - \frac{1}{2} (B G(\eta)\psi + \nabla_x \psi \cdot V) \\
& \qquad \qquad \qquad
= B G(\eta) \psi - \frac{1}{2} \big( B (B-V\cdot \nabla_x \eta) + (V+B\nabla_x \eta) \cdot V \big) \\
& \qquad \qquad \qquad
= B G(\eta) \psi - \frac{1}{2} B^2 - \frac{1}{2} |V|^2.
\end{align*}
Then by~\eqref{eq:composition-for-paraproduct},
\begin{align*}
B G(\eta) \psi - & \demi B^2 - \demi |V|^2 \\
& = \T{B} G(\eta) \psi + \T{G(\eta)\psi} B - \T{B}B - \T{V}\cdot V + R_1(\eta,\psi)\psi \\
& = \T{B} G(\eta) \psi + \T{G(\eta)\psi -B} B - \T{V}\cdot (\nabla_x\psi - B\nabla_x\eta) + R_1(\eta,\psi)\psi \\
& = \T{B} G(\eta) \psi - \T{V\cdot\nabla_x\eta} B - \T{V}\cdot\nabla_x\psi + \T{V}\cdot (B\nabla_x\eta) + R_1(\eta,\psi)\psi \\
& = \T{B} G(\eta) \psi + \T{B}\T{V}\cdot \nabla_x\eta - \T{V}\cdot\nabla_x\psi + R_1(\eta,\psi)\psi + R_2(\eta,\psi)\psi,
\end{align*}
with the two operators defined by
\begin{align*}
R_1(\eta,\psi) v^1
& = R(B,G(\eta)v^1) - \demi R(B,B(\eta) v^1) - \demi R(V,V(\eta) v^1) \\
R_2(\eta,\psi) v^2
& = \T{V}\cdot R(B(\eta) v^2,\nabla_x\eta) + (\T{V}\cdot\T{\nabla_x\eta}-\T{V\cdot\nabla_x\eta})B(\eta) v^2 + [\T{V},\T{B(\eta)v^2}]\nabla_x\eta.
\end{align*}
Hence $ R_S(\eta,\psi) = R_1(\eta,\psi) + R_2(\eta,\psi) $ satisfies the desired estimate by~\eqref{eq:composition-for-paraproduct} and~\eqref{eq:remainder-of-paraproduct}.
\end{proof}

\begin{lemma}
\label{lem:equation-good-unknown}
With $ \Pext $ of the form~\eqref{eq:form-Pext}, the good unknown~$ \omega $ satisfies the equation,
\begin{equation*}
\pt\omega + \nabla_x \cdot \T{V} \omega + \T{\ell}\eta + g\eta + R^1_\omega(\psi,\eta)\psi + R^2_\omega(\psi,\eta)\eta = \chi_T^{} \big( \phiw \Re F - \T{B(\eta) (\phiw \Re F)} \eta \big),
\end{equation*}
where $ R^i(\psi,\eta) $ $ (i=1,2) $ are linear operators such that, for $ s > 2 + d/2 $,
\begin{align*}
\|R^1_\omega(\eta,\psi)\|_{\L(\H{s},\H{s})}
& \le C(\|(\eta,\nabla_x \psi)\|_{\H{s+1/2}\times\H{s-1}})\|\nabla_x \psi\|_{\H{s-1}}, \\
\|R^2_\omega(\eta,\psi)\|_{\L(\H{s+1/2},\H{s})}
& \le C(\|(\eta,\nabla_x \psi)\|_{\H{s+1/2}\times\H{s-1}}) (\|\nabla_x \psi\|_{\H{s-1}} + \|\eta\|_{\H{s+1/2}}).
\end{align*}
\end{lemma}
\begin{proof}
Combining the paralinearization results above, $ \omega $ satisfies the equation,
\begin{equation*}
\pt\omega + \nabla_x \cdot \T{V} \omega + \T{\ell}\eta + g\eta + \tilde{R} = \chi_T^{} \phiw\Re F,
\end{equation*}
with the remainder $ \tilde{R} $ being,
\begin{equation*}
\tilde{R} = \T{\pt B} \eta + (\nabla_x\T{V}\T{B} - \T{B}\T{V}\cdot\nabla_x) \eta + R_H(\eta)\eta + R_S(\eta,\psi)\psi + (\T{V}\cdot \nabla_x - \nabla_x \cdot \T{V}) \psi,
\end{equation*}
where $ R_x(\eta,\psi) $ is defined in Proposition~\ref{prop:Paralinearization-Velocity}. The term $ \T{\pt B} \eta $ should be treated with attention, where $ \pt \psi $ is involved. 
\begin{align*}
\pt B
& = \pt (B(\eta) \psi) 
= B(\eta) \pt\psi + [\pt,B(\eta)] \psi \\
& = \chi_T^{} B(\eta) (\phiw \Re F) + B(\eta)(\pt \psi - \chi_T^{} \phiw \Re F) + [\pt,B(\eta)] \psi,
\end{align*}
where, using the equation,
\begin{align*}
\pt \psi - \chi_T^{} \phiw \Re F 
& = -g\eta + H(\eta) + \frac{1}{2}\frac{(\nabla_x \eta \cdot \nabla_x  \psi + G(\eta)\psi)^2}{1+|\nabla_x \eta|^2} - \frac{1}{2}|\nabla_x \psi|^2, \\
[\pt,B(\eta)]\psi
& = \pt\Big( \frac{1}{1+|\nabla_x\eta|^2} \Big) (\nabla_x\eta\cdot\nabla_x \psi + G(\eta) \psi) \\ 
& \qquad \qquad + \frac{1}{1+|\nabla_x\eta|^2} \big(\nabla_x\pt\eta\cdot\nabla_x \psi + \big(\pt G(\eta)\big) \psi\big).
\end{align*}
Therefore,
\begin{equation*}
\|\pt \psi - \chi^{}_T\phiw \Re F \|_{\H{s-3/2}}
\le C(\|\eta,\nabla_x \psi\|_{\H{s+1/2}\times\H{s-1}}) (\|\nabla_x\psi\|_{\H{s}} + \|\eta\|_{\H{s+1/2}}).
\end{equation*}
And by~\eqref{eq:formula-shapde-derivative},
\begin{equation*}
(\pt G(\eta)) \psi = -G(\eta)(B\pt\eta) - \nabla_x \cdot (V\pt\eta),
\end{equation*}
from which, replacing $ \pt \eta $ with $ G(\eta)\psi $ by the equation,
\begin{equation*}
\|[\pt,B(\eta)]\psi\|_{\H{s-5/2}} \le C(\|(\eta,\nabla_x \psi)\|_{\H{s+1/2}\times\H{s-1}}) (\|\nabla_x\psi\|_{\H{s}} + \|\eta\|_{\H{s+1/2}}).
\end{equation*}
As for an exact formula, it suffices to set
\begin{equation*}
\begin{split}
R_\omega^1(\eta,\psi) v^1
&  = R_S(\eta,\psi) v^1 + (\T{V}\cdot \nabla_x - \nabla_x \cdot \T{V}) v^1\\
R_\omega^2(\eta,\psi) v^2
&  = \T{A(\eta,\psi)} v^2 + (\nabla_x\T{V}\T{B} - \T{B}\T{V}\cdot\nabla_x) v^2 + R_H(\eta)v^2,
\end{split}
\end{equation*}
where $ R_H(\eta) $ is defined in Proposition~\ref{prop:Paralinearization-Surface-Tension}, and by expanding $ \pt B - \chi_T^{} B(\eta) (\phiw \Re F) $ thoroughly using the identities above,
\begin{align*}
A(\eta,\psi) & = -g\eta + H(\eta) + \frac{1}{2}\frac{(\nabla_x \eta \cdot \nabla_x  \psi + G(\eta)\psi)^2}{1+|\nabla_x \eta|^2} - \frac{1}{2}|\nabla_x \psi|^2 \\
& \qquad  - \frac{2\nabla_x\eta \cdot \nabla_x G(\eta) \psi}{(1+|\nabla_x \eta|^2)^2} (\nabla_x\eta \cdot \nabla_x\psi + G(\eta)\psi) \\
& \qquad  \qquad \frac{1}{1+|\nabla_x\eta|^2} \big( \nabla_x G(\eta) \psi \cdot \nabla_x \psi - G(\eta) (B G(\eta)\psi) - \nabla_x \cdot (V G(\eta) \psi) \big).
\end{align*}
\end{proof}

\subsection{Symmetrization}

Now by Proposition~\ref{prop:AM-Paradiff-of-DN} and Lemma~\ref{lem:equation-good-unknown}, for $ s > 3 + d/2 $, the system of water waves can be reformulated as
\begin{align*}
& \left\{
\begin{aligned}
& \pt \eta - \T{\lambda} \omega + \nabla_x \cdot \T{V} \eta - M_{\b}\psi - R_{G}(\eta)\psi  = 0,\\
& \pt \omega + \T{\ell}\eta + \nabla_x \cdot \T{V} \omega + g\eta + R^1_\omega(\psi,\eta)\psi + R^2_\omega(\psi,\eta)\eta 
\end{aligned}
\right.
\\
& \qquad \qquad \qquad \qquad \qquad \qquad \qquad \qquad 
= \chi_T^{} \big( \phiw \Re F - \T{B(\eta) (\phiw \Re F)} \eta \big).
\end{align*}
Following~\cite{ABZ-Surface-Tension}, let symbols which depend solely on $ \eta $,
\begin{equation*}
q  = q^{(0)},
\quad
p  = p^{(1/2)} + p^{(-1/2)}, 
\quad 
\gamma  = \gamma^{(3/2)} + \gamma^{(1/2)},
\end{equation*}
be defined by
\begin{align*}
q^{(0)} & = (\SurfArea{\eta})^{1/4}, \\
\gamma^{(3/2)} & = \sqrt{\ell^{(2)} \lambda^{(1)}}, \\
\gamma^{(1/2)} & = \sqrt{\frac{\ell^{(2)}}{\lambda^{(1)}}} \frac{\Re \lambda^{(0)}}{2} - \demi \pxi \cdot D_x \sqrt{\ell^{(2)}\lambda^{(1)}}, \\
p^{(1/2)} & = (\SurfArea{\eta})^{-1/2} \sqrt{\lambda^{(1)}}, \\
p^{(-1/2)} & = \frac{1}{\gamma^{(3/2)}}( q^{(0)} \ell^{(1)} - \gamma^{(1/2)} p^{(1/2)} + i\pxi \gamma^{(3/2)} \cdot \px p^{(1/2)}).
\end{align*}
Then $ \T{p}\eta $ and $ \T{q}\omega $ satisfies the equation
\begin{align}
\label{eq:equation-pre-symmetrization}
& \left\{
\begin{aligned}
& \pt (\T{p}\eta) - \T{\gamma} \T{q} \omega + \nabla_x \cdot \T{V} \T{p} \eta - \T{p}M_{\b} \psi + R_1(\eta) \psi + R_2(\eta)\eta = 0, \\
& \pt (\T{q} \omega) + \T{\gamma} \T{p} \eta + \nabla_x \cdot \T{V} \T{q} \omega + g \T{q} \eta + R_3(\eta,\psi)\psi + R_4(\eta,\psi) \eta 
\end{aligned}
\right.
\\ \nonumber
& \qquad \qquad \qquad \qquad \qquad \qquad \qquad \qquad  
= \chi_T^{} \big( \T{q} \phiw \Re F - \T{q} \T{B(\eta) \phiw \Re F} \eta \big),
\end{align}
where the remainders are defined by
\begin{align*}
R_1(\eta) v^1
& = - (\T{p}\T{\lambda} - \T{\gamma}\T{q}) \omega(\eta) v^1 - \T{p} R_G(\eta) v^1  - [\nabla_x \cdot \T{V(\eta)v^1},\T{p}] \eta,  \\
R_2(\eta) v^2
& = - \T{\pt p} v^2,  \\
R_3(\eta,\psi) v^3
& = - \T{\pt q}\omega(\eta) v^3 - [\nabla_x \cdot \T{V},\T{q}] \omega(\eta) v^3 + \T{q} R^1_\omega(\eta,\psi) v^3, \\
R_4(\eta,\psi) v^4
& =  - (\T{\gamma}\T{p} - \T{q}\T{\ell})v^4 + \T{q} R^2_\omega(\eta,\psi) v^4,
\end{align*}
and satisfy the estimates
\begin{align*}
\|R_1(\eta)\|_{\L(\H{s},\H{s})} 
& \le C(\|\eta\|_{\H{s+1/2}}) \|\eta\|_{\H{s+1/2}}^\vartheta, \quad \vartheta > 0, \\
\|R_2(\eta)\|_{\L(\H{s+1/2},\H{s})} 
& \le C(\|\eta\|_{\H{s+1/2}}) \|\eta\|_{\H{s+1/2}}, \\
\|R_3(\eta,\psi)\|_{\L(\H{s},\H{s})} 
& \le C(\|(\eta,\nabla_x \psi)\|_{\H{s+1/2}\times\H{s-1}}) (\|\nabla_x\psi\|_{\H{s}} + \|\eta\|_{\H{s+1/2}}), \\
\|R_4(\eta,\psi)\|_{\L(\H{s+1/2},\H{s})} 
& \le C(\|(\eta,\nabla_x \psi)\|_{\H{s+1/2}\times\H{s-1}}) (\|\nabla_x\psi\|_{\H{s}} + \|\eta\|_{\H{s+1/2}}),
\end{align*}
We introduce the new variable
\begin{equation}
\label{eq:identity-u-eta-psi}
u = \T{q} \omega - i \T{p} \eta.
\end{equation}
Then $ u \in \Hdot{s}(\Td) $, and
\begin{equation*}
\|u\|_{\H{s}} \le C(\|\eta\|_{\H{s+1/2}})(\|\nabla_x \psi\|_{\H{s-1}} + \|\eta\|_{\H{s+1/2}}).
\end{equation*}

\begin{proposition}
\label{prop:reversibility-u-eta-psi}
For $ s > 3 + d/2 $ and $ \varepsilon_0 $ sufficiently small, if $ u \in \Hdot{s}(\Td) $ such that $ \|u\|_{\H{s}} < \varepsilon_0,$ then there exists a unique $ (\psi,\eta) \in \Hdot{s}(\Td) \times \Hdot{s+1/2}(\Td) $ satisfying simultaneously the estimate
\begin{equation*} 
\|(\psi,\eta)\|_{\H{s}\times\H{s+1/2}} \le 2 \|u\|_{\H{s}},
\end{equation*}
and the identity~\eqref{eq:identity-u-eta-psi}. 
Moreover, if $ u \in \Cls^{1,s}(T,\varepsilon_0) $, then $ \eta \in \Cls^{1,s+1/2}(T,2\varepsilon_0) $, with
\begin{equation*}
\|\pt \eta\|_{\Linf([0,T],\H{s-1})} \le 2 \|\pt u\|_{\Linf([0,T],\H{s-3/2})}.
\end{equation*}
\end{proposition}
\begin{proof}
Observe that~$ q $ is a real symbol and an even function in~$ \xi $, and the real part of~$ p $ is even in~$ \xi $, while its imaginary part is odd in~$ \xi $, so by Lemma~\ref{lem:real-part-preserving-property}, $ \T{p} $ and $ \T{q} $ preserves the real part of a function. Therefore, by the definition of $ u $, we have
\begin{equation}
\label{eq:inversion-formula-u-eta-psi}
\T{p(\eta)} \eta = - \Im\, u,
\end{equation}
where we write $ p = p(\eta) $ to emphasize the fact that~$ p $ depends on~$ \eta $. Equivalently,
\begin{equation*}
\eta = \Psi(\eta),
\end{equation*}
where $ \Psi : \Hdot{\mu+1/2}(\Td) \to \Hdot{\mu+1/2}(\Td) $ for any $ s-3/2 \le \mu \le s $, is defined by
\begin{equation*}
\Psi(\eta) \bydef - |D_x|^{-1/2} \Im\, u + |D_x|^{-1/2}\T{|\xi|^{1/2}-p(\eta)} \eta.
\end{equation*}
A direct application of the Banach fixed point theorem is not possible because $ \Psi $ is a contraction only when $ \|\eta\|_{\Holder{2}} $ is sufficiently small. We need to apply the same iterative scheme while proving that each element of the sequence of iteration is sufficiently small. Let $ \eta_0 = 0 $ and $ \eta_{n+1} = \Psi(\eta_n) $ for $ n \in \N $. Denote $ p_n = p(\eta_n) $ for simplicity. We claim that for all $ n\in\N $, 
\begin{equation*}
\|\eta_n\|_{\H{s+1/2}} \le 2 \|\Im\,u\|_{\H{s}}.
\end{equation*}
We prove this by induction. Clearly it is true for $ n = 0 $. Suppose it is proven for $ n $, then,
\begin{align*}
\|\eta_{n+1}\|_{\H{s+1/2}} 
 = \|\Psi(\eta_n)\|_{\H{s+1/2}}  
& \le \|\Im\, u\|_{\H{s}} + C \M{1/2}{0}{d/2+1}{|\xi|^{1/2}-p_n}\|\eta_n\|_{\H{s+1/2}} \\
& \le \|\Im\, u\|_{\H{s}} + C(\|\eta_n\|_{\H{s+1/2}}) \|\eta_n\|_{\H{s+1/2}} \|\Im\, u\|_{\H{s}} \\
& \le 2 \|\Im\, u\|_{\H{s}},
\end{align*}
whenever $ \varepsilon_0 $ is sufficiently small.
In particular $ \|\eta_{n}\|_{\H{s+1/2}} \le 2\varepsilon_0$ for $ n \in \N $, and
\begin{align*}
\|\eta_{n+1}-\eta_n\|_{\H{\mu+1/2}}
& \le \|\T{|\xi|^{1/2}-p_n} (\eta_n-\eta_{n-1})\|_{\H{\mu}} + \|\T{p_n-p_{n-1}} \eta_{n-1}\|_{\H{\mu}} \\
& \le C(\|\eta_n\|_{\H{s+1/2}}) \|\eta_n\|_{\H{s+1/2}} \|\eta_n-\eta_{n-1}\|_{\H{\mu+1/2}}  \\
& \qquad + C(\|(\eta_{n-1},\eta_n)\|_{\H{s+1/2}\times\H{s+1/2}}) \|\eta_n - \eta_{n-1}\|_{\H{\mu+1/2}} \|\eta_{n-1}\|_{\H{\mu+1/2}} \\
& \le C \varepsilon_0 \|\eta_n-\eta_{n-1}\|_{\H{\mu+1/2}}.
\end{align*}
Hence when $ \varepsilon_0 $ is so small that $ C \varepsilon_0 < 1 $, $ \{\eta_n\}_n $ is a Cauchy sequence in $ \Hdot{s+1/2}(\Td) $, and converges to some $ \eta \in \Hdot{s+1/2}(\Td) $ such that
\begin{equation*}
\eta = \Psi(\eta), \quad \|\eta\|_{\H{s+1/2}} \le 2 \|\Im\, u\|_{\H{s}}.
\end{equation*}
The uniqueness of such $ \eta \in \Hdot{\mu+1/2}(\Td) $ for any $ s-3/2 \le \mu \le s $ comes from the same contraction estimate. Therefore~$ \eta $ is independent of the choice of the space $ \Hdot{\mu+1/2}(\Td) $ whenever $ s-3/2 \le \mu \le s $.

Write $ \T{q} = \pi(D_x) + \T{q-1} $, $ \T{p} = (\pi(D_x) + \T{p/|\xi|^{1/2}-1})|D_x|^{1/2} $, where
\begin{align*}
\|\T{q-1}\|_{\L(\Hdot{s},\Hdot{s})}
& \le C \|q-1\|_{\Linf}
\le C(\|\eta\|_{\H{s+1/2}}) \|\eta\|_{\H{s+1/2}}
\le C\varepsilon_0. \\
\|\T{p/|\xi|^{1/2}-1}\|_{\L(\Hdot{s},\Hdot{s})}
& \le C \M{0}{0}{d/2+1}{p/|\xi|^{1/2}-1}
\le C(\|\eta\|_{\H{s+1/2}}) \|\eta\|_{\H{s+1/2}}
\le C \varepsilon_0.
\end{align*} 
Observe that $ \pi(D_x) = \Id_{\Hdot{\sigma}} : \Hdot{\sigma} \to \Hdot{\sigma} $ for $ \sigma \in \R $ (see Proposition~\ref{prop:T_1=pi}), hence by means of Neumann series, $ \T{p} : \Hdot{s+1/2}(\Td) \to \Hdot{s}(\Td) $ and $ \T{q} : \Hdot{s}(\Td) \to \Hdot{s}(\Td) $ are invertible with their inverses being
\begin{equation*}
\T{q}^{-1} = (\Id + \T{q-1})^{-1},
\quad
\T{p}^{-1} = |D_x|^{-1/2} (\Id + \T{p/|\xi|^{1/2}-1})^{-1}.
\end{equation*}
and satisfying the estimates,
\begin{equation*}
\|\T{q}^{-1}\|_{\L(\Hdot{s},\Hdot{s})}  \le 1 + C\varepsilon_0, 
\quad
\|\T{p}^{-1}\|_{\L(\Hdot{s},\Hdot{s+1/2})} \le 1 + C\varepsilon_0.
\end{equation*}
Therefore there is a unique $ \psi \in \Hdot{s}(\Td) $ such that~\eqref{eq:identity-u-eta-psi} is satisfied, which is given by the formula
\begin{equation*}
\psi = (\T{q}^{-1} \Re - \T{B}\T{p}^{-1}\Im) u,
\end{equation*}
from which the estimate
\begin{equation*}
\|\psi\|_{\H{s}} \le (1 + C\varepsilon_0) \|u\|_{\H{s}} \le 2 \|u\|_{\H{s}},
\end{equation*}
for sufficiently small $ \varepsilon_0 $.

As for the estimate of time derivative, observe that the symbol $ p $ is a function of $ \nabla_x \eta $, $ \nabla_x^2 \eta $ and $ \xi $. More precisely, one may write 
\begin{equation*}
p(t,x,\xi) = f(\nabla_x \eta(t,x),\nabla_x^2 \eta(t,x),\xi)
\end{equation*}
with some $ f=f(v,M,\xi) \in \Cinf(\Rd\times\R^{d\times d}\times \Rd_*) $. Denote for simplicity $ \nabla = \nabla_{v,M} $, then
\begin{align*}
\pt p
& = \big\langle \nabla_v f(\nabla_x \eta,\nabla_x^2 \eta,\xi), \nabla_x \pt \eta \big\rangle + \big\langle \nabla_M f(\nabla_x \eta,\nabla_x^2 \eta,\xi), \nabla_x^2 \pt \eta \big\rangle \\
& = \big\langle \nabla f(\nabla_x \eta,\nabla_x^2 \eta,\xi), (\nabla_x \pt \eta,\nabla_x^2 \pt \eta) \big\rangle.
\end{align*}
Therefore a formal differentiation of the formula~\eqref{eq:inversion-formula-u-eta-psi} gives,
\begin{equation*}
\T{p} \pt \eta + \T{\langle \nabla f(\nabla_x \eta,\nabla_x^2 \eta,\xi), (\nabla_x \pt \eta,\nabla_x^2 \pt \eta) \rangle} \eta = - \Im\, \pt u,
\end{equation*}
which leads to the consideration of the map $ \tilde{\Psi}_t : \Hdot{s-1}(\Td) \to \Hdot{s-1}(\Td) $,
\begin{equation*}
\tilde{\Psi}_t(\zeta) = |D_x|^{-1/2} \T{|\xi|^{1/2}-p} \zeta - |D_x|^{-1/2} \T{\langle \nabla f(\nabla_x \eta,\nabla_x^2 \eta,\xi), (\nabla_x \zeta,\nabla_x^2 \zeta) \rangle}\eta - |D_x|^{-1/2} \Im\, \pt u,
\end{equation*}
defined for almost every time $ t \in [0,T] $, such that, whenever $ \pt\eta $ is defined,
\begin{equation*}
\pt \eta = \tilde{\Psi}_t(\pt \eta).
\end{equation*}
We show that $ \tilde{\Psi}_t $ is a contraction for~$ \varepsilon_0 $ sufficiently small, indeed,
\begin{align*}
\|\tilde{\Psi}_t(\zeta_1)-\tilde{\Psi}_t(\zeta_2)\|_{\H{s-1}}
& \le \| \T{|\xi|^{1/2}-p} (\zeta_1-\zeta_2) \|_{\H{s-3/2}} \\
& \qquad + \|\T{\langle \nabla f(\nabla_x \eta,\nabla_x^2 \eta,\xi), (\nabla_x (\zeta_1-\zeta_2),\nabla_x^2 (\zeta_1-\zeta_2) \rangle)}\eta \|_{\H{s-3/2}} \\
& \le C(\|\eta\|_{\H{s+1/2}})\|\eta\|_{\H{s+1/2}} \|\zeta_1-\zeta_2\|_{\H{s-1}} \\
& \le C \varepsilon_0 \|\zeta_1-\zeta_2\|_{\H{s-1}}.
\end{align*}
Therefore, by Banach fixed point theorem, for almost every $ t \in [0,T] $, there exists a unique $ \zeta(t) \in \Hdot{s-1}(\Td) $ such that
\begin{equation*}
\tilde{\Psi}_t(\zeta(t)) = \zeta(t), \quad \|\zeta(t)\|_{\H{s-1}} \le 2 \|\Im\,\pt u(t)\|_{\H{s-1}}.
\end{equation*}
We claim that $ \pt \eta = \zeta $. Indeed, define 
\begin{equation*}
 \tilde{\eta}(t) = \eta(0) + \int_0^t \zeta(s) \ds,
\end{equation*}
Then $ \tilde{\eta}(0) = \eta(0) $, $ \pt \tilde{\eta} = \zeta $, $ \tilde{\eta} \in \Holder{1}([0,T],\Hdot{s-1}(\Td)) $, and
\begin{equation*}
\|\tilde{\eta}\|_{C([0,T],\H{s-1})} \le (1+T) C\varepsilon_0.
\end{equation*}
To show that $ \eta \equiv \tilde{\eta} $, consider the function $ g \in \Holder{1}([0,T],\Hdot{s-3/2}(\Td)) $ defined by
\begin{equation*}
g(t) = \T{p(\tilde{\eta}(t))} \tilde{\eta}(t) + \Im\, u(t).
\end{equation*}
Then $ g(0) = 0 $ by the definition of $ \eta(0) $, while $ \pt g \equiv 0 $ by the definition of $ \zeta(t) $. Therefore $ g \equiv 0 $. That implies, by the uniqueness of the $ \eta $ in $ \Hdot{s-3/2}(\Td) $, for $ \varepsilon_0 $ sufficiently small, and for each time $ t \in [0,T] $, $ \eta(t) = \tilde{\eta}(t) $.
Consequently, $ \eta = \tilde{\eta} \in \Holder{1}([0,T],\Hdot{s-1}(\Td)) $. 
The continuity $ \eta \in C([0,T],\Hdot{s+1/2}(\Td)) $, is a direct consequence of Lemma~\ref{lem:stability-u-eta} to be proven later.
\end{proof}

\begin{remark}
All the symbols that depends on~$ \eta $, say $ a = a(\eta) \in \Symbol{m}{\rho} $, that appear in this article are functions of $ \nabla_x\eta $ and $ \nabla_x^2 \eta $, that is, following the proof of the previous lemma,
\begin{equation*}
a(t,x,\xi) = f(\nabla_x\eta(t,x),\nabla_x^2\eta(t,x),\xi),
\end{equation*}
for some $ f = f(v,M,\xi) \in \Cinf(\Rd\times\R^{d\times d}\times \Rd_*) $.

As a consequence of this lemma, when $ u \in \Cls^{1,s}(T,\varepsilon_0) $ is defined by~\eqref{eq:def-u} with~$ \varepsilon_0 $ being sufficiently small, we can express $ \eta = \eta(u) \in \Hdot{s+1/2}(\Td) $, and thus consider $ a = a(u) $ as a symbol depending on~$ u $,
\begin{equation*}
a(u) = f(\nabla_x \eta(u), \nabla_x^2 \eta(u),\xi).
\end{equation*}
When considering time derivatives of~$ a $, we use the formula,
\begin{equation*}
\pt a( u) = \big\langle (\nabla_{v,M} f)(\nabla_x \eta(u),\nabla_x^2 \eta(u),\xi), (\nabla_x \pt \eta(u), \nabla_x^2 \pt \eta(u)) \big\rangle.
\end{equation*}

Consequently, we write $ p = p(u) $, $ q = q(u) $. Expressing also $ \dot{\psi} = \dot{\psi}(u) \in \Hdot{s}(\Td) $ where $ \dot{\psi} = \pi(D_x)\psi $, so that $ B=B(u) $, because~$ B $ does not depend on the zero frequency of~$ \psi $. Then we have explicitly the following formula,
\begin{equation*}
\pi(D_x) \eta = - \T{p(u)}^{-1} \Im u,\quad
\pi(D_x) \psi = (\T{q(u)}^{-1} \Re - \T{B(u)}\T{p(u)}^{-1}\Im) u,
\end{equation*}
which will be used in Proposition~\ref{prop:paralinearized-pb-of-control} to derive an equation for~$ u $.

For $ \R $-linear operators of the form $ R_i(\eta,\psi) $, $ (i=1,2) $ that depend on nonzero frequencies of $ (\eta,\psi) $, such that $ R_1 : \H{s}(\Td) \to \H{s}(\Td),\ R_2 : \H{s+1/2}(\Td) \to \H{s}(\Td) $, and $ R_i = R_i \pi(D_x) $, we can correspond to them $ \R $-linear operators of the form $ \tilde{R}_i(u) : \H{s}(\Td) \to \H{s}(\Td) $, such that
\begin{equation*}
R_1(\eta,\psi) \psi = \tilde{R}_1(u)u, 
\quad
R_2(\eta,\psi) \eta = \tilde{R}_2(u)u.
\end{equation*}
Indeed, we simply let
\begin{equation*}
\begin{split}
\tilde{R}_1(u) & = R_1(\eta(u),\dot{\psi}(u)) (\T{q(u)}^{-1} \Re - \T{B(u)}\T{p(u)}^{-1}\Im), \\
\tilde{R}_2(u) & = - R_2(\eta(u),\dot{\psi}(u)) \T{p(u)}^{-1} \Im.
\end{split}
\end{equation*}
For simplicity, we write henceforth by an abuse of notation $ R_i(u) = \tilde{R}_i(u) $, for this correspondence.
\end{remark}

\begin{proposition}
\label{prop:paralinearized-pb-of-control}
Let $ s > 3 + d/2 $, and $ u \in \Cls^{0,s}(T,\varepsilon_0) $ for $ T > 0 $ and $ \varepsilon_0 $ sufficiently small, then $ u $ satisfies the equation
\begin{equation}
\label{eq:equation-control-of-u}
\pt u + P(u) u + R(u) u = \B(u) F + \beta(u) F,
\end{equation}
with $ P(u) $ being a para\-differential operator defined by,
\begin{equation}
\label{eq:definition-P}
P(u)  = i \T{\gamma(u)} + \nabla_x \cdot \T{V(u)}  - g \T{r(u)^{-1}} \Im + i\T{r(u)}M_{\b}\Re, 
\end{equation}
where the symbol $ r(u) = p^{(1/2)}(u)/q(u) $ is homogeneous in $ \xi $ of order $ 1/2 $.
The operators $ R(u) $, $ \B(u) $ and $ \beta(u) $ depend on $ u $, the latter two being explicitly defined as follows,
\begin{equation*}
\B(u) F  = \chi_T^{} \T{q(u)} \phiw \Re F, \quad 
\beta(u) F  = \chi_T^{} \T{q(u)} \T{B(u)(\phiw\Re F)} \T{p(u)}^{-1} \Im u. 
\end{equation*}
Furthermore,
\begin{equation}
\label{eq:estimate-B-beta-R}
\begin{split}
\B(u) & \in C([0,T],\L(\Hdot{\sigma},\Hdot{\sigma})),\quad \forall \sigma \ge 0, \\
\beta(u) & \in \Linf([0,T],\L(\Hdot{s},\Hdot{s+1/2})) \cap C([0,T],\L(\Hdot{s},\Hdot{s-1})), \\
R(u) & \in \Linf([0,T],\L(\Hdot{s},\Hdot{s})) \cap C([0,T],\L(\Hdot{s},\Hdot{s-3/2})),
\end{split}
\end{equation}
and they satisfy the following estimates
\begin{equation}
\label{eq:estimate-various-operators}
\begin{split}
\|\B(u)\|_{C([0,T],\L(\Hdot{\sigma},\Hdot{\sigma}))}  & \lesssim 1,\\
\|\beta(u)\|_{\Linf([0,T],\L(\Hdot{s},\Hdot{s+1/2}))} & \lesssim \varepsilon_0, \\
\|R(u)\|_{\Linf([0,T],\L(\Hdot{s},\Hdot{s}))} & \lesssim \varepsilon_0^\vartheta,
\end{split}
\end{equation}
for some $ \vartheta > 0 $.
\end{proposition}

\begin{proof}
By~\eqref{eq:equation-pre-symmetrization} and the definition of $ u $,
\begin{equation*}
\pt u + \tilde{P}(u) u + \tilde{R}(u) u = \B(u) F + \beta(u) F,
\end{equation*}
with $ \tilde{R}(u) =  - i R_1(u) - i R_2(u) + R_3(u) + R_4(u)$ and 
\begin{equation*}
\tilde{P}(u) = i \T{\gamma} + \nabla_x \cdot \T{V} - g \T{q}\T{p^{-1}} \Im + i\T{p}M_{\b} (\T{q^{-1}}\Re - \T{B}\T{p^{-1}}\Im).
\end{equation*}
Therefore it suffices to put
\begin{equation*}
R(u) = \tilde{R}(u) + g(\T{r^{-1}}-\T{q}\T{p}^{-1}) \Im 
+ i (\T{p}M_{\b}\T{q}^{-1} - \T{r^{-1}}M_{\b}) \Re
- i \T{p}M_{\b}\T{B}\T{p}^{-1} \Im.
\end{equation*}
Now the estimate for $ R(u) $ comes from that of $ \tilde{R}(u) $ and a symbolic calculus. The estimate for $ \B(u) $ comes directly from Theorem~\ref{thm:Operator-Norm-Estimate-Paradiff}. The estimate of~$ \beta(u) $ uses Lemma~\ref{lemma:estimate-B,V},
\begin{equation*}
\|\beta(u) F\|_{\H{s+1/2}}
\lesssim \|B(u) F\|_{\Linf} \|u\|_{\H{s}}
\lesssim C(\|u\|_{\H{s}}) \|F\|_{\H{s}} \|u\|_{\H{s}}.
\end{equation*}
The continuity of the these three operators is a consequence of Lemma~\ref{lem:stability-B-B^*}, Lemma~\ref{lem:stability-beta} and Lemma~\ref{lem:stability-R} to be proven later.
\end{proof}

\section{$ \Ltwo $ Linear Control}
\label{sec:Ltwo-linear-control}

Let~$ s $ be sufficiently large, $ T > 0 $, $ \varepsilon_0 > 0 $, fix $ \underline{u} \in \Cls^{1,s}(T,\varepsilon_0) $, and denote for simplicity
\begin{equation*}
P = P(\underline{u}), \quad \B = \B(\underline{u}).
\end{equation*}
The purpose of this section is to prove the $ \Ltwo $-null controllability, when~$ \varepsilon_0 $ is sufficiently small, of the following equation without the perturbation terms,
\begin{equation}
\label{eq:equation-paradiff-linear-non-perturbed}
(\pt + P) u = \B F.
\end{equation}

\begin{proposition}
\label{prop:Ltwo-linear-control-HUM}
Suppose that~$ \omega $ satisfies the geometric control condition, $ s $ is sufficiently large, $ T > 0 $, $ \varepsilon_0 > 0 $, $ \underline{u} \in \Cls^{1,s}(T,\varepsilon_0) $, then when~$ \varepsilon_0 $ is sufficiently small, the $ \Ltwo $-null controllability of~\eqref{eq:equation-paradiff-linear-non-perturbed} holds. More precisely, there exists a linear operator
\begin{equation*}
\Theta = \Theta(\underline{u}) : \Ldottwo(\Td) \to C([0,T],\Ldottwo(\Td)),
\end{equation*}
satisfying the estimate
\begin{equation*}
\|\Theta\|_{\L(\Ldottwo, C([0,T],\Ldottwo))} \lesssim 1,
\end{equation*}
such that, for $ u_0 \in \Ldottwo(\Td) $ and
\begin{equation*}
F = \Theta u_0 \in C([0,T],\Ldottwo(\Td)),
\end{equation*}
the unique solution $ u \in C([0,T],\Ldottwo(\Td)) $ of~\eqref{eq:equation-paradiff-linear-non-perturbed} with initial data $ u(0) = u_0 $, vanishes at time~$ T $, that is $ u(T) = 0 $.
\end{proposition}

We use the \textit{Hilbert Uniqueness Method}, or HUM for short, to prove this proposition.

\subsection{Hilbert Uniqueness Method}
\label{sec:HUM}

The HUM establishes by a duality argument on a well-chosen Hilbert space, the equivalence between the null controllability of the original equation (that is in our case~\eqref{eq:equation-paradiff-linear-non-perturbed}) and an \textit{observability inequality} to its dual equation (see~\eqref{eq:equation-dual-non-perturbed}), with respect to the Hilbertian structure. Because~$ P $ and~$ \B $ are not $ \C $-linear but only $ \R $-linear, we then choose the real Hilbert space $ \Ldottwo(\Td,\C) $ equipped with the scalar product $ \Re(\cdot,\cdot)_{\Ltwo} $. Therefore our dual equation is
\begin{equation}
\label{eq:equation-dual-non-perturbed}
\pt u - P^* u = 0,
\end{equation}
where~$ P^* $ is the formal adjoint of~$ P $, with respect to $ (\Ldottwo(\Td),\Re (\cdot,\cdot)_{\Ltwo}) $, in the sense that, for all $ f,g \in \Cinf(\Td) $,
\begin{equation*}
\Re( P f, g)_{\Ltwo} = \Re(f, P^* g)_{\Ltwo}.
\end{equation*}

\begin{proposition}
$ P^* = -i\T{\gamma}^* - \T{V}^*\cdot\nabla_x - g i \T{r^{-1}}^* \Re + M_{\b}\T{r}^* \Im, $ where for a para\-differential operator $ \T{a} $, $ \T{a}^* $ denotes its formal adjoint with respect to $ (\cdot,\cdot)_{\Ltwo} $.
\end{proposition}
\begin{proof}
For a para\-differential operator $ \T{a} $, its formal adjoint $ \T{a}^* $ with respect to the scalar product $ (\cdot,\cdot)_{\Ltwo} $ is evidently at the same time its formal adjoint with respect to the scalar product $ \Re(\cdot,\cdot)_{\Ltwo} $. Moreover, by a direct verification, $ \Im^* = i \Re $. Therefore,
\begin{align*}
P^* 
& = \T{\gamma}^* i^* - \T{V}^*\cdot\nabla_x - g \Im^* \T{r^{-1}}^* + (i\Re)^* M_{\b}\T{r}^* \\
& = -i\T{\gamma}^* - \T{V}^*\cdot\nabla_x - g i\Re \T{r^{-1}}^* + \Im M_{\b}\T{r}^*.
\end{align*}
It remains to show that for $ a = r^{\pm 1} $ or $ m_{\b} $, $ [\T{a}^*,\Re] = [\T{a}^*,\Im] = 0 $. Equivalently $ [\T{a},\Re] = 0 $. This is a consequence of Lemma~\ref{lem:real-part-preserving-property}, as in either case~$ a $ is real valued, and is an even function of~$ \xi $.
\end{proof}

Now HUM proceeds as follows. Define the range operator $ \Range = \Range(\underline{u}) $ by
\begin{equation*}
\Range : \Ltwo([0,T],\Ldottwo(\Td)) \to \Ldottwo(\Td), \quad G \mapsto u(0),
\end{equation*}
where $ u \in C([0,T],\Ldottwo(\Td)) $ is the unique solution to the Cauchy problem
\begin{equation*}
(\pt + P)u = G, \quad u(T) = 0.
\end{equation*}
Also define the solution operator $ \Sol = \Sol(\underline{u}) $ by
\begin{equation*}
\Sol : \Ldottwo(\Td) \to C([0,T],\Ldottwo(\Td)) \subset \Ltwo([0,T],\Ldottwo(\Td)), \quad v_0 \mapsto v,
\end{equation*}
where $ v \in C([0,T],\Ldottwo(\Td)) $ is the unique solution to the dual equation~\eqref{eq:equation-dual-non-perturbed} with initial data $ v(0) = v_0 $.

\begin{remark}
For the well-posedness of these equations, we refer to Appendix~\ref{app:linear-equations}. Moreover, for all $ \mu \ge 0 $, we have
\begin{equation*}
\Range|_{\Ltwo([0,T],\Hdot{\mu})} : \Ltwo([0,T],\Hdot{\mu}(\Td)) \to \Hdot{\mu}(\Td), \quad
\Sol|_{\Hdot{\mu}}  : \Hdot{\mu}(\Td) \to C([0,T],\Hdot{\mu}(\Td)).
\end{equation*}
\end{remark}

\begin{remark}
In classical literatures, the operator $ \Range\B $ is called the range operator, while $ \B^*\Sol $ is called the solution operator. We isolate the operator~$ \B $ for later simplicity.
\end{remark}

\begin{proposition}
\label{prop:duality-HUM}
Let $ F \in \Ltwo([0,T],\Ldottwo(\Td)) $, and $ v_0 \in \Ldottwo(\Td) $, then the duality holds
\begin{equation}
\label{eq:duality-relation-HUM}
-\Re(\Range \B F, v_0)_{\Ltwo} = \Re(F,\B^* \Sol v_0)_{\Ltwo([0,T],\Ltwo)},
\end{equation}
or formally, with respect to $ \Re(\cdot,\cdot)_{\Ltwo} $, $ \Range \B = - (\B^* \Sol)^* $.
\end{proposition}
\begin{proof}
By a density argument, it suffice to prove the identity for $ F \in C([0,T],\Hdot{\infty}(\Td)) $, and $ v_0 \in \Hdot{\infty}(\Td) $. Then $ v = \Sol v_0 \in C^1([0,T],\Hdot{\infty}(\Td))$. And the solution $ u $ to~\eqref{eq:equation-paradiff-linear-non-perturbed} with $ u(T) = 0 $ belongs to $ C^1([0,T],\Hdot{\infty}(\Td)) $. By the Newton-Leibniz formula,
\begin{align*}
\Re (F,\B^* v)_{\Ltwo([0,T],\Ltwo)}
& = \Re (\B F, v)_{\Ltwo([0,T],\Ltwo)} 
= \Re ((\pt + P) u, v)_{\Ltwo([0,T],\Ltwo)} \\
& = \Re (u(t),v(t))_{\Ltwo}|_0^T + \Re (u, (-\pt + P^*) v)_{\Ltwo([0,T],\Ltwo)} \\
& = \Re (u(T),v(T))_{\Ltwo} - \Re (u(0),v(0))_{\Ltwo}.
\end{align*}
Now that $ u(T) = 0 $, the duality relation~\eqref{eq:duality-relation-HUM} follows.
\end{proof}

\begin{proposition}
\label{prop:observability-Ltwo-pafadiff-nonperturbed}
Suppose that~$ \omega $ satisfies the geometric control condition, $ s $ is sufficiently large, $ T > 0 $, $ \varepsilon_0 > 0 $, and $ \underline{u} \in \Cls^{1,s}(T,\varepsilon_0) $. Then for~$ \varepsilon_0 $ sufficiently small, the $ \Ltwo $-observability of~\eqref{eq:equation-dual-non-perturbed} holds. That is, for all its solution~$ u $, with initial data in~$ \Ldottwo(\Td) $,
\begin{equation}
\label{eq:observability-Ltwo-paradiff-nonperturbed}
\|u(0)\|_{\Ltwo}^2 \lesssim \int_0^T \|\B^* u(t)\|_{\Ltwo}^2 \dt.
\end{equation}
Here $ \B^* = \chi^{}_T \phiw \T{q}^* \Re $ is the dual operator of~$ \B $ with respect to $ \Re(\cdot,\cdot)_{\Ltwo} $.
\end{proposition}

\begin{remark}
\label{remark:observability-erasing-chi_T}
Now that $ \chi_T^{}(t) = 1 $ for $ t \le T/2 $, and that Proposition~\ref{prop:observability-Ltwo-pafadiff-nonperturbed} is stated for all $ T > 0 $, replacing~$ T $ with~$ 2T $ in~\eqref{eq:observability-Ltwo-paradiff-nonperturbed}, then for~\eqref{eq:observability-Ltwo-paradiff-nonperturbed} to be satisfied, we only need to prove the following observability for all $ T > 0 $,
\begin{equation*}
\|u(0)\|_{\Ltwo}^2 \lesssim \int_0^T \|\phiw \T{q}^* \Re\, u(t)\|_{\Ltwo}^2 \dt.
\end{equation*}
Therefore, we may simply omit the factor $ \chi_T^{} $ in~\eqref{eq:observability-Ltwo-paradiff-nonperturbed}.
\end{remark}

The effort of the rest of the section will be devoted to proving this observability, which states the coercivity of the operator $ \B^*\Sol $ because~\eqref{eq:observability-Ltwo-paradiff-nonperturbed} writes in a compact form as
\begin{equation*}
\|u_0\|_{\Ltwo} \lesssim \|\B^* \Sol u_0\|_{\Ltwo([0,T],\Ltwo)}.
\end{equation*}
Once this is proven, define the $ \R $-linear operator
\begin{equation}
\label{eq:definition-K}
\K \bydef - \Range \B \B^* \Sol : \Ldottwo(\Td) \to \Ldottwo(\Td), \quad \forall \mu \ge 0.
\end{equation}
\begin{remark}
By Proposition~\ref{prop:duality-HUM}, $ \K = (\B^* \Sol)^* \B^* \Sol $.
By an energy estimate, we have for all $ \mu \ge 0 $, $ \K|_{\Hdot{\mu}} : \Hdot{\mu}(\Td) \to \H{\mu}(\Td) $.
\end{remark}

Consider the continuous $ \R $-bilinear form on $ \Ldottwo(\Td) $,
\begin{equation*}
\varpi(f_0,v_0) \bydef \Re(\K f_0, v_0)_{\Ltwo}.
\end{equation*}
It is coercive by~\eqref{eq:observability-Ltwo-paradiff-nonperturbed},
\begin{equation*}
\varpi(v_0,v_0) 
= \Re ((\B^* \Sol)^* \B^* \Sol v_0, v_0)_{\Ltwo} 
= \Re (\B^* \Sol v_0,\B^* \Sol v_0)_{\Ltwo([0,T],\Ltwo)} 
\gtrsim \|v_0\|_{\Ltwo}^2.
\end{equation*}
By Lax-Milgram's theorem, for $ u_0 \in \Ldottwo(\Td) $, there exists a unique $ f_0 \in \Ldottwo(\Rd) $, such that
\begin{equation*}
\varpi(f_0,v_0) = \Re(u_0,v_0)_{\Ltwo}, \quad \forall v_0 \in \Ldottwo(\Td),
\end{equation*}
and consequently, we have the invertibility of~$ \K $,
\begin{equation}
\label{eq:inverse-K-f-u}
\K f_0 = u_0,\quad \|f_0\|_{\Ltwo} \lesssim \|u_0\|_{\Ltwo}.
\end{equation}
Therefore, we have proven the following proposition, as a consequence of the $ \Ltwo $-observability.
\begin{proposition}
\label{prop:HUM-Ltwo}
Suppose that~$ \omega $ satisfies the geometric control condition, $ s $ is sufficiently large, $ T > 0 $, $ \varepsilon_0 > 0 $, and $ \underline{u} \in \Cls^{1,s}(T,\varepsilon_0) $. Then for $ \varepsilon_0 $ sufficiently small, $ \K $ defines an isomorphism on $ \Ldottwo(\Td) $ with
\begin{equation*}
\|\K\|_{\L(\Ldottwo,\Ldottwo)} + \|\K^{-1}\|_{\L(\Ldottwo,\Ldottwo)} \lesssim 1.
\end{equation*}
\end{proposition}

To construct $ \Theta $ and thus prove Proposition~\ref{prop:Ltwo-linear-control-HUM}, we set
\begin{equation*}
F = -\B^* \Sol f_0 = -\B^* \Sol \K^{-1} u_0 \in C([0,T],\Ldottwo(\Td)),
\end{equation*}
then by~\eqref{eq:inverse-K-f-u} and the definition of $ \K $, $ \Range \B F = u_0 $. Therefore,
\begin{equation}
\label{eq:definition-Theta}
\Theta \bydef - \B^*\Sol\K^{-1} : \Ldottwo(\Td) \to C([0,T],\Ldottwo(\Rd))
\end{equation}
defines a desired control operator.

It remains to prove Proposition~\ref{prop:observability-Ltwo-pafadiff-nonperturbed}.

\subsection{Reduction to Pseudodifferential Equation}

This sections shows that, the observability of the para\-differential equation~\eqref{eq:equation-dual-non-perturbed} can be reduced to that of a pseudo\-differential equation. To do this, write
\begin{equation*}
-P^* = i Q + R_Q,
\end{equation*}
where the pseudo\-differential operator $ Q = Q(\udlu) $ is
\begin{equation}
\label{eq:definition-Q}
Q = \pi(D_x) \op(\gamma\pi) + \pi(D_x) \big( V \cdot D_x + g\op(r^{-1}\pi) \Re + \op(r \cdot m_b) \, i \Im  \big).
\end{equation}

\begin{lemma}
\label{lem:estimate-R_Q}
Suppose that~$ s $ is sufficiently large, $ T > 0 $, $ \varepsilon_0 > 0 $, and $ \udlu \in \Cls^{0,s}(T,\varepsilon_0) $. Then for~$ \varepsilon_0 $ sufficiently small,
\begin{equation*}
R_Q \in \Linf([0,T],\L(\Ldottwo,\Ldottwo)) \cap C([0,T],\L(\Ldottwo,\Hdot{-3/2})),
\end{equation*}
with the estimate
\begin{equation*}
\| R_Q \|_{\Linf([0,T],\L(\Ldottwo,\Ldottwo))} \lesssim \varepsilon_0. 
\end{equation*}
\end{lemma}
\begin{proof}
Lemma~\ref{lem:stability-P-P^*} later shows that $ P^* \in C([0,T],\L(\Ldottwo,\Hdot{-3/2})) $. To prove that $ Q \in C([0,T],\L(\Ldottwo,\Hdot{-3/2})) $, the main estimate is a similar contraction estimate for the principal term. Let $ \udlu_i \in \Hdot{s}(\Td)\ (i=0,1) $ be such that $ \|\udlu_i\|_{\H{s}} < \varepsilon_0 $, then
\begin{equation*}
\|\op(\gamma(\udlu_1)^{(3/2)}\pi) - \op(\gamma(\udlu_2)^{(3/2)}\pi)\|_{\L(\Ldottwo,\Hdot{-3/2})}
\lesssim \|\udlu_1-\udlu_2\|_{\H{s}}.
\end{equation*}

To estimate the $ \Linf $-norm, write $ - R_Q = P^* + iQ = (P^* + P) + (P - iQ) $. For $ P^* + P $, we estimate $ \T{\gamma}^* - \T{\gamma} $. By the definition of $ \gamma = \gamma^{(3/2)} + \gamma^{(1/2)} $, 
\begin{equation}
\label{eq:gamma-almost-selfadjoint}
i\, \Im \gamma^{(1/2)} = \frac{1}{2} \pxi \cdot D_x \gamma^{(3/2)}.
\end{equation} 
Therefore, a symbolic calculus shows,
\begin{equation*}
\T{\gamma}^* = \T{\gamma} + O(\varepsilon_0)_{\L(\Ltwo,\Ltwo)}.
\end{equation*}
\begin{remark}
\label{rmk:smallness-symbol}
The remainder is of size $ O(\varepsilon_0)_{\L(\Ldottwo,\Ldottwo)} $. To prove this, we write $ \T{\gamma}^* = \T{\gamma^{(3/2)}-|\xi|^{3/2}}^* + \T{|\xi|^{3/2}}^* + \T{\gamma^{(1/2)}}^* $, and proceeds with the symbolic calculus, using the estimates that $ \M{3/2}{3/2}{d/2+1}{\gamma^{(3/2)}-|\xi|^{3/2}} \lesssim \varepsilon_0 $, $ \M{1/2}{1/2}{d/2+1}{\gamma^{(1/2)}} \lesssim \varepsilon_0 $. The idea is that the symbols do not differ much from some Fourier multipliers, in the sense that their differences are of size $ \varepsilon_0 $. This idea will be frequently used in later estimates of remainders, and will not again be explained in detail.
\end{remark}

To estimate $ P + iQ $, write
\begin{equation*}
\begin{split}
P + iQ
& = i \pi(D_x) \big(\T{\gamma}-\op(\gamma\pi)\big) + i \pi(D_x) (\T{V\cdot\xi}-\op(V\cdot\xi)) \\
& \qquad + ig \pi(D_x) \big(\T{r^{-1}}-\op(r^{-1}\pi)\big) \Re - \pi(D_x) (\T{r\cdot m_b}-\op(r\cdot m_b)) \Im.
\end{split}
\end{equation*}
We conclude with Proposition~\ref{prop:unparadifferentialization} that $ P + i Q = O(\varepsilon_0)_{\L(\Ldottwo,\Ldottwo)} $.
\end{proof}

Consider the following pseudo\-differential equation,
\begin{equation}
\label{eq:Pseudo-Linear-Equation}
D_t u + Q u = 0.
\end{equation}
\begin{proposition}
\label{prop:observability-Ltwo-Q}
Suppose that~$ \omega $ satisfies the geometric control condition, $ s $ is sufficiently large, $ T > 0 $, $ \varepsilon_0 > 0 $, and $ \underline{u} \in \Cls^{1,s}(T,\varepsilon_0) $. Then for~$ \varepsilon_0 $ sufficiently small, the $ \Ltwo $-observability of~\eqref{eq:Pseudo-Linear-Equation} holds.
That is, for all its solution with initial data in $ \Ldottwo(\Td) $,
\begin{equation}
\label{eq:observability-Ltwo-Q}
\|u(0)\|_{\Ltwo}^2 \lesssim \int_0^T \|\phiw \Re\, u(t)\|_{\Ltwo}^2 \dt.
\end{equation}
\end{proposition}
By Duhamel's formula and Lemma~\ref{lem:estimate-R_Q}, we can deduce Proposition~\ref{prop:observability-Ltwo-pafadiff-nonperturbed} from Proposition~\ref{prop:observability-Ltwo-Q}. Indeed, for every solution $ u \in C([0,T],\Ldottwo(\Td)) $ of equation~\eqref{eq:equation-dual-non-perturbed}, write $ u = v + w $ with $ v,w \in C([0,T],\Ldottwo(\Td)) $ such that
\begin{equation*}
(\pt + i Q) v = 0, \quad v(0) = u(0); 
\qquad
(\pt + i Q + R_Q) w  = - R_Q v, \quad w(0) = 0.
\end{equation*}
Then by Proposition~\ref{prop:app-eq-pseudo-well-posedness},
\begin{align*}
\|v\|_{C([0,T],\Ltwo)} 
& \lesssim \|u(0)\|_{\Ltwo}, \\
\|w\|_{C([0,T],\Ltwo)} 
& \lesssim \|R_Q v\|_{\Lone([0,T],\Ltwo)} 
 \lesssim \varepsilon_0 \|v\|_{\Lone([0,T],\Ltwo)} 
 \lesssim \varepsilon_0 \|u(0)\|_{\Ltwo}.
\end{align*}
Applying Proposition~\ref{prop:observability-Ltwo-Q},
\begin{align*}
\|u(0)\|_{\Ltwo}^2
& \lesssim\int_0^T \| \phiw \Re\,  v(t)\|_{\Ltwo}^2 \dt \\
& \lesssim \int_0^T \| \B^* u(t)\|_{\Ltwo}^2 \dt + \int_0^T \|(\B^* - \phiw \Re) u(t)\|_{\Ltwo}^2 \dt + \int_0^T \| \phiw \Re\, w(t)\|_{\Ltwo}^2 \dt \\
& \lesssim \int_0^T \| \B^* u(t)\|_{\Ltwo}^2 \dt + \varepsilon_0 \|u(0)\|_{\Ltwo}^2.
\end{align*}
The estimate above uses $ \|(\B^* - \phiw\Re) u\|_{\Ltwo} = \|\phiw\Re(\T{q}^*-\pi) u\|_{\Ltwo} \lesssim \varepsilon_0 \|u\|_{\Ltwo} $ (by Proposition~\ref{prop:T_1=pi}).
The observability for~$ u $ then follows for~$ \varepsilon_0 $ sufficiently small.

To deal with the problem caused by the $ \R $-linearity of the equation, we are going to consider $ u $ and $ \bar{u} $ simultaneously, and study the equation satisfied by the pair~$ \binom{u}{\bar{u}} $. First we derive the equation for~$ \bar{u} $ with the help of the following observation. For a pseudo\-differential operator with symbol~$ a $, we have
\begin{equation*}
\op(a) u = \overline{\op(\tilde{a}) \bar{u}}, \qquad
\tilde{a}(x,\xi) = \overline{a(x,-\xi)}.
\end{equation*}
If $ a $ is a real even function or pure imaginary odd function of $ \xi $, then $ \tilde{a} = a $. In particular $ \tilde{\gamma} = \gamma $, because $ \gamma^{(3/2)} $ and $ \Re\, \gamma^{(1/2)} $ are real and even functions of $ \xi $, while $ i\, \Im\, \gamma^{(1/2)} $ is a pure imaginary odd function of $ \xi $. Taking the complex conjugation of~\eqref{eq:Pseudo-Linear-Equation}, we have
\begin{equation*}
(D_t - \tilde{Q}) \bar{u} = 0,
\end{equation*}
with (using the identity $ \pi^2|_{\Z^d} = \pi|_{\Z^d} $)
\begin{equation*}
\tilde{Q} = \pi(D_x) \op(\gamma\pi) + \pi(D_x) \big( - V \cdot D_x + g\op(r^{-1}\pi) \Re + \op(r \cdot m_b) i \Im \big).
\end{equation*}
This suggest the consideration of the following system of equations,
\begin{equation}
\label{eq:equation-pseudo-diff-sys}
D_t \vec{w} + \A \vec{w} = 0,
\end{equation}
where $ \vec{w} = \binom{w_+}{w_-} \in \Ldottwo(\Td) \times \Ldottwo(\Td) $, and $ \A = \pi(D_x) \op(A\pi) $ with the symbol
\begin{equation}
\label{eq:def-symbol-A}
A 
= \gamma \begin{pmatrix} 1 & \phantom{*}0 \\ 0 & - 1 \end{pmatrix} 
+ V \cdot \xi
+ \frac{g}{2r} \begin{pmatrix} \phantom{*}1 & \phantom{*}1 \\ -1 & -1 \end{pmatrix}
+ \frac{r\cdot m_{\b}}{2} \begin{pmatrix} \phantom{*}1 & -1 \\ -1 & \phantom{*}1 \end{pmatrix}.
\end{equation}
Indeed,~$ \binom{u}{\bar{u}} $ satisfies~\eqref{eq:equation-pseudo-diff-sys} whenever~$ u $ satisfies~\eqref{eq:Pseudo-Linear-Equation}.
Then the $ \Ltwo $-observability of~\eqref{eq:Pseudo-Linear-Equation} is a simple consequence of that of~\eqref{eq:equation-pseudo-diff-sys}.

\begin{proposition}
\label{prop:Ltwo-observability-sys}
Suppose that~$ \omega $ satisfies the geometric control condition, $ s $ is sufficiently large, $ T > 0 $, $ \varepsilon_0 > 0 $, and $ \underline{u} \in \Cls^{1,s}(T,\varepsilon_0) $. Then for~$ \varepsilon_0 $ sufficiently small, the $ \Ltwo $-observability of~\eqref{eq:equation-pseudo-diff-sys} holds, that is, for all solution~$ \vec{w} $ with initial data in $ \Ldottwo(\Td) \times \Ldottwo(\Td) $,
\begin{equation}
\label{eq:observability-Ltwo-system}
\|\vec{w}(0)\|_{\Ltwo}^2 \lesssim \int_0^T \|\phiw \vec{e}\cdot\vec{w}\|_{\Ltwo}^2 \dt,
\end{equation}
where $ \vec{e} = \binom{1}{1} $ defines a linear map $ \vec{e} \cdot : \Ldottwo(\Td) \times \Ldottwo(\Td) \to \Ldottwo(\Td) \  \vec{w} \mapsto w^+ + w^- $.
\end{proposition}

The rest of the section is devoted to proving this proposition.

\subsection{Reduction to Semiclassical Equation}
\label{sec:reduction-to-semiclassical-system}

This section derives the equation satisfied by frequency localized quasi-modes of~\eqref{eq:equation-pseudo-diff-sys}. By the hypothesis that~$ \omega $ satisfies the geometric control condition, for fixed $ T > 0 $, there exists some $ \upsilon \in {}]0,\infty[ $ such that
\begin{equation}
\label{eq:condition-propagation-speed}
\frac{3}{2}\upsilon^{1/2} \cdot T > \min\{L \ge 0 : \forall (x,\xi) \in \Td \times \S^{d-1}, [x,x+L\xi] \cap \omega \ne \emptyset \},
\end{equation}
where $ [x,x+L\xi] = \{x+t\xi : t \in [0,L] \} $. Indeed the right hand side is finite by the geometric control condition and the compactness of~$ \Td $. The expression of the left hand side is in accordance with the group velocity~\eqref{eq:intro-dispersion-relation}, so that wave packets of the linearized equation around frequencies with modulus $ \ge \upsilon $ will travel into~$ \omega $ within time~$ T $.
We define a class of cutoff functions
\begin{equation*}
\Xi(\upsilon) = \big\{\chi \in \Ccinf(\R\backslash\{0\}) : 0 \le \chi \le 1, \ \supp\chi\subset\{1 \le |z|/\upsilon^{3/2} \le 5\},\ \chi|_{2\le |z|/\upsilon^{3/2}\le 4} \equiv 1 \big\}.
\end{equation*}
Fix $ \chi \in \Xi(\upsilon) $. Let $ \phi \in \Cinf(\R) $ be such that $ 0 \le \phi \le 1 $, $ \phi(z) = 1 $ for $ z \ge \upsilon^{3/2}/2 $, and $ \phi(z) = 0 $ for $ z \le \upsilon^{3/2}/4 $. In particular, $ \phi \chi = \chi $. Then set $ \varphi(\xi) = \phi(|\xi|^{3/2}) $. 

Now for~$ s $ sufficiently large, $ T > 0 $, $ \varepsilon_0 > 0 $, $ \udlu \in \Cls^{1,s}(T,\varepsilon_0) $, set
\begin{equation*}
h = 2^{-j}
\end{equation*}
as a semi\-classical parameter, and define the operator,
\begin{equation*}
\Zh = \op_h(\varphi) \op_h\big(\gamma_h \varphi\big), 
\qquad
\gamma_h = \gamma^{(3/2)} + h \gamma^{(1/2)},
\end{equation*}
so that $ \gamma_h(x,h\xi) = h^{3/2} \gamma_1(x,\xi) = h^{3/2} \gamma(x,\xi) $. For~$ \varepsilon_0 $ sufficiently small, $ \Zh $ is elliptic of order~$ 3/2 $, and the symmetric operator
\begin{equation*}
\Re \Zh \bydef \frac{1}{2} (\Zh + \Zh^*), \quad D(\Re \Zh) = \Cinf(\Td)
\end{equation*}
has a Friedrichs extension. We still denote by $ \Re \Zh $ this extension, and define
\begin{equation*}
\Delta_h = \op_h(\chi(\gamma^{(3/2)})),\quad \Pi_h = \chi(\Re \Zh)
\end{equation*}
where the latter is defined by the functional calculus of $ \Re \Zh $.Finally set 
\begin{equation}
\label{eq:definition-w-h}
\vec{w}_h = \Pi_h \Delta_h \vec{w}.
\end{equation}
We will derive the equation satisfied by $ \vec{w}_h $, under the coordinates $ (s,x) $, where
\begin{equation}
s = h^{-1/2} t,
\end{equation}
so that $ D_s = h^{1/2} D_t $. The equation~\eqref{eq:equation-pseudo-diff-sys} is equivalent to
\begin{equation}
\label{eq:equation-pseudo-diff-sys-(s,x)}
h D_s \vec{w} + h^{3/2} \A \vec{w} = 0.
\end{equation}
The equation for $ \vec{w}_h $ will be obtained by commutating successively $ \Delta_h $ and $ \Pi_h $ with~\eqref{eq:equation-pseudo-diff-sys-(s,x)}. A careful study of the operator $ \Pi_h $ is needed.

\subsubsection{Semiclassical Functional Calculus for $ \Pi_h $}

We use the notations $ [a,b]_s $ and $ [a,b]_t $ to denote the time intervals $ \{s : a \le s \le b\} $ and $ \{t : a \le t \le b\} $ to avoid ambiguity.
\begin{lemma}
\label{lem:semiclassical-functional-calculus}
Suppose that~$ s $ is sufficiently large, $ T > 0 $, $ \varepsilon_0 > 0 $, and $ \udlu \in \Cls^{0,s}(T,\varepsilon_0) $. Then for~$ \varepsilon_0 $ sufficiently small, $ \Delta_h \in C([0,T]_t,\L(\Ltwo,\Ltwo)) $ and
\begin{equation*}
\|\Pi_h - \Delta_h\|_{\Linf([0,T]_t,\L(\Ltwo,\Ltwo))} \lesssim h.
\end{equation*}
Moreover, if $ \udlu \in \Cls^{1,s}(T,\varepsilon_0) $, then for~$ \varepsilon_0 $ sufficiently small,
\begin{equation*}
\|D_s \Pi_h \|_{\Linf([0,T]_t,\L(\Ltwo,\Ltwo))} \lesssim (\varepsilon_0 + h) h^{1/2},
\end{equation*}
and in particular $ \Pi_h \in \Holder{1}([0,T]_t,\L(\Ltwo,\Ltwo)) $.
\end{lemma}
\begin{proof}
We omit the time variable~$ t $ for simplicity. The idea of the proof is to use Helffer-Sj\"{o}strand's formula (see for example~\cite{DS-Semiclassical-Limit,BGT-Strichartz,Zworski-Semiclassical-Analysis}),
\begin{equation}
\label{eq:formula-Helffer-Sjostrand}
\Pi_h = -\frac{1}{\pi}\int_{\C} \bar{\partial}\tilde{\chi}(z)(z-\Re \Zh)^{-1} \d z,
\end{equation}
where $ \tilde{\chi} $ is an almost analytic extension of $ \chi $, such that for some $ n \in \N $ to be fixed later,
\begin{equation*}
\tilde{\chi}(z) = \varrho(\Im\, z) \sum_{k=0}^n \frac{\chi^{(k)}(\Re\, z)}{k !} (i\Im\, z)^k,
\end{equation*}
with $ \varrho \in \Ccinf(\R) $ such that $ \varrho = 1 $ near zero. Notice that
\begin{equation}
\label{eq:estimate-almost-analytic}
|\bar{\partial} \tilde{\chi}(z)| \lesssim |\Im\, z|^{n+1}.
\end{equation}
In order to perform a semi\-classical parametrix construction for $ z-\Re\Zh $, we first determine the symbol for $ \Re \Zh $. By a semi\-classical symbolic calculus,
\begin{align}
\label{eq:symbolic-expansion-Z}
\Zh 
& = \op_h(\gamma^{(3/2)}\varphi^2) + h \op_h(\gamma^{(1/2)}\varphi^2) + h \op_h(\varphi \pxi \varphi \cdot D_x \gamma^{(3/2)}) + O(h^2)_{\L(\Ltwo,\Ltwo)} \nonumber \\
\Zh^*
& = \op_h(\gamma_h\varphi)^* \op_h(\varphi) \\
& = \big(\op_h(\gamma^{(3/2)}\varphi) + h\op_h(\overline{\gamma^{(1/2)}}\varphi) + h \op_h(\pxi \cdot D_x (\gamma^{(3/2)}\varphi)) \big) \op_h(\varphi) \nonumber \\
& \qquad \qquad + O(h^2)_{\L(\Ltwo,\Ltwo)} \nonumber \\
& = \op_h(\gamma^{(3/2)}\varphi^2) + h\op_h(\overline{\gamma^{(1/2)}}\varphi^2) + h \op_h(\varphi\pxi\varphi\cdot D_x \gamma^{(3/2)}) \nonumber  \\
& \qquad \qquad + h \op_h(\pxi\cdot D_x \gamma^{(3/2)}\varphi^2) + O(h^2)_{\L(\Ltwo,\Ltwo)}. \nonumber 
\end{align}
Consequently,
\begin{equation*}
\Re \Zh = \op_h(\gamma^{(3/2)}\varphi^2) + h \op_h(\zeta^{(1/2)}\varphi) + O(h^2)_{\L(\Ltwo,\Ltwo)},
\end{equation*}
where
\begin{equation*}
\zeta^{(1/2)} = \big( \frac{1}{2} \pxi\cdot D_x \gamma^{(3/2)} + \Re\,\gamma^{(1/2)}\big) \varphi^2 + \varphi\pxi\varphi\cdot D_x \gamma^{(3/2)}.
\end{equation*}
Now let $ q_0(z,x,\xi) = (z-\gamma^{(3/2)}\varphi^2)^{-1} $ for $ z \in \C\backslash\R $, then for some $ N > 0 $,
\begin{equation*}
\op_h(q_0(z)) \comp (z-\Re\Zh)
= 1 + O(h|\Im\, z|^{-N})_{\L(\Ltwo,\Ltwo)}.
\end{equation*}
Apply both sides to $ (z-\Re\Zh)^{-1} $, and use the resolvent estimate 
\begin{equation*}
\|(z-\Re\Zh)^{-1}\|_{\L(\Ltwo,\Ltwo)} \le |\Im\, z|^{-1},
\end{equation*}
we have,
\begin{equation*}
(z-\Re\Zh)^{-1} = \op_h(q_0(z)) + O(h|\Im z|^{-(N+1)})_{\L(\Ltwo,\Ltwo)}.
\end{equation*}
Plug this into~\eqref{eq:formula-Helffer-Sjostrand} with $ n \ge N $. Use Cauchy's integral formula and~\eqref{eq:estimate-almost-analytic},
\begin{align*}
\Pi_h 
& = -\frac{1}{\pi}\int_\C \bar{\partial}\tilde{\chi}(z) \big( \op_h\big(q_0(z)\big) + O\big(h|\Im z|^{-(N+1)}\big)_{\L(\Ltwo,\Ltwo)} \big) \d z \\
& = \op_h\big(-\frac{1}{\pi}\int_\C \bar{\partial}\tilde{\chi}(z) (z-\gamma^{(3/2)}\varphi^2)^{-1} \d z \big) \\
& \qquad \qquad -\frac{1}{\pi}\int_\C \bar{\partial}\tilde{\chi}(z) O\big(h|\Im z|^{-(N+1)}\big)_{\L(\Ltwo,\Ltwo)} \d z \\
& = \op_h\big(\chi(\gamma^{(3/2)}\varphi^2)\big) + O(h)_{\L(\Ltwo,\Ltwo)} \\
& = \Delta_h + O(h)_{\L(\Ltwo,\Ltwo)},
\end{align*}
where the last equality uses $ \chi(\gamma^{(3/2)}\varphi^2) = \chi(\gamma^{(3/2)}) $, for small~$ \varepsilon_0 $. 

As for $ D_s \Pi_h $, we apply the identity
\begin{equation*}
D_t (z-A)^{-1} = (z-A)^{-1}(D_t A)(z-A)^{-1}
\end{equation*}
to~\eqref{eq:formula-Helffer-Sjostrand} and deduce
\begin{equation*}
D_s \Pi_h = -\frac{h^{1/2}}{\pi} \int_{\C} \bar{\partial}\tilde{\chi}(z) (z-\Re\Zh)^{-1} (D_t\Re\Zh) (z-\Re\Zh)^{-1} \d z,
\end{equation*}
where $ D_t\Re\Zh = \Re (D_t \Zh) = \op_h(D_t \gamma^{(3/2)}\varphi^2) + h \op_h(D_t \zeta^{(1/2)}\varphi) + O(h^2)_{\L(\Ltwo,\Ltwo)} $. Therefore, use again Cauchy's integral formula,
\begin{equation}
\label{eq:expansion-D_s-Pi_h}
\begin{split}
D_s \Pi_h  =
& \op_h\Big(-\frac{h^{1/2}}{\pi}\int_{\C} \bar{\partial} \tilde{\chi}(z) D_t \gamma^{(3/2)} \varphi^2 (z-\gamma^{(3/2)} \phi^2)^{-2} \Big) \d z + O(h^{3/2})_{\L(\Ltwo,\Ltwo)} \\
& = -h^{1/2} \op_h\big(\partial_z\chi(\gamma^{(3/2)} \phi^2)D_t \gamma^{(3/2)} \varphi^2\big) + O(h^{3/2})_{\L(\Ltwo,\Ltwo)} \\
& = O((\varepsilon_0 + h) h^{1/2})_{\L(\Ltwo,\Ltwo)}.
\end{split}
\end{equation}
\end{proof}

\begin{corollary}
\label{cor:h-oscillation}
Suppose that~$ s $ is sufficiently large, $ T > 0 $, $ \varepsilon_0 > 0 $, and $ \udlu \in \Cls^{1,s}(T,\varepsilon_0) $. Let $ \chi'\in\Xi(\upsilon) $ be such that $ \chi'\chi=\chi $, then, for~$ \varepsilon_0 $ sufficiently small, 
\begin{equation}
\label{eq:h-oscillance-of-w(h)}
\vec{w}_h = \Delta'_h \vec{w}_h + R_{\Delta,h}\vec{w}_h,
\end{equation}
where $ \Delta'_h = \op_h\big(\chi'(\gamma^{(3/2)})\big) $, and
\begin{equation*}
\|R_{\Delta,h}\|_{C([0,T]_t,\L(\Ltwo,\Ltwo))} \lesssim h.
\end{equation*}
\end{corollary}
\begin{proof}
Let $ \Pi'_h = \chi'(\Re \Zh) $, then Lemma~\ref{lem:semiclassical-functional-calculus} shows that
\begin{equation*}
\Pi'_h = \Delta'_h + O(h)_{C([0,T]_t,\L(\Ltwo,\Ltwo))}.
\end{equation*}
We conclude with the functional calculus $ \Pi'_h \vec{w}_h = \Pi'_h \Pi_h \Delta_h \vec{w} = \Pi_h \Delta_h \vec{w} = \vec{w}_h $.
\end{proof}

\begin{corollary}
\label{cor:Littlewood-Paley-for-Pi-h-Delta-h}
Suppose that~$ s $ is sufficiently large, $ T > 0 $, $ \varepsilon_0 > 0 $, and $ \udlu \in \Cls^{1,s}(T,\varepsilon_0) $. Then for~$ \varepsilon_0 $ sufficiently small, for all $ t \in [0,T] $, $ \mathfrak{D}_h = \Pi_h, \ \Delta_h $ or $ \Pi_h \Delta_h $, where we omit the time variable~$ t $ in $ \mathfrak{D}_h = \mathfrak{D}_h(t) $, for $ u \in \Ltwo(\Td) $,
\begin{equation*}
\sum_{\substack{h = 2^{-j} \\ j \ge 0}} \|\mathfrak{D}_h u\|_{\Ltwo}^2
\lesssim \|u\|_{\Ltwo}^2.
\end{equation*}
Moreover, for $ j_0 > 0 $ sufficiently large, and all $ N > 0 $,
\begin{equation*}
\sum_{\substack{h = 2^{-j} \\ j \ge j_0}} \|\mathfrak{D}_h u\|_{\Ltwo}^2
\gtrsim \|u\|_{\Ltwo}^2 - \|u\|_{\H{-N}}^2.
\end{equation*}
\end{corollary}

\begin{proof}
We only prove the estimates for $ \mathfrak{D}_h = \Pi_h \Delta_h $. The rest is similar. By Lemma~\ref{lem:semiclassical-functional-calculus},
\begin{equation*}
\Pi_h \Delta_h 
= (\Delta_h + O(h)_{\L(\Ltwo,\Ltwo)}) \Delta_h
= \op_h(\chi(\gamma^{(3/2)})^2) + O(h)_{\L(\Ltwo,\Ltwo)}.
\end{equation*}
Let $ \chi',\chi'' \in \Xi(\upsilon) $ be such that $ \chi \chi' = \chi', \chi'' \chi = \chi $, then
\begin{align*}
\Pi_h\Delta_h
& = \Pi_h\Delta_h \chi''(|hD_x|^{3/2}) + \Pi_h \Delta_h (1-\chi''(|hD_x|^{3/2})) \\
& = O(1)_{\L(\Ltwo,\Ltwo)} \chi''(|hD_x|^{3/2}) + O(h)_{\L(\Ltwo,\Ltwo)}. \\
\chi'(|hD_x|^{3/2})
& =  \chi'(|hD_x|^{3/2}) \Pi_h\Delta_h +  \chi'(|hD_x|^{3/2}) (1 - \Pi_h\Delta_h) \\
& = O(1)_{\L(\Ltwo,\Ltwo)} \Pi_h\Delta_h + O(h)_{\L(\Ltwo,\Ltwo)}.
\end{align*}
Therefore, by Littlewood-Paley's theory,
\begin{align*}
\sum_{\substack{h = 2^{-j} \\ j \ge 0}} \|\Pi_h \Delta_h u\|_{\Ltwo}^2 
\lesssim \sum_{\substack{h = 2^{-j} \\ j \ge 0}} \| \chi''(|hD_x|^{3/2}) u\|_{\Ltwo}^2 + h^2 \|u\|_{\Ltwo}^2 
& \lesssim \|u\|_{\Ltwo}^2. \\
\sum_{\substack{h = 2^{-j} \\ j \ge j_0}} \|\Pi_h \Delta_h u\|_{\Ltwo}^2 + h^2 \|u\|_{\Ltwo}^2
\gtrsim \sum_{\substack{h = 2^{-j} \\ j \ge j_0}} \|\chi'(|hD_x|^{3/2}) u \|_{\Ltwo}^2 
& \gtrsim \|u\|_{\Ltwo}^2 - \|u\|_{\H{-N}}^2.
\end{align*}
We conclude for $ j_0 $ sufficiently large.
\end{proof}

\subsubsection{Equation for $ \Delta_h \vec{w} $}
\begin{proposition}
\label{prop:equation-Delta-h-w(h)}
Suppose that~$ s $ is sufficiently large, $ T > 0 $, $ \varepsilon_0 > 0 $, $ \udlu \in \Cls^{1,s}(T,\varepsilon_0) $. Then for~$ \varepsilon_0 $ sufficiently small, $ \Delta_h\vec{w} $ satisfies the equation
\begin{equation}
\label{eq:equation-Delta-h-w}
(h D_s + \A_h) \Delta_h \vec{w} + \tilde{R}_h \vec{w} = 0,
\end{equation}
where $ \A_h = h^{3/2} \op_h(\varphi) \A \op_h(\varphi) $, and
\begin{equation}
\label{eq:estimate-tilde-R-h}
\|\tilde{R}_h\|_{\Linf([0,T]_t,\L(\Ltwo,\Ltwo))} \lesssim (\varepsilon_0 + h^{1/2})h^{3/2}.
\end{equation}
\end{proposition}

\begin{remark}
\label{remark:A_h}
$ \A_h $ is of the form $ \A_h = \begin{pmatrix}
\Zh & 0 \\ 0 & -\Zh
\end{pmatrix} 
+{} \mathrm{small\ lower\ order\ terms}
$.
\end{remark}

\begin{proof}
The remainder $ \tilde{R}_h $ is explicitly written as follows,
\begin{equation*}
\begin{split}
\tilde{R}_h
& = [\Delta_h,h D_s] + h^{3/2} [\Delta_h,\A]  \\
& \qquad + h^{3/2} \A (1-\op_h(\varphi)) \Delta_h + h^{3/2} (1-\op_h(\varphi)) \A \op_h(\varphi) \Delta_h, \\
& = (\mathrm{I}) + (\mathrm{II}) + (\mathrm{III}) + (\mathrm{IV}).
\end{split}
\end{equation*}
Estimate the four terms respectively.

\textit{Estimate of }(I).
Use $ D_s = h^{1/2} D_t $, and bound $ D_t \gamma^{(3/2)} $ by $ \varepsilon_0 $,
\begin{equation*}
\begin{split}
(\mathrm{I}) 
& = - h^{3/2} \op_h(\partial_z\chi( \gamma^{(3/2)} )\cdot D_t \gamma^{(3/2)}) 
 = O(\varepsilon_0 h^{3/2})_{\L(\Ltwo,\Ltwo)}.
\end{split}
\end{equation*} 

\textit{Estimate of }(II).
The estimate of the commutator $ [\pi(D_x)\op(\gamma^{(3/2)}\pi),\Delta_h] $ is the main difficulty. By a symbolic calculus, and Remark~\ref{rmk:smallness-symbol},
\begin{equation*}
\begin{split}
\pi(D_x) \op(\gamma^{(3/2)}\pi)
& = \op(\gamma^{(3/2)}\pi) + \op(\pi \pxi\pi \cdot D_x \gamma^{(3/2)}) + O(\varepsilon_0)_{\L(\Ltwo,\Ltwo)} \\
& = \op(\gamma^{(3/2)}\pi) + O(\varepsilon_0)_{\L(\Ltwo,\Ltwo)},
\end{split}
\end{equation*}
where we use $ \pi \pxi \pi = 0 $ on $ \Z^d $, whence $\op(\pi \pxi\pi \cdot D_x \gamma^{(3/2)}) = 0$. Therefore,
\begin{equation*}
\begin{split}
[\Delta_h,\pi(D_x)\op(\gamma^{(3/2)}\pi)] 
& = [\Delta_h,\op(\gamma^{(3/2)}\pi)] + O(\varepsilon_0)_{\L(\Ltwo,\Ltwo)} \\
& = \frac{1}{i} \op \big( \{ \chi(h^{3/2}\gamma^{(3/2)}), \gamma^{(3/2)}\pi \} \big) + O(\varepsilon_0)_{\L(\Ltwo,\Ltwo)}.
\end{split} 
\end{equation*}
Observe that the Poisson bracket vanishes,
\begin{equation*}
\{ \chi(h^{3/2}\gamma^{(3/2)}), \gamma^{(3/2)}\pi \}
= \{ \chi(h^{3/2}\gamma^{(3/2)}), \gamma^{(3/2)} \} \pi
+ \{ \chi(h^{3/2}\gamma^{(3/2)}), \pi \} \gamma^{(3/2)} 
= 0.
\end{equation*}
Indeed, the first term vanishes because $ \chi(h^{3/2}\gamma^{(3/2)}) $ is a function of $ \gamma^{(3/2)} $. The singularity at $ 0 $ is not a problem when $ \varepsilon_0 $ is sufficiently small. The second term vanishes because the supports of $ \chi(h^{3/2}\gamma^{(3/2)}) $ and $ \pi $ are disjointed, for $ h $ and $ \varepsilon_0 $ sufficiently small. Combining the estimates of commutators of lower orders, $ (\mathrm{II}) = O((\varepsilon_0 + h^{1/2}) h^{3/2})_{\L(\Ltwo,\Ltwo)} $.

\textit{Estimate of }(III).
Write
\begin{equation*}
(\mathrm{III}) = \A \langle D_x \rangle^{-3/2} h^{3/2} \langle D_x \rangle^{3/2} (1-\op_h(\varphi)) \Delta_h.
\end{equation*}
Recall that $ \phi\chi = \chi $, so for $ \varepsilon_0 $ sufficiently small, 
$ (1-\op_h(\varphi)) \Delta_h = O(h^2)_{\L(\Ltwo,\H{2}_h)}$.
And $ h^{3/2} \langle D_x \rangle^{3/2} : \H{2}_h \to \H{1/2}_h \hookrightarrow \Ltwo $, $ \A \langle D_x \rangle^{-3/2} : \Ltwo \to \Ltwo $. Therefore, $ (\mathrm{III}) = O(h^2)_{\L(\Ltwo,\Ltwo)} $.

\textit{Estimate of (IV)}.
Observe that $ h^{3/2}\A \op_h(\varphi) $ is a semi\-classical pseudo\-differential operator with principle symbol $  \begin{pmatrix} \gamma^{(3/2)} & 0 \\ 0 & -\gamma^{(3/2)} \end{pmatrix} \varphi $. It suffices therefore to estimate $ (1-\op_h(\varphi))\op_h(\gamma^{(3/2)}\varphi)\Delta_h $. By the symbolic calculus, this is of order $ O(h^2)_{\L(\Ltwo,\H{2}_h)} $. Therefore, $ (\mathrm{IV}) = O(h^2)_{\L(\Ltwo,\Ltwo)} $.
\end{proof}

\subsubsection{Equation for $ \vec{w}_h = \Pi_h \Delta_h \vec{w} $}

\begin{proposition}
\label{prop:equation-w(h)}
Suppose that~$ s $ is sufficiently large, $ T > 0 $, $ \varepsilon_0 > 0 $, $ \udlu \in \Cls^{1,s}(T,\varepsilon_0) $. Then for~$ \varepsilon_0 $ sufficiently small, $ \vec{w}_h $ satisfies the equation
\begin{equation}
\label{eq:equation-w(h)}
(h D_s \vec{w}_h + \A_h) \vec{w}_h + R_h \vec{w} = 0,
\end{equation}
Let $ \Pi'_h = \chi'(\Re \Zh) $ with $ \chi'\in\Xi(\eta) $ and $ \chi'\chi=\chi $, then for some operators $ R_h^i $ ($ i=1,2,3 $) satisfying the estimates
\begin{equation}
\label{eq:estimate-R-h-i}
\|R_h^i\|_{\Linf([0,T]_t,\L(\Ltwo,\Ltwo))} \lesssim (\varepsilon_0 + h^{1/2}) h^{3/2},
\end{equation}
the following decomposition holds,
\begin{equation}
\label{eq:decomposition-R-h}
R_h = \Pi_h R_h^1 + R_h^2 \Pi'_h + R_h^3\Delta_h.
\end{equation}
\end{proposition}
\begin{proof}
Commuting~\eqref{eq:equation-Delta-h-w} with $ \Pi_h $, using the functional calculus $ \Pi_h = \Pi_h \Pi'_h $, the remainder $ R_h $ writes explicitly
\begin{equation*}
\begin{split}
R_h 
& = \Pi_h \tilde{R}_h + [\Pi_h,h D_s] + [\Pi_h,\A_h]\Delta_h \\
& = \Pi_h (\tilde{R}_h - hD_s \Pi'_h) + (-hD_s \Pi_h) \Pi'_h + [\Pi_h,\A_h]\Delta_h.
\end{split}
\end{equation*}
Therefore $ R^1_h = \tilde{R}_h - hD_s \Pi'_h $, $ R^2_h = -hD_s \Pi_h $, $ R^3_h = [\Pi_h,\A_h] $.
The estimates of $ R^1_h $ and $ R^2_h $ follow from Lemma~\ref{lem:semiclassical-functional-calculus} and~\eqref{eq:estimate-tilde-R-h}.
In view of Remark~\ref{remark:A_h}, the main estimate of $ R^3_h $ is that of the commutator $ [\Pi_h,\Zh] $. To do this, write $ \Zh = \Re\Zh + \frac{1}{2}(\Zh - \Zh^*) $, where by a functional calculus $ [\Pi_h,\Re \Zh] = 0 $, while by~\eqref{eq:symbolic-expansion-Z} and~\eqref{eq:gamma-almost-selfadjoint},
\begin{equation}
\label{eq:estimate-Z_h-Z_h}
\Zh-\Zh^*
= h\op\big( (2i\Im\gamma^{(1/2)}-\pxi\cdot D_x \gamma^{(3/2)}) \varphi^2 \big) + O(h^2)_{\L(\Ltwo,\Ltwo)}
= O(h^2)_{\L(\Ltwo,\Ltwo)}.
\end{equation}
Therefore $ [\Pi_h,\Zh] = O(h^2)_{\L(\Ltwo,\Ltwo)} $, and combining the estimates of commutators with lower order terms, $ (\mathrm{III}) = O((\varepsilon_0 + h^{1/2})h^{3/2})_{\L(\Ltwo,\Ltwo)} $.
\end{proof}

\begin{corollary}
\label{cor:littlewood-paley-R_h}
Suppose that~$ s $ is sufficiently large, $ T > 0 $, $ \varepsilon_0 > 0 $, $ \udlu \in \Cls^{1,s}(T,\varepsilon_0) $. Then for~$ \varepsilon_0 $ sufficiently small, for almost every $ t \in [0,T] $ and for all $ j_0 \in \N $,
\begin{equation*}
\sum_{\substack{h=2^{-j} \\ j \ge j_0}} \|h^{-3/2} R_h u\|_{\Ltwo}^2
\lesssim (\varepsilon_0^2 + 2^{-j_0}) \|u\|_{\Ltwo}^2.
\end{equation*}
\end{corollary}
\begin{proof}
We omit the time variable~$ t $ for simplicity. By Corollary~\ref{cor:Littlewood-Paley-for-Pi-h-Delta-h} and the estimates that $  h^{-3/2} R^i_h = O(\varepsilon_0 + h^{1/2})_{\L(\Ltwo,\Ltwo)} $,
\begin{equation*}
\begin{split}
\sum_{\substack{h=2^{-j} \\ j \ge j_0}} \|h^{-3/2} R^2_h \Pi'_h u\|_{\Ltwo}^2 
 + & \sum_{\substack{h=2^{-j} \\ j \ge j_0}} \|h^{-3/2} R^3_h \Delta_h u\|_{\Ltwo}^2 \\
\lesssim (\varepsilon_0^2 + 2^{-j_0}) 
& \Big( \sum_{\substack{h=2^{-j} \\ j \ge j_0}}  \|\Pi'_h u\|_{\Ltwo}^2
+ \sum_{\substack{h=2^{-j} \\ j \ge j_0}} \|\Delta_h u\|_{\Ltwo}^2   \Big)
\lesssim (\varepsilon_0^2 + 2^{-j_0}) \|u\|_{\Ltwo}^2.
\end{split}
\end{equation*}
The main difficulty is to estimate the square sum for $ \Pi_h R^1_h $. Recall that
\begin{equation*}
R^1_h = \tilde{R}_h - h D_s \Pi'_h.
\end{equation*}
By Lemma~\ref{lem:semiclassical-functional-calculus}, in particular~\eqref{eq:expansion-D_s-Pi_h}, and apply the same proof as Corollary~\ref{cor:Littlewood-Paley-for-Pi-h-Delta-h},
\begin{equation*}
\sum_{\substack{h=2^{-j} \\ j \ge j_0}} \|h^{-3/2} \Pi_h h D_s\Pi'_h u\|_{\Ltwo}^2
\lesssim \sum_{\substack{h=2^{-j} \\ j \ge j_0}} \|h^{-1/2} D_s \Pi'_h u\|_{\Ltwo}^2 
\lesssim (\varepsilon_0^2 + 2^{- j_0}) \|u\|_{\Ltwo}^2.
\end{equation*}
Indeed, we use the identity
\begin{equation*}
h^{-1/2} D_s \Pi'_h = - \op_h\big(\partial_z\chi'(\gamma^{(3/2)} \phi^2)D_t \gamma^{(3/2)} \varphi^2\big) + O(h)_{\L(\Ltwo,\Ltwo)},
\end{equation*}
and apply Littlewood-Paley's theory.

It remains to estimate $ \Pi_h \tilde{R}_h $. Use the decomposition for $ \tilde{R}_h $ (Proposition~\eqref{prop:equation-Delta-h-w(h)}), the terms (I),  (III)  and  (IV)  pose no problem because they each ends with $ \Delta_h $ or $ h D_s \Delta_h $. A similar argument as above works. However (II) should be treated with care. Write
\begin{equation*}
(\mathrm{II}) = (\mathrm{II}) \chi'(|hD_x|^{3/2}) + (\mathrm{II}) (1-\chi'(|hD_x|^{3/2})).
\end{equation*}
By $ \Pi_h (\mathrm{II}) = O((\varepsilon_0 + h^{1/2})h^{3/2})_{\L(\Ltwo,\Ltwo)} $, the estimate of $ \Pi_h (\mathrm{II}) \chi'(|h D_x|^{3/2}) $ is exactly the same as before, 
\begin{equation*}
\sum_{\substack{h=2^{-j} \\ j \ge j_0}} \|h^{-3/2} \Pi_h (\mathrm{II}) \chi''(|h D_x|^{3/2}) u\|_{\Ltwo}^2
\lesssim (\varepsilon_0^2 + 2^{-j_0}) \|u\|_{\Ltwo}^2.
\end{equation*}
As for the second term, separate it into two halves by expanding the commutator,
\begin{equation*}
(\mathrm{II}) (1-\chi'(|hD_x|^{3/2})) = h^{3/2} \Delta_h \A (1-\chi'(|hD_x|^{3/2})) - h^{3/2} \A \Delta_h (1-\chi'(|hD_x|^{3/2})),
\end{equation*}
and estimate them separately by semi\-classical symbolic calculus. It suffices to show that each half is of order $ O(h^{2})_{\L(\Ltwo,\Ltwo)} $. For the second half,
\begin{equation*}
\Delta_h (1-\chi'(|hD_x|^{3/2})) = O(h^2)_{\L(\Ltwo,\H{2}_h)} = O(h^2)_{\L(\Ltwo,\H{3/2}_h)} = O(h^{1/2})_{\L(\Ltwo,\H{3/2})}.
\end{equation*}
Therefore $ h^{3/2} \A \Delta_h (1-\chi'(|hD_x|^{3/2})) = O(h^{2})_{\L(\Ltwo,\Ltwo)} $. As for the first half, we only do the estimate for the principal term. However, the basic semi\-classical symbolic calculus cannot be directly applied here, for 
\begin{equation*}
h^{3/2}\op(\gamma \pi) (1-\chi'(|hD_x|^{3/2}))
= \op_h\big(\gamma_h(1-\chi'(|\xi|^{3/2}))\pi_h\big),
\end{equation*}
with $ \pi_h(\xi) = \pi(\xi/h) $, whose $ \xi $-derivative is unbounded as~$ h $ to~0. However, we have the following Lemma~\ref{lem:semiclassical-calculus-singularity-0}, which shows that no problems can be caused by the low frequency, where the estimate of the $ \xi $-derivatives of the symbols are not required. Therefore the first term is also of order $ O(h^2)_{\L(\Ltwo,\Ltwo)} $, as long as~$ s $ is sufficiently large.
\end{proof}

\begin{lemma}
\label{lem:semiclassical-calculus-singularity-0}
Let $ \varphi \in C(\Rd) \cap \Linf(\Rd) $ and let $ \{a_h\}_{h > 0} \subset \Symbol{m}{\rho} $ (see Appendix~\ref{sec:paradifferential-calculus}) be a family of symbols depending on~$ h $. Suppose that there exists $ 0 < C_1 < C_2 $, such that for~$ h $ sufficiently small,
\begin{equation*}
\supp \varphi \subset \{\xi:|\xi|>C_2\}, \quad
\supp a_h \subset \Td \times \{\xi:|\xi|<C_1 h^{-1}\}.
\end{equation*}
Then if~$ \rho = 2N $, with $ \N \ni N > d $,
\begin{equation*}
\|\op_h(\varphi) \op(a_h)\|_{\L(\Ltwo,\Ltwo)} \lesssim h^{\rho/2-m} \M{m}{\rho}{0}{a_h} \|u\|_{\Ltwo}.
\end{equation*}
\end{lemma}
\begin{proof}
Set $ A_h = \op_h(\varphi) \op(a_h) $, $ \Lambda^\mu = \langle D_x \rangle^\mu $ for $ \mu \in \R $. Let $ \widehat{a_h} $ denote the Fourier transform of~$ a_h $ with respect to the~$ x $. Then for $ \xi \in \Z^d \cap \supp \varphi(h\cdot) $,
\begin{align*}
\widehat{A_hu}(\xi)
= \varphi(h\xi) \widehat{\op(a_h) u}(\xi) 
& = (2\pi)^{-d} \sum_{\eta\in\Z^d} \varphi(h\xi) \widehat{a_h}(\xi-\eta,\eta) \hat{u}(\eta) \\
& = (2\pi)^{-d} \sum_{|\eta| < C_1 h^{-1}} \frac{1}{\jp{\xi-\eta}^{2N}} \varphi(h\xi) \widehat{\Lambda^{2N} a_h}(\xi-\eta,\eta) \hat{u}(\eta).
\end{align*}
Observe that by our hypothesis, $ |\xi - \eta| \gtrsim h^{-1} $ whenever $ \xi \in \sup \varphi(h\cdot) \subset \{\xi : |\xi| > C_2 h^{-1}\} $ and $ |\eta| < C_1 h^{-1} $, and use the fact that the Fourier transform defines a bounded operator from $ \Linf(\Td) \subset \Lone(\Td) $ to $ \ell^\infty(\Z^d) $, we have
\begin{align*}
|\widehat{Au}(\xi)|
& \lesssim h^N \sum_{|\eta| < C_1 h^{-1}} \frac{1}{\jp{\xi-\eta}^{N}} |\varphi(h\xi) \widehat{\Lambda^{2N} a_h}(\xi-\eta,\eta) \hat{u}(\eta)| \\
& \lesssim h^{N} \|\varphi\|_{\Linf} \sum_{|\eta| < C_1 h^{-1}} \|a_h(\cdot,\eta)/\jp{\eta}^m\|_{W^{2N,\infty}} \frac{\jp{\eta}^m}{\jp{\xi-\eta}^{N}} |\hat{u}(\eta)| \\
& \lesssim h^{N-m} \M{m}{2N}{0}{a_h} \sum_{\eta \in \Z^d} \frac{|\hat{u}(\eta)|}{\jp{\xi-\eta}^{N}}.
\end{align*}
Consequently, for $ \rho /2 = N > d $,
\begin{align*}
\|Au\|_{\Ltwo(\Td)} 
& \lesssim \|\widehat{Au}\|_{\ell^2(\Z^d)}
\lesssim h^{N-m} \M{m}{\rho}{0}{a_h} \|\jp{\cdot}^{-N} \ast |\hat{u}|\|_{\ell^2(\Z^d)} \\
& \lesssim h^{N-m} \M{m}{\rho}{0}{a_h} \|\jp{\cdot}^{-N}\|_{\ell^1(\Z^d)} \|\hat{u}\|_{\ell^2(\Z^d)}
\lesssim h^{N-m} \M{m}{\rho}{0}{a_h} \|u\|_{\Ltwo(\Td)}.
\end{align*}
\end{proof}

\subsection{Semiclassical Observability}

In this section, unless otherwise specified, the time variable will be $ s $, and thus we write $ [0,T] = [0,T]_s $ for simplicity. The purpose is to prove the following Proposition~\ref{prop:semiclassical-observability}, by carefully adapting~\cite{Gerard-Semiclassical-Measure} to our case.

\begin{remark}
We will deal with both the symbols defined on $ T^*\Td $ and $ T^*(\R\times\Td) $. To void ambiguity, the semiclassical quantification operator $ \bop_h $ will be applied to symbols defined on $ T^*(\R\times\Td) \simeq \R_s\times\Td_x\times\R_\sigma\times\Rd_\xi $, that is, for $ a = a(s,x,\sigma,\xi) $
\begin{equation*}
\bop_h(a) \bydef a(s,x,hD_s,hD_x);
\end{equation*}
while $ \op_h $ will only be applied to symbols defined on $ T^*\Td \simeq \Td_x\times\Rd_\xi$, that is, for $ a = a(x,\xi) $,
\begin{equation*}
\op_h(a) = a(x,hD_x).
\end{equation*}
\end{remark}

\begin{proposition}
\label{prop:semiclassical-observability}
Suppose that~$ \omega $ satisfies the geometric control condition, $ s $ is sufficiently large, $ T > 0 $, then for some $ \varepsilon_0 > 0 $ and $ h_0 > 0 $ sufficiently small, for all $ 0 < h < h_0 $ and for $ \udlu \in \Cls^{0,s}(h^{1/2}T,\varepsilon_0) $ (we only put $ h^{1/2}T $ here because we are dealing with the time variable $ s = h^{-1/2} t $), the solution $ \vec{w}_h = \binom{w^+_h}{w^-_h} \in C([0,T],\Ltwo(\Td)) $ to the equation
\begin{equation}
\label{eq:equation-semiclassical-w-f}
(hD_s + \A_h(\udlu)) \vec{w}_h = \vec{f}_h,
\end{equation}
where $ \vec{f}_h \in \Ltwo([0,T],\Ltwo(\Td)) $, satisfies a semi\-classical observability
\begin{equation}
\label{eq:observability-semiclassical-Ltwo-system}
\|\vec{w}_h(0)\|_{\Ltwo}^2 \lesssim \int_0^T \|\phiw \vec{e}\cdot\vec{w}_h\|_{\Ltwo}^2 \ds + h^{-2} \|\vec{f}_h\|_{\Ltwo([0,T],\Ltwo)}^2,
\end{equation}
provided that
\begin{equation}
\label{eq:h-oscillation}
\vec{w}_h = \Pi'_h(\udlu) \vec{w}_h + o(1)_{C([0,T],\L(\Ltwo,\Ltwo))}\vec{w}_h,
\end{equation}
where $ \Pi'_h(\udlu) $ is defined as in Corollary~\ref{cor:h-oscillation}.
\end{proposition}

We prove this proposition by contradiction. If~\eqref{eq:observability-semiclassical-Ltwo-system} is not true, then we can find sequences of $ \varepsilon_n > 0 $, $ h_n > 0 $, $ \udlu_n \in \Cls^{0,s}(T,\varepsilon_n) $, and
\begin{equation*}
\begin{split}
(\vec{w}_n,\vec{f}_n) \in C([0,T],\Ltwo(\Td)) \times \Ltwo([0,T],\Ltwo(\Td)),
\end{split}
\end{equation*}
such that $ \vec{w}_n $ satisfies the condition~\eqref{eq:h-oscillation},
\begin{equation*}
	(h_n D_s + \A_{n}) \vec{w}_{n} = \vec{f}_n,
\end{equation*}
where we denote $ \A_n = \A_{h_n}(\udlu_n) $. And as $ n \to \infty $, $ \varepsilon_n = o(1) $, $ h_n = o(1) $,
\begin{equation*}
\|\vec{w}_n(0)\|_{\Ltwo} = 1, \quad
\|\phiw \vec{e} \cdot \vec{w}_{n}\|_{\Ltwo([0,T],\Ltwo)} = o(1), \quad
\|\vec{f}_n\|_{\Ltwo([0,T],\Ltwo)} = o(h_n).
\end{equation*}

We perform an energy estimate to show that $ \{\vec{w}_n\}_n $ is bounded in $ \Ltwo([0,T],\Ltwo(\Td)) $ so that its space-time semi\-classical defect measure is well defined. Indeed, recall that $ \mathcal{Z}_{h_n}(\udlu_n) - \mathcal{Z}_{h_n}(\udlu_n)^* = O(h_n^{2})_{\Linf([0,T],\L(\Ltwo,\Ltwo))} $ (see~\eqref{eq:estimate-Z_h-Z_h}), combining with the estimates of the lower order terms,
\begin{equation}
\label{eq:A_n-is-almost-self-adjoint}
\A_n - \A_n^* = O((\varepsilon_n+h_n^{1/2}){h_n^{3/2}})_{\Linf([0,T],\L(\Ltwo,\Ltwo))}.
\end{equation}
Using only a weaker estimate $ \A_n - \A_n^* = O(h_n)_{\L(\Ltwo,\Ltwo)} $,
\begin{equation}
\label{eq:energy-estimate-eq-w-h}
\begin{split}
h_n \ps \|\vec{w}_{n}\|_{\Ltwo}^2
& = \big(i \big(\A_n - \A_n^*\big)\vec{w}_{n},\vec{w}_{n}\big)_{\Ltwo} + 2 \Re (\vec{f}_n,\vec{w}_{n})_{\Ltwo}\\
& \lesssim h_n \|\vec{w}_{n}\|_{\Ltwo}^2 + h_n^{-1} \|\vec{f}_n\|_{\Ltwo}^2.
\end{split}
\end{equation}
By Gronwall's inequality, and the hypothesis that $ \|\vec{f}_n\|_{\Ltwo([0,T],\Ltwo)} = o(h_n) $,
\begin{equation}
\label{eq:boundness-space-time-w(h)}
\begin{split}
\|\vec{w}_{n}\|_{C([0,T],\Ltwo)}^2 
\lesssim \|\vec{w}_{n}(0)\|_{\Ltwo}^2 + h_n^{-2} \|\vec{f}_n\|_{\Ltwo([0,T],\Ltwo)}^2 
\lesssim 1,
\end{split}
\end{equation}
proving that $ \{\vec{w}_{n}\}_n $ is bounded in $ C([0,T],\Ltwo(\Td)) $. Therefore, up to a subsequence, we may suppose that $ \{\vec{w}_{n}|_{]0,T[\times\Td}\}_n $ is pure, and let
\begin{equation*}
\mu = \mu(s,x,\sigma,\xi) \in M_{2\times 2}\big(\mathcal{M}(T^*( ]0,T[{} \times \Td),\C)\big)
\end{equation*}
be its space-time semi\-classical defect measure, such that
\begin{equation}
\label{eq:formula-semiclassical-measure}
\lim_{n\to\infty} \big(\bop_{h_n}(X) (\psi \vec{w}_n),\psi\vec{w}_n\big)_{\Ltwo(]0,T[,\Ltwo)}
= \int_{T^*(]0,T[ \times \Td)} \tr(X\d\mu),
\end{equation}
for all $ \psi \in \Ccinf(]0,T[,\R) $ and $ X \in M_{2\times 2} (\Ccinf(T^*(]0,T[\times\Td))) $, such that $ \psi X = X $.
It is classical that $ \mu $ is hermitian in the sense that, for all $ c \in \C^2 $, $ c^t \mu c $ is a non-negative Radon measure. It admits the following form,
\begin{equation*}
\mu = \begin{pmatrix} \mu_{+} & \mu_{*} \\ \overline{\mu_{*}} & \mu_{-} \end{pmatrix},
\end{equation*}
where $ \mu_\pm = \mu[w^\pm_n] \ge 0 $ are respectively semi\-classical defect measures of the pure sequences $ \{w^\pm_n|_{]0,T[\times\Td}\}_n $. And moreover,
\begin{equation}
\label{eq:mu-absolutely-continuous}
\mu_* \ll \sqrt{\mu_+  \mu_-}. 
\end{equation}

\begin{lemma}
The following properties hold.
\begin{enumerate}[nosep]
\item $ \supp \mu \subset \{|\xi| \ge \upsilon \} $,
\item $ \supp \mu_{\pm} \subset \Sigma^\pm = \{\sigma \pm |\xi|^{3/2} = 0\} $, 
\item $ \mu_* = 0 $, therefore $ \mu $ is diagonal.
\end{enumerate}
\end{lemma}

\begin{proof}
The first statement comes from~\eqref{eq:h-oscillation}. Indeed, let $ \psi \in \Ccinf(]0,T[,\R) $ and $ X \in M_{2\times 2} (\Ccinf(T^*(]0,T[\times\Td))) $ be such that $ \psi X = X $. Denote $ \Pi'_n = \Pi'_{h_n}(\udlu_n) $, and $ \Delta'_n = \Delta'_{h_n}(\udlu_n) $, then by Lemma~\ref{lem:semiclassical-functional-calculus}
\begin{equation*}
\Pi'_n
= \Delta'_n + o(1)_{\Linf([0,T],\L(\Ldottwo,\Ldottwo))}
= \op_{h_n}(\chi'(|\xi|^{(3/2)})) + o(1)_{\Linf([0,T],\L(\Ltwo,\Ltwo))}.
\end{equation*}
The second equality comes from the fact that $ \chi'(\gamma^{(3/2)}) - \chi'(|\xi|^{3/2}) $ is bounded by~$ \varepsilon_0 $.
Consequently, by our hypothesis
\begin{align*}
0 
& = \lim_{n\to\infty} (\bop_{h_n}(X) (1-\Pi'_{h_n}) \psi \vec{w}_n, \psi \vec{w}_n)_{\Ltwo(]0,T[,\Ltwo)} \\
& = \int_{T^*(]0,T[ \times \Td)} \tr\big(X(1-\chi'(|\xi|^{3/2}))\d\mu\big),
\end{align*}
implying that $ \supp \mu \subset \supp \chi'(|\xi|^{3/2}) \subset \{|\xi| \ge \upsilon\} $.

To prove the second statement, let $ a \in \Ccinf(T^*(]0,T[\times\Td)) $ be such that $ \psi a = a $, and set $ X^+ = \begin{pmatrix} a & 0 \\ 0 & 0 \end{pmatrix} $, and $ X^- = \begin{pmatrix} 0 & 0 \\ 0 & a \end{pmatrix} $. Then write
\begin{equation*}
(h_nD_s \pm \A_n^*) \bop_{h_n}(X^\pm)
 = (h_nD_s \pm \A_{h_n}(0)) \bop_{h_n}(X^\pm) 
\pm (\A_n^* - \A_{h_n}(0)) \bop_{h_n}(X^\pm).
\end{equation*}
For the first term,
\begin{equation*}
\A_{h_n}(0) = h_n^{3/2} \op_{h_n}(\varphi) \A(0) \op_{h_n}(\varphi),
\end{equation*}
where $ \A(0) = \A(\udlu = 0) $ whose exact expression will be explicitly written in~\eqref{eq:A(0)-formula}, such that the principal symbol of $ \A_{h_n}(0) $ is 
$ \varphi(\xi)^2
\begin{pmatrix} 
|\xi|^{3/2} & 0 \\
0 & - |\xi|^{3/2}
\end{pmatrix} $,
and, considering $ \psi $ as a multiplication operator,
\begin{align*}
(h_nD_s \pm \A_{h_n}(0)) \bop_{h_n}(X^\pm) \psi
& = \bop_{h_n} \bigg(
\begin{pmatrix} 
\sigma + \varphi^2 |\xi|^{3/2} & 0 \\
0 & \sigma - \varphi^2|\xi|^{3/2}
\end{pmatrix} 
X^\pm\bigg)  \psi \\
& \qquad \qquad + O(h_n)_{\L(\Ltwo([0,T],\Ltwo),\Ltwo([0,T],\Ltwo))}.
\end{align*}
While for the second term, by~\eqref{eq:A_n-is-almost-self-adjoint},
\begin{align*}
(\A_n^* - \A_{h_n}(0) )\langle h_nD_x\rangle^{-3/2} 
& = (\A_{h_n}(\udlu_n)^* - \A_{h_n}(\udlu_n))\langle h_nD_x\rangle^{-3/2}  \\
& \qquad + (\A_{h_n}(\udlu_n) - \A_{h_n}(0))\langle h_nD_x\rangle^{-3/2}  \\
& = O(\varepsilon_n + h_n^{1/2} )_{\Linf([0,T],\L(\Ltwo,\Ltwo)}.
\end{align*}
and $ \langle hD_x\rangle^{3/2}\op_{h_n}(X^\pm ) : \Ltwo([0,T],\Ltwo(\Td))\to\Ltwo([0,T],\Ltwo(\Td)) $. Therefore
\begin{equation*}
\pm (\A_n^* - \A_{h_n}(0)) \bop_{h_n}(X^\pm)
=O(\varepsilon_n)_{\L(\Ltwo([0,T],\Ltwo),\Ltwo([0,T],\Ltwo))}.
\end{equation*}
Consequently, since $ \varphi \equiv 1 $ on $ \{|\xi| \ge \upsilon \} \supset \supp \mu $,
\begin{equation}
\label{eq:mu-support-calcul-I}
\begin{split}
\lim_{n\to\infty} \big( (h_nD_s \pm \A_n^*) \bop_{h_n}(& X^\pm)  (\psi \vec{w}_n), \psi \vec{w}_n \big)_{\Ltwo([0,T],\Ltwo)} \\
& = \int_{T^*(]0,T[\times\Td)} (\sigma \pm |\xi|^{3/2} ) a(s,x,\sigma,\xi) \d\mu_\pm,
\end{split}
\end{equation}
On the other hand,
\begin{equation}
\label{eq:mu-support-calcul_II-1}
\begin{split}
\big( (h_nD_s \pm \A_n^*) \bop_{h_n}(& X^\pm) (\psi \vec{w}_n), \psi \vec{w}_n \big)_{\Ltwo([0,T],\Ltwo)} \\
& = \big(\bop_{h_n}(X^\pm) (\psi \vec{w}_n), (h_nD_s \pm \A_n) (\psi \vec{w}_n) \big)_{\Ltwo([0,T],\Ltwo)} \\
& = \big( \bop_{h_n}(X^\pm) (\psi \vec{w}_n), [h_n D_s, \psi] \vec{w}_n + \psi \vec{f}_n \big)_{\Ltwo([0,T],\Ltwo)},
\end{split}
\end{equation}
where, since the support of $ [h_n D_s,\psi] \vec{w}_n = (h_n D_s \psi) \vec{w}_n $ and $ \bop_{h_n}(X^\pm) \vec{w}_n $ are disjointed, and by the hypothesis that $ \vec{f}_n = o(h_n)_{\Ltwo([0,T],\Ltwo)} $,
\begin{equation}
\label{eq:mu-support-calcul_II-2}
\begin{split}
\big( \bop_{h_n}(X^\pm) (\psi \vec{w}_n), [h D_s, \psi] \vec{w}_n \big)_{\Ltwo([0,T],\Ltwo)} & = 0, \\
\big( \bop_{h_n}(X^\pm) (\psi \vec{w}_n), \psi \vec{f}_n \big)_{\Ltwo([0,T],\Ltwo)} & = o(h_n).
\end{split}
\end{equation}
Combining \eqref{eq:mu-support-calcul-I}, \eqref{eq:mu-support-calcul_II-1} and \eqref{eq:mu-support-calcul_II-2},
\begin{equation*}
\int_{T^*(]0,T[\times\Td)} (\sigma \pm |\xi|^{3/2} ) a(s,x,\sigma,\xi) \d\mu_\pm = 0,
\end{equation*}
which implies that $ \supp \mu_\pm \subset \{ \sigma \pm |\xi|^{3/2} = 0 \} $, and by the first statement,
\begin{equation*}
\supp \mu_\pm \subset \Sigma^\pm \cap \{|\xi| > \upsilon \}.
\end{equation*}
In particular $ \supp \mu_+ \cap \supp \mu_- = \emptyset $, so
\begin{equation}
\label{eq:orthogonality-mu-pm}
\mu_+ \perp \mu_-,
\end{equation}
whence by \eqref{eq:mu-absolutely-continuous}, $\mu_* \ll \sqrt{\mu_+\mu_-} = 0 $.
\end{proof}

The following corollary states that the formula~\eqref{eq:formula-semiclassical-measure} of the semi\-classical defect measure for the sequence $ \{\vec{w}_n\}_n $ remains valid for some symbols which are not of compact support in the $ \sigma $ variable. The main idea is that on the supports of $ \mu_\pm $, which are contained in the hypersurfaces $ \Sigma^\pm $, $ \sigma $ is automatically bounded whenever~$ \xi $ is bounded.
\begin{corollary}
\label{cor:extension-of-definition-of-mu}
Define the projection
\begin{equation*}
\kappa : T^*(]0,T[ \times \Td) \to {}]0,T[ \times T^*\Td, 
\quad (s,\sigma,x,\xi) \mapsto (s,x,\xi),
\end{equation*}
and its corresponding pullback
\begin{equation*}
\kappa^* : \Ccinf(]0,T[ \times T^*\Td) \to \Cinf(T^*(]0,T[ \times \Td)), 
\quad a \mapsto \kappa^* a = a \comp \kappa.
\end{equation*}
Let $ a \in \Ccinf(]0,T[ \times T^*\Td) $, and $ u \in \Ltwo([0,T], \Ltwo(\Td)) $, then
\begin{equation}
\label{eq:identity-a-kappa-a}
\op(a) u = \bop(\kappa^* a) u.
\end{equation}
Moreover, if $ a \in \Ccinf(]0,T[\times T^*\Td) $, and $ \psi \in \Ccinf(]0,T[) $, such that $ \psi a = a $, then
\begin{equation}
\label{eq:semiclassical-measure-non-compact-symbol}
\lim_{n\to\infty} \big( \op_{h_n}(a) (\psi w^\pm_{n}), \psi w^\pm_{n} \big)_{\Ltwo([0,T],\Ltwo)}
= \int_{\Sigma^\pm} \kappa^* a \d\mu_\pm.
\end{equation}
\end{corollary}
\begin{proof}
The identity~\eqref{eq:identity-a-kappa-a} follows by a direct verification of the definitions. To prove~\eqref{eq:semiclassical-measure-non-compact-symbol}, observe that $ w_n^\pm $ satisfies the equation
\begin{equation}
\label{eq:equation-w-pm-n}
(h_n D_s \pm \mathcal{Z}_{n} + h_n^{1/2}\op_{h_n}(V_n \cdot \xi \varphi^2)) w^\pm_{n} = o(h_n)_{\Ltwo([0,T],\Ltwo)}.
\end{equation}
Apply both sides by~$ \psi $, then $ \psi w_n^\pm $ satisfies
\begin{equation*}
\big(\bop_{h_n}(\sigma \pm \gamma^{(3/2)}\varphi^2) + O(h_n^{1/2}) \times \mathrm{lower\ order\ terms}\big) (\psi w^\pm_n) = o(1)_{\Ltwo(\R\times\Td)}.
\end{equation*}
Let $ \zeta \in \Ccinf(\R) $ be such that $ \zeta(z) = 1 $ in a neighborhood of $ 0 $, and apply both sides by $ \bop_{h_n}(\frac{1-\zeta(\sigma \pm \gamma^{(3/2)}\varphi^2)}{\sigma \pm \gamma^{(3/2)}\varphi^2}) $, we see that
\begin{equation}
\label{eq:estimate-outside-characteristic-manifold}
(1- \bop_{h_n}(b^\pm)) (\psi w^\pm_n) = o(1)_{\Ltwo(\R\times\Td)}.
\end{equation}
where $ b^\pm = \zeta(\sigma \pm \gamma^{(3/2)}\varphi^2)) $. Now let $ \theta \in \Ccinf(\R_s\times\Td_x\times\Rdxi) $ be such that $ \theta \equiv 1 $ on the support of $ a $, then by a symbolic calculus and~\eqref{eq:estimate-outside-characteristic-manifold},
\begin{equation*}
\begin{split}
\op_{h_n}(a) (\psi w^\pm_n) = \bop_{h_n}(\kappa^* a) (\psi w^\pm_n)
& = \bop_{h_n}(b^\pm \kappa^* a) (\psi w^\pm_n) + \bop_{h_n}((1- b^\pm)\kappa^* a) (\psi w^\pm_n) \\
& = \bop_{h_n}(\theta b^\pm \kappa^* a) (\psi w^\pm_n) + o(1)_{\Ltwo(\R\times\Td)}.
\end{split}
\end{equation*}
Notice that here the symbol $ \theta b^\pm \kappa^* a $ is of compact support in $ T^*(]0,T[\times\Td) $, and that $ \theta b^\pm \kappa^* a|_{\Sigma^\pm} = \kappa^* a|_{\Sigma^\pm} $. Therefore
\begin{equation*}
\begin{split}
\lim_{n\to\infty} \big( \op_{h_n}(a) (\psi w^\pm_{n}), \psi w^\pm_{n} \big)_{\Ltwo(]0,T[,\Ltwo)}
& = \lim_{n\to\infty} \big(\bop_{h_n}(\theta b^\pm \kappa^* a) (\psi w^\pm_{n}), \psi w^\pm_{n})_{\Ltwo(]0,T[,\Ltwo)} \\
& = \int_{\Sigma^\pm} \theta b^\pm \kappa^* a \,\d\mu_\pm
= \int_{\Sigma^\pm} \kappa^* a \, \d\mu_\pm.
\end{split}
\end{equation*}
\end{proof}

Now that by~\eqref{eq:boundness-space-time-w(h)}, $ \{\vec{w}_n\}_n $ is bounded in $ C([0,T],\Ltwo(\Td)) $, for all $ s \in [0,T] $, up to a subsequence, $ \{\vec{w}_{n}(s)\}_n $ is pure, and admits a semi\-classical measure. The next proposition proves that this subsequence can be so chosen that $ \{\vec{w}_{n}(s)\}_n $ is pure for all $ s \in [0,T] $.

\begin{proposition}
\label{prop:nu-is-well-defined}
Up to a subsequence, for all $ s \in [0,T] $, the sequence $ \{\vec{w}_{n}(s)\}_n $ is pure in $ \Ltwo(\Td) $, and thus admits a semi\-classical measure
\begin{equation*}
\nu_s = \nu_s(x,\xi) \in M_{2\times 2}\big(\mathcal{M}\big(T^*\Td,\C\big)\big),
\end{equation*}
which is positive definite and has the following form
\begin{equation*}
\nu_s = \begin{pmatrix} \nu_{s,+} & \nu_{s,*} \\ \overline{\nu_{s,*}} & \nu_{s,-} \end{pmatrix},
\end{equation*}
where $ \nu_{s,\pm} $ are semi\-classical measures of the sequences $ \{w^\pm_{n}(s)\}_n $, we denote this by $ \nu_{s,\pm} = \nu[w^\pm_n(s)] $.
Moreover, the distribution valued function $ \nu : s \to \nu_s $ is continuous in time, that is, it belongs to $ C([0,T],M_{2\times 2}(\mathcal{D}'(T^*\Td))) $.
\end{proposition}
\begin{proof}
By a diagonal argument, we assume that $ \{\vec{w}_{n}(s)\}_n $ is pure for each $ s $ in a countable and dense subset of $ [0,T] $, for example $ \mathbb{Q} \cap [0,T] $. In order to conclude, it suffices to show that for each $ X \in \Ccinf(T^*\Td) $, the following family of functions is equicontinuous,
\begin{equation*}
g^X_n : [0,T] \ni s \mapsto (\op_{h_n}(X) \vec{w}_{n}(s),\vec{w}_{n}(s))_{\Ltwo}.
\end{equation*}
We prove this by showing that $ \{\ps g^X_n\}_n $ is bounded in $ \Lone([0,T]) $.
\begin{align*}
h_n \ps g^X_n(s)
& = \big(i(\A_n^* \op_{h_n}(X) - \op_{h_n}(X) \A_n ) \vec{w}_{n}(s), \vec{w}_{n}(s)\big)_{\Ltwo} \\
& \qquad \qquad + 2 \Re\big(i(\op_{h_n}(X)+\op_{h_n}(X)^*)\vec{f}_n(s),\vec{w}_n(s)\big)_{\Ltwo} \\
& \lesssim h_n \|\vec{w}_{n}(s)\|_{\Ltwo}^2 + h_n^{-1} \|\vec{f}_{n}(s)\|_{\Ltwo}^2,
\end{align*}
where we use the estimate
\begin{equation*}
\begin{split}
\A_n^* \op_{h_n}(X) - \op_{h_n}(X) \A_n 
& = [\A_n,\op_{h_n}(X)] + (\A_n^*-\A_n) \op_{h_n}(X) \\
& = O(h_n)_{\Linf([0,T],\L(\Ltwo,\Ltwo))}.
\end{split}
\end{equation*}
Therefore, take the integration over $ [0,T] $,
\begin{equation*}
\|\ps g^X_n\|_{\Lone([0,T])}
\lesssim \|\vec{w}_{n}\|_{\Ltwo([0,T],\Ltwo)}^2 + h_n^{-2} \|\vec{f}_{n}\|_{\Ltwo([0,T],\Ltwo)}^2
\lesssim 1.
\end{equation*}
\end{proof}

Now consider the following distribution valued continuous functions,
\begin{equation*}
\nu_\pm : [0,T] \to \mathcal{D}'(T^*\Td), \quad s \mapsto \nu_{s,\pm}.
\end{equation*}
Then $ \nu_\pm $ define two distributions on $ ]0,T[\times T^*\Td $ such that, for $ \phi \in \Ccinf(]0,T[ \times T^*\Td) $,
\begin{equation*}
\langle \nu_\pm, \phi \rangle_{\D',\D(]0,T[\times T^*\Td)} 
\bydef \int_0^T \int_{T^*\Td} \phi(s,x,\xi) \d\nu_{s,\pm}(x,\xi) \ds.
\end{equation*}

\begin{proposition}
\label{prop:properties-of-mu}
The following properties of $ \nu_\pm $ holds.
\begin{enumerate}[nosep]
\item $ \supp \nu_\pm \subset \{|\xi| \ge \upsilon \} $.
\item Consider the two sections of $ \kappa $, 
\begin{equation*}
\zeta^\pm :{} ]0,T[{} \times T^*\Td \to \Sigma^\pm \subset T^*(]0,T[\times\Td), 
\quad (s,x,\xi) \mapsto (s,x,\mp |\xi|^{-3/2},\xi).
\end{equation*} 
Then for $ \phi = \phi_1 \otimes \phi_2 $ with $ \phi_1 \in \Ccinf(]0,T[_s) $ and $ \phi_2 \in \Ccinf(\R_\sigma) \otimes \Ccinf(T^*\Td) $,
\begin{equation}
\label{eq:relation-mu-nu}
\langle \mu_\pm, \phi \rangle_{\D',\D(T^*(]0,T[{}\times\Td))} = \langle \nu_\pm, (\zeta^\pm)^*\phi  \rangle_{\D',\D(]0,T[{}\times T^*\Td)},
\end{equation}
where $ (\zeta^\pm)^* \phi = \phi(\zeta^\pm) $; or equivalently,
\begin{equation*}
\int_{\Sigma^\pm} \phi(s,x,\sigma,\xi) \d\mu_\pm(s,x,\sigma,\xi)
= \int_0^T \int_{T^*\Td} \phi(s,x,\mp|\xi|^{3/2},\xi) \d\nu_{s,\pm}(x,\xi) \ds.
\end{equation*}
\item $ \nu_\pm $ is propagated via the following transportation equation,
\begin{equation}
\label{eq:propagation-of-nu}
\big( \ps \pm \frac{3}{2}|\xi|^{-1/2} \xi \cdot \nabla_x \big) \nu_\pm = 0.
\end{equation}
\end{enumerate}
\end{proposition}
\begin{proof}
The first statement is by the same reason as that for $ \mu $ (see Proposition~\ref{prop:properties-of-mu}).
To prove the second, use the uniform convergence following from the equicontinuity of $ g^{X}_n $ (here $ X = \phi_2\comp\zeta^\pm $) proved in Proposition~\ref{prop:nu-is-well-defined}, and Corollary~\ref{cor:extension-of-definition-of-mu}. Let $ \psi \in \Ccinf(]0,T[) $ be such that $ \psi \phi = \psi $, then
\begin{align*}
\int_0^T\int_{T^*\Td} (\zeta^\pm)^*\phi \d\nu_{s,\pm} \ds
& = \lim_{n\to\infty} \int_0^T (\op_{h_n}((\zeta^\pm)^*\phi) \psi w^\pm_{n},\psi w^\pm_{n})_{\Ltwo}|_s \ds \\
& = \int_{\Sigma^\pm} \kappa^*(\zeta^\pm)^*\phi \d\mu_\pm
= \int_{\Sigma^\pm} \phi \d\mu_\pm,
\end{align*}
where the last equality is due to $ \kappa^*(\zeta^\pm)^*\phi |_{\Sigma^\pm} = \phi |_{\Sigma^\pm} $.

To prove the propagation property, we use the equation~\eqref{eq:equation-w-pm-n} satisfied by $ w^\pm_{n} $, and omit the factor~$ \psi $ above for simplicity,
\begin{align*}
\langle \ps \nu_\pm, \phi \rangle_{\D',\D}
& = - \langle \nu_\pm, \ps\phi \rangle_{\D',\D} 
= -  \int_0^T \lim_{n\to\infty} \big(\op_{h_n}(\ps\phi) w^\pm_{n},w^\pm_{n}\big)_{\Ltwo} \ds\\
& = -  \int_0^T \lim_{n\to\infty} \big(\ps (\op_{h_n}(\phi) w^\pm_{n})-\op_{h_n}(\phi) \ps w^\pm_{n}, w^\pm_{n}\big)_{\Ltwo} \ds \\
& =  \int_0^T \lim_{n\to\infty} \big(\op_{h_n}(\phi) w^\pm_{n}, \ps w^\pm_{n}\big)_{\Ltwo} + \big(\op_{h_n}(\phi) \ps w^\pm_{n}, w^\pm_{n}\big)_{\Ltwo} \ds \\
& =   \int_0^T \mp   \lim_{n\to\infty} \frac{1}{h_n} \big(\op_{h_n}(\phi) w^\pm_{n}, i \mathcal{Z}_n w^\pm_{n}\big)_{\Ltwo} + \big(\op_{h_n}(\phi) i \mathcal{Z}_n w^\pm_{n}, w^\pm_{n}\big)_{\Ltwo} \ds \\
& = \int_0^T \pm \lim_{n\to\infty} \frac{i}{h_n} \big((\mathcal{Z}_n^* \op_{h_n}(\phi) - \op_{h_n}(\phi) \mathcal{Z}_n) w^\pm_{n}, w^\pm_{n}\big)_{\Ltwo} \ds.
\end{align*}
To continue, we use the explicit calculus of~$ \Zh $ and~$ \Zh^* $,
\begin{align*}
\mathcal{Z}_n^* \op_{h_n}(\phi) - \op_{h_n}(\phi) \mathcal{Z}_n
& = [\mathcal{Z}_n,\op_{h_n}(\phi)] + (\mathcal{Z}_n^* - \mathcal{Z}_n) \op_{h_n}(\phi) \\
& = \frac{h_n}{i} \op_{h_n} \{\gamma^{(3/2)} \varphi^2, \phi\} + O(h_n^2)_{\Linf([0,T],\L(\Ltwo,\Ltwo))}.
\end{align*}
Plug this into the limit above, and use the fact that $ \varphi \equiv 1 $ on a neighborhood of $ \supp \nu_\pm $, whence $ \{|\xi|^{3/2}\varphi^2,\phi \} = \{|\xi|^{3/2}, \phi\} $ on $ \supp \nu_\pm $,
\begin{align*}
\langle \ps \nu_\pm, \phi \rangle_{\D',\D}
& = \pm \langle \nu_\pm, \{|\xi|^{3/2}\varphi^2,\phi \} \rangle _{\D',\D} \\
& = \pm \langle \nu_\pm, \{|\xi|^{3/2}, \phi\} \rangle_{\D',\D}
= \mp \langle \{|\xi|^{3/2},\nu_\pm\}, \phi \rangle_{\D',\D}.
\end{align*}
Therefore, $  \ps \nu_\pm \pm \{|\xi|^{3/2},\nu_\pm\}  = 0 $ in the sense of distribution, which is the desired transportation equation for $ \nu_\pm $.
\end{proof}

We continue with the proof of the semi\-classical observability.
\begin{proof}[Proof of Proposition~\ref{prop:semiclassical-observability}]
By the hypothesis that $ \|\phiw \vec{e} \cdot \vec{w}_{n}\|_{\Ltwo([0,T],\Ltwo)} = o(1) $, the semi\-classical measure for the sequence $ \vec{e} \cdot \vec{w}_{n} = w^+_{n} + w^-_{n} $ (which we denote by $ \mu[\vec{e} \cdot \vec{w}_{n}] $) vanishes on $ T^*(]0,T[ \times \omega) $. Combining this with the orthogonality~\eqref{eq:orthogonality-mu-pm},
\begin{equation*}
0 = \mu[\vec{e} \cdot \vec{w}_{n}] |_{T^*(]0,T[ \times \omega)} 
= \mu_+|_{T^*(]0,T[ \times \omega)} + \mu_-|_{T^*(]0,T[ \times \omega)},
\end{equation*}
which implies $ \mu_\pm|_{T^*(]0,T[ \times \omega)} = 0 $ because $ \mu_\pm \ge 0 $.
Then by the identity~\eqref{eq:relation-mu-nu},
\begin{equation*}
\nu_\pm |_{]0,T[ \times T^*\omega} = 0.
\end{equation*}
By the geometric control condition, the propagation law~\eqref{eq:propagation-of-nu}, and the condition for propagation speed~\eqref{eq:condition-propagation-speed},
\begin{equation*}
\nu_\pm(0) = \nu_{0,\pm} = 0.
\end{equation*}
We then conclude by contradiction. Since by~\eqref{eq:h-oscillation}, $ \{\vec{w}_{n}(0)\}_n $ is $ h_n $-oscillating, and by hypothesis $ \|\vec{w}_n(0)\|_{\Ltwo} = 1 $, whence
\begin{equation*}
\int_{T^*\Td} \chi'(\gamma^{(3/2)}\varphi^2)|_{s=0} \, \tr\d\nu_0
= \lim_{n\to\infty} (\Pi'_{n} \vec{w}_{n}, \vec{w}_{n})_{\Ltwo}|_{s=0} 
= \lim_{n\to\infty} (\vec{w}_{n}, \vec{w}_{n})_{\Ltwo}|_{s=0} 
= 1,
\end{equation*}
and $ \nu_{0,+} + \nu_{0,-} = \tr \, \nu_0  \ne 0 $.
\end{proof}

\begin{corollary}
Suppose that~$ \omega $ satisfies the geometric control condition, $ s $ is sufficiently large, $ T > 0 $, $ \varepsilon_0 > 0 $, and $ \udlu \in \Cls^{1,s}(T,\varepsilon_0) $. Let~$ \vec{w} $ satisfy~\eqref{eq:equation-pseudo-diff-sys}, and let $ \vec{w}_h $ be defined by~\eqref{eq:definition-w-h}. Then for~$ \varepsilon_0 $ and $ h_0 > 0 $ sufficiently small, and all $ 0 < h = 2^{-j} < h_0 $ with $ j \in 2\N $, the following observability holds for all $ k = 0, 1, \ldots, h^{-1/2}-1 $,
\begin{equation}
\label{eq:observability-on-I-k-h}
h^{1/2}\|\vec{w}_{h}(t = kh^{1/2}T)\|_{\Ltwo}^2 \lesssim \int_{I^k_h} \|\phiw \vec{e} \cdot \vec{w}_{h}(t)\|_{\Ltwo}^2 \dt + h^{-2}\|R_h \vec{w}\|_{\Ltwo_t(I^k_h,\Ltwo)}^2,
\end{equation}
where $ I^k_h \bydef [kh^{1/2}T,(k+1)h^{1/2}T]_t $. 
\end{corollary}
\begin{proof}
The condition~\eqref{eq:h-oscillation} of Proposition~\ref{prop:semiclassical-observability} is verified by~\eqref{eq:h-oscillance-of-w(h)}, therefore, by Proposition~\ref{prop:semiclassical-observability}, for some $ \varepsilon_0 > 0 $ and $ h_0 > 0 $, uniformly for $ k = 0, 1, \ldots, h^{-1/2}-1 $,
\begin{equation*}
\|\vec{w}_{h}(s = kT)\|_{\Ltwo}^2 \lesssim \int_{kT}^{(k+1)T} \|\phiw \vec{e} \cdot \vec{w}_{h}(s)\|_{\Ltwo}^2 \ds + h^{-2}\|R_h \vec{w}\|_{\Ltwo_s([k T,(k+1)T],\Ltwo)}^2.
\end{equation*}
We conclude by $ \|\cdot\|_{\Ltwo([a,b]_s,\Ltwo)}^2 = h^{-1/2} \|\cdot\|_{\Ltwo([h^{1/2}a,h^{1/2}b]_t,\Ltwo)}^2 $, following from the change of time variable, $ s = h^{-1/2} t $.
\end{proof}

\subsection{Weak Observability}

\begin{proposition}
\label{prop:weak-observability}
Suppose that~$ \omega $ satisfies the geometric control condition, $ s $ is sufficiently large, $ T > 0 $, $ \varepsilon_0 > 0 $, and $ \udlu \in \Cls^{1,s}(T,\varepsilon_0) $. Then for~$ \varepsilon_0 > 0 $ sufficiently small, and for any $ N > 0 $, the following weak observability of~\eqref{eq:equation-pseudo-diff-sys} holds, that for all its solution~$ \vec{w} $ with $ \Ltwo $ initial data,
\begin{equation}
\label{eq:weak-observability-Ltwo}
\|\vec{w}(0)\|_{\Ltwo}^2 \lesssim \int_0^T \|\phiw\vec{e}\cdot\vec{w}\|_{\Ltwo}^2 \dt + \|\vec{w}(0)\|_{\H{-N}}^2.
\end{equation}
\end{proposition}

\begin{proof}
The idea is to first sum up~\eqref{eq:observability-on-I-k-h} to obtain an observability for $ \vec{w}_h $ on the whole interval $ [0,T] $, and then use Littlewood-Paley's theory to conclude.
However, the left hand side of~\eqref{eq:observability-on-I-k-h} is $ \|\vec{w}_{h}(t = kh^{1/2}T)\|_{\Ltwo}^2 $, an energy estimate of the equation~\eqref{eq:equation-w(h)} should be performed to bound it from below by $ \|\vec{w}_{h}(0)\|_{\Ltwo}^2 $. As in~\eqref{eq:energy-estimate-eq-w-h}, but we use the estimate $ \A_h - \A_h^* = O(h^{3/2})_{\L(\Ltwo,\Ltwo)} $ (the size~$ h^{3/2} $ is crucial here to obtain a uniform energy estimate, independent of~$ h $, on the interval $ [0,T]_t=[0,h^{-1/2}T]_s $),
\begin{equation*}
\begin{split}
h \ps \|\vec{w}_h\|_{\Ltwo}^2
& = (i(\A_h - \A_h^*)\vec{w}_h,\vec{w}_h)_{\Ltwo} + 2 \Re (R_h \vec{w}, \vec{w}_h)_{\Ltwo} \\
& \lesssim h^{3/2} \|\vec{w}_h\|_{\Ltwo}^2 + h^{-3/2} \|R_h \vec{w}\|_{\Ltwo}^2.
\end{split}
\end{equation*}
Therefore, for some constant $ C > 0 $,
\begin{equation*}
\ps \|\vec{w}_h\|_{\Ltwo}^2 - C h^{1/2} \|\vec{w}_h\|_{\Ltwo}^2
\lesssim h^{-5/2} \|R_h \vec{w}\|_{\Ltwo}^2,
\end{equation*}
and consequently,
\begin{equation*}
\ps (e^{-Ch^{1/2}s} \|\vec{w}_h\|_{\Ltwo}^2 )
= e^{-Ch^{1/2}s} (\ps \|\vec{w}_h\|_{\Ltwo}^2 - C h^{1/2} \|\vec{w}_h\|_{\Ltwo}^2) 
\lesssim e^{-Ch^{1/2}s} h^{-5/2} \|R_h \vec{w}\|_{\Ltwo}^2.
\end{equation*}
Now that $ h^{1/2}s $ is bounded for $ s \in [0,h^{-1/2}T] $, by Newton-Leibniz's rule,
\begin{equation*}
\|\vec{w}_h(0)\|_{\Ltwo}^2
\lesssim \|\vec{w}_h(s)\|_{\Ltwo}^2 + h^{-5/2}\|R_h \vec{w}\|_{\Ltwo([0,h^{-1/2}T]_s,\Ltwo)}^2.
\end{equation*}
Or equivalently, for $ t \in [0,T] $,
\begin{equation*}
h^{1/2}\|\vec{w}_h(0)\|_{\Ltwo}^2
\lesssim h^{1/2}\|\vec{w}_h(t)\|_{\Ltwo}^2 + h^{-5/2}\|R_h \vec{w}\|_{\Ltwo_t([0,T]_t,\Ltwo)}^2.
\end{equation*}
Set $ t = kh^{1/2}T $ for $ h = 2^{-j} $ with $ j \in 2\N $ sufficiently large, and $ k = 0,1,\ldots,h^{-1/2}-1 $, and use~\eqref{eq:observability-on-I-k-h} by absorbing $ h^{-2}\|R_h \vec{w}\|_{\Ltwo_t(I^k_h,\Ltwo)}^2 $ into $ h^{-5/2}\|R_h \vec{w}\|_{\Ltwo_t([0,T]_t,\Ltwo)}^2 $, we have
\begin{equation*}
\begin{split}
h^{1/2}\|\vec{w}_h(0)\|_{\Ltwo}^2
\lesssim \int_{I^k_h} \|\phiw \vec{e} \cdot \vec{w}_{h}(t)\|_{\Ltwo}^2 \dt + h^{-5/2}\|R_h \vec{w}\|_{\Ltwo_t([0,T],\Ltwo)}^2.
\end{split}
\end{equation*}
Sum up for $ k = 0, 1, \ldots, h^{-1/2} - 1 $,
\begin{equation}
\label{eq:observability-[0,T]-semiclassical}
\|\vec{w}_h(0)\|_{\Ltwo}^2 
\lesssim \int_0^T \|\phiw \vec{e}\cdot \vec{w}_{h}(t)\|_{\Ltwo}^2 \dt + h^{-3}\|R_h \vec{w}\|_{\Ltwo_t([0,T],\Ltwo)}^2.
\end{equation}
For the integrand, write
$\phiw \vec{e} \cdot \vec{w}_{h}
= \Pi_h \Delta_h \phiw \vec{e} \cdot \vec{w} + O(h)_{\L(\Ltwo,\Ltwo)}\vec{w}$.
Then sum up for $ h = 2^{-j} < h_0 = 2^{-j_0} $ sufficiently small, using Littlewood-Paley's theory for $ \Pi_h\Delta_h $ and $ R_h $ developed respectively by Corollary~\ref{cor:Littlewood-Paley-for-Pi-h-Delta-h} and Corollary~\ref{cor:littlewood-paley-R_h}, we obtain
\begin{equation*}
\|\vec{w}(0)\|_{\Ltwo}^2 
\lesssim \int_0^T \|\phiw \vec{e} \cdot \vec{w}(t)\|_{\Ltwo}^2 \dt
+ (\varepsilon_0^2 + 2^{-2j_0}) \|\vec{w}\|_{\Ltwo([0,T],\Ltwo)}^2 + \|\vec{w}(0)\|_{\H{-N}}^2 .
\end{equation*}
It suffices to bound $ \|\vec{w}\|_{\Ltwo([0,T],\Ltwo)} \lesssim  \|\vec{w}(0)\|_{\Ltwo}^2 $ and then absorb $ (\varepsilon_0^2 + 2^{-2j_0}) \|\vec{w}(0)\|_{\Ltwo}^2 $ into the left hand side.
\end{proof}

\subsection{Unique Continuation and Strong Observability}

We remove the remainder in~\eqref{eq:weak-observability-Ltwo} by the uniqueness-compactness argument to finish the proof of Proposition~\ref{prop:Ltwo-observability-sys}.

\begin{proof}[Proof of Proposition~\ref{prop:Ltwo-observability-sys}]
We proceed by contradiction. Suppose that the strong observability does not hold, then there exists a sequence $ \{\varepsilon_n,\udlu_n,\vec{w}_n\}_{n} $, with $ \varepsilon_n > 0 $, $ \udlu_n \in \Cls^{1,s}(T,\varepsilon_n) $, and $ \vec{w}_n \in C([0,T],\Ldottwo(\Td)) $ satisfying the equation
\begin{equation}
\label{eq:equation-w_n-uniqueness-compactness}
(D_t + \A(\udlu_n)) \vec{w}_n = 0,
\end{equation}
such that, as $ n \to \infty $,
\begin{equation*}
\varepsilon_n = o(1), 
\quad \|\vec{w}_n(0)\|_{\Ltwo}  = 1, 
\quad \int_0^T \|\phiw \vec{e}\cdot\vec{w}_n\|_{\Ltwo}^2 \dt  = o(1).
\end{equation*}
By an energy estimate, $ \{\vec{w}_n\}_n $ is bounded in $ C([0,T],\Ldottwo(\Td)) $, $ \{\pt \vec{w}_n \}_n $ is bounded in $ \Linf([0,T],\Hdot{-3/2}(\Td)) $. Therefore, by Arzelà-Ascoli's theorem, we may assume that, up to a subsequence, 
\begin{itemize}[noitemsep]
\item $ \vec{w}_n \to \vec{w} $ strongly in $ C([0,T],\Hdot{-3/2}(\Td)) $,
\item $ \vec{w}_n \rightharpoonup \vec{w} $ weakly in $ \Ltwo([0,T],\Ldottwo(\Td)) $,
\item $ \vec{w}_n(0) \rightharpoonup \vec{w}(0) $ weakly in $ \Ldottwo(\Td) $.
\end{itemize}
Now that $ \underline{u}_n \to 0 $ in $ C([0,T],\Hdot{s}(\Td)) $, we also have $ \A(\udlu_n) \vec{w}_n \to \A(0) \vec{w} $ strongly in $ C([0,T],\Hdot{-3}(\Td)) $. Therefore, passing to the limit $ n\to\infty $ of~\eqref{eq:equation-w_n-uniqueness-compactness} in the sense of distribution, we see that $ \vec{w} \in C([0,T],\Ldottwo(\Td)) $ as it satisfies the following equation,
\begin{equation}
\label{eq:limit-equation-u=0}
D_t \vec{w} + \A(0) \vec{w} = 0,
\end{equation}
where $ \A(0) = \op(A(0) \pi) $ with
\begin{equation}
\label{eq:A(0)-formula}
\begin{split}
A(0) 
= |\xi|^{3/2} \begin{pmatrix} 1 & \phantom{*}0 \\ 0 & -1 \end{pmatrix} + \frac{g}{2|\xi|^{1/2}}  \begin{pmatrix} \phantom{*}1 & \phantom{*}1 \\ -1 & -1 \end{pmatrix} + \frac{|\xi|^{1/2} m_{\b}}{2}  \begin{pmatrix} \phantom{*}1 & -1 \\ -1 & \phantom{*}1 \end{pmatrix}.
\end{split}
\end{equation}
By the weak convergence, $ \phiw  \vec{e}\cdot\vec{w}_n \rightharpoonup \phiw  \vec{e}\cdot\vec{w} $ in $ \Ltwo([0,T],\Ltwo(\Td)) $, we have
\begin{equation*}
\int_0^T \|\phiw  \vec{e}\cdot\vec{w}\|_{\Ltwo}^2 \dt \le \liminf_{n \to \infty} \int_0^T \|\phiw  \vec{e}\cdot\vec{w}_n\|_{\Ltwo}^2 \dt = 0,
\end{equation*}
implying that $ \vec{e}\cdot\vec{w}|_{]0,T[ \times \omega} = 0 $ in the sense of distribution.
Then by the weak observability~\eqref{eq:weak-observability-Ltwo} and Rellich–Kondrachov's compact injection theorem,
\begin{equation*}
\|\vec{w}(0)\|_{\H{-N}}^2 
= \lim_{n \to \infty} \|\vec{w}_n(0)\|_{\H{-N}}^2 
\gtrsim \limsup_{n\to\infty} \Big( \|\vec{w}_n(0)\|_{\Ltwo}^2 - \int_0^T \|\phiw  \vec{e} \cdot \vec{w}_n\|_{\Ltwo}^2 \dt \Big)
= 1,
\end{equation*}
whence $ \vec{w}(0) \ne 0 $.
To conclude, it suffices to prove the unique continuation property of~\eqref{eq:limit-equation-u=0} and obtain a contradiction. This is done in the following lemma.
\end{proof}

\begin{lemma}
Under the hypothesis of Proposition~\ref{prop:Ltwo-observability-sys}, suppose that $ \vec{w} \in C([0,T],\Ldottwo(\Td)) $ satisfies~\eqref{eq:limit-equation-u=0} and that $ \vec{e}\cdot\vec{w}|_{I \times \omega} = 0 $ for some interval $ I \subset [0,T] $ with non-empty interior, then $ \vec{w} \equiv 0 $.
\end{lemma}
\begin{proof}
There is no harm in assuming that $ I = [0,T] $. For any $ 0 \le \delta < T $, define the following $ \C $-linear subspace space of $ \Ldottwo(\Td) $
\begin{equation*}
\mathscr{N}_\delta = \{\vec{w}_0 \in \Ldottwo(\Td) : \vec{e}\cdot \exp\{-it\A(0)\}\vec{w}_0|_{[0,T-\delta] \times \omega} = 0\},
\end{equation*}
where $ \mathrm{e}^{-it\A(0)} \vec{w}_0 \in C([0,T],\Ldottwo(\Td)) $ denotes the solution to equation~\eqref{eq:limit-equation-u=0} with initial data $ \vec{w}_0 $.
It suffices to show that for some $ 0 \le \delta_0 < T $, $ \mathscr{N}_{\delta_0} = \{0\}  $.

Applying the weak observability~\eqref{eq:weak-observability-Ltwo} with $ \varepsilon_0 = 0 $, $ \udlu = 0 $, $ N > 0 $, for time $ T - \delta > 0 $, and for $ \vec{w}_0 \in \mathscr{N}_\delta $,
\begin{equation}
\label{eq:compactness-of-N}
\|\vec{w}_0\|_{\Ltwo} \le C(T-\delta) \|\vec{w}_0\|_{\H{-N}},
\end{equation}
where the constant $ C(T-\delta) $ is uniformly bounded as long as $ T-\delta $ stays away from $ 0 $.
This implies that, by the compact injection theorem, the closed unit ball of $ (\mathscr{N}_\delta,\|\cdot\|_{\Ltwo}) $ is compact, and thus
\begin{equation*}
\dim \mathscr{N}_\delta < \infty, \quad \forall \delta \in [0,T).
\end{equation*}
Moreover, by definition
\begin{equation*}
\delta < \delta' \Rightarrow \mathscr{N}_\delta \subset \mathscr{N}_{\delta'},
\end{equation*}
which implies that the family $ \{\mathscr{N}_\delta\}_{0\le\delta<T} $ is totally ordered by the inclusion relation~$ \subset $. If $ \dim \mathscr{N}_0 = 0 $, then the proof is closed. Otherwise, there exists a $ \delta_0 > 0 $, such that for all $ 0 < \delta \le \delta_0 $,
\begin{equation*}
\dim \mathscr{N}_\delta = \dim \mathscr{N}_{\delta_0} \ge \dim \mathscr{N}_0 > 0,
\end{equation*}
or equivalently,
\begin{equation*}
\mathscr{N} \bydef \mathscr{N}_\delta = \mathscr{N}_{\delta_0} \supset \mathscr{N}_0 \ne \{0\}.
\end{equation*}
We will show that $ \mathscr{N} = \{0\} $ to obtain a contradiction.
Let $ \vec{w}_0 \in \mathscr{N} $ and set $ \vec{w}(t) = \mathrm{e}^{-it\A(0)}\vec{w}_0$. Then for $ 0 < \epsilon < \delta_0 $, by the identity $ \vec{w}(t) = \mathrm{e}^{-i(t-\epsilon)\A(0)} \vec{w}(\epsilon) $, we see that $ \vec{w}(\epsilon) \in \mathscr{N} $. And since $ \mathscr{N} $ is a linear vector space, $ \frac{1}{i\epsilon} (\vec{w}(\epsilon)-\vec{w}(0)) \in \mathscr{N} $. Moreover, apply the compactness~\eqref{eq:compactness-of-N} with $ N = 3/2 $,
\begin{align*}
\big\|\frac{1}{i\epsilon} (\vec{w}(\epsilon)-\vec{w}(0))\big\|_{\Ltwo}
& \lesssim \big\|\frac{1}{i\epsilon}(\vec{w}(\epsilon)-\vec{w}(0))\big\|_{\H{-3/2}} 
\lesssim \sup_{0 \le t \le \epsilon} \|D_t \vec{w}(t)\|_{\H{-3/2}} \\
& \lesssim \sup_{0 \le t \le \epsilon} \|\A(0) \vec{w}(t)\|_{\H{-3/2}}
\lesssim \sup_{0 \le t \le \epsilon} \|\vec{w}(t)\|_{\Ltwo} 
\lesssim \|\vec{w}_0\|_{\Ltwo}.
\end{align*}
So the family $ \{\frac{1}{i\epsilon}(\vec{w}(\epsilon)-\vec{w}(0))\}_{0 < \epsilon < \delta} $ is bounded in $ (\mathscr{N},\|\cdot\|_{\Ltwo}) $, and consequently relatively compact. Therefore, up to a subsequence $ \epsilon_n \to 0 $, $ \frac{1}{i\epsilon_n}(\vec{w}(\epsilon_n)-\vec{w}(0)) \to D_t \vec{w}(0) = - \A(0) \vec{w}_0 $ strongly in $ (\mathscr{N},\|\cdot\|_{\Ltwo}) $, and we have a well defined $ \C $-linear map on $ \mathscr{N} $,
\begin{equation*}
\vec{w}_0 \mapsto \A(0) \vec{w}_0,
\end{equation*}
which admits an eigenfunction, say $ \mathscr{N} \ni \vec{w}_0 = \binom{w^+_0}{w^-_0} \ne 0 $, with $ \A(0) \vec{w} = \lambda \vec{w} $ for some $ \lambda \in \C $. By the definition of $ \A(0) $,
\begin{align*}
|D_x|^{3/2} w^+_0 + \frac{g}{2} |D_x|^{-1/2} (w^+_0 + w^-_0) + \frac{1}{2}|D_x|^{1/2}M_{\b} (w^+_0 - w^-_0) & = \lambda w^+_0, \\
-|D_x|^{3/2} w^-_0 - \frac{g}{2} |D_x|^{-1/2} (w^+_0 + w^-_0) - \frac{1}{2} |D_x|^{1/2}M_{\b} (w^+_0 - w^-_0) & = \lambda w^-_0,
\end{align*}
Taking respectively the sum and the difference of the above two equations,
\begin{align*}
|D_x|^{3/2} (w^+_0 - w^-_0) & = \lambda (w^+_0 + w^-_0), \\
(|D_x|^{3/2} + g |D_x|^{-1/2}) (w^+_0 + w^-_0) + |D_x|^{1/2} M_{\b} (w^+_0 - w^-_0)  & = \lambda (w^+_0 - w^-_0).
\end{align*}
Apply $ |D_x|^{3/2} $ to the second equation, and use the first one to eliminate $ w^+_0 - w^-_0 $,
\begin{equation*}
(|D_x|^{3} + g |D_x| + \lambda |D_x|^{1/2}M_{\b}) (w^+_0 + w^-_0) = \lambda^2 (w^+_0 + w^-_0).
\end{equation*}
Now that $ M_{\b} $ is of order $ -\infty $, it is an elliptic equation, implying that $ w^+_0 + w^-_0 $ has only a finite number of Fourier modes. So it is analytic and can never vanish on a nonempty open set unless it is identically zero. Hence by the definition of $ \mathscr{N} $, $ (w^+_0 + w^-_0)|_\omega = \vec{e} \cdot \vec{w}_0|_\omega = 0 $, we have $ w^+_0 + w^-_0 \equiv 0 $. Then by the first equation, $ |D_x|^{3/2} (w^+_0 - w^-_0) = \lambda (w^+_0 + w^-_0) = 0 $, we have $ w^+_0 - w^-_0 \equiv 0 $, for~$ \vec{w}_0 $ has no zero frequency. Therefore $ \vec{w}_0 = 0 $, which is a contradiction.
\end{proof}

\section{$ \H{s} $ Linear Control}
\label{sec:H^s-linear-control}

\subsection{Sobolev Regularity of HUM Control Operator}

\begin{proposition}
\label{prop:regularity-H^s}
Suppose that~$ \omega $ satisfies the geometric control condition, $ s $ is sufficiently large, $ \mu \ge 0 $, $ T > 0 $, $ \varepsilon_0 > 0 $, and $ \udlu \in \Cls^{1,s}(T,\varepsilon_0) $. Then for~$ \varepsilon_0 $ sufficiently small, $ \Theta|_{\Hdot{\mu}} = \Theta(\underline{u})|_{\Hdot{\mu}} $ defines a bounded $ \R $-linear operator from $ \Hdot{\mu}(\Td) $ to $ C([0,T],\Hdot{\mu}(\Td)) $, such that 
\begin{equation*}
\|\Theta|_{\Hdot{\mu}}\|_{\L(\Hdot{\mu}, C([0,T],\Hdot{\mu}))} \lesssim 1.
\end{equation*}
\end{proposition} 
To prove this, recall that $ \Theta = - \B^* \Sol \K^{-1}$, with $ \K = - \Range \B \B^* \Sol $, where by the Hilbert uniqueness method (Proposition~\ref{prop:HUM-Ltwo}), $ \K$ defines an isomorphism on $ \Ldottwo(\Td) $; by Theorem~\ref{thm:Operator-Norm-Estimate-Paradiff} and Corollary~\ref{cor:Range-Sol-regularity}, $ \K|_{\Hdot{\mu}} $ sends $ \Hdot{\mu}(\Td) $ to itself. Therefore, to prove Proposition~\ref{prop:regularity-H^s}, it remains to show that $ \K|_{\Hdot{\mu}} $ defines an isomorphism on $ \Hdot{\mu}(\Td) $. Our idea is to use the Lax-Milgram's theorem as in the Hilbert uniqueness method. However, for technical reasons, we will work in the semi\-classical Sobolev spaces $ \Hdot{\mu}_h(\Td) = (\Hdot{\mu}(\Td), \|\cdot\|_{\H{\mu}_h})$, equipped with a real scalar product,
\begin{equation*}
\Re (u,v)_{\H{\mu}_h} = \Re (\langle hD_x \rangle^\mu u,\langle hD_x \rangle^\mu v)_{\Ltwo}.
\end{equation*}
For each fixed~$ h $, $ \Hdot{\mu}_h(\Td) $ and $ \Hdot{\mu}(\Td) $ are isomorphic as Banach spaces with equivalent norms, even though not uniformly in~$ h $. Therefore, Proposition~\ref{prop:regularity-H^s} is a consequence of the following Proposition~\ref{prop:K-is-H^s_h-isomorphism}.

\begin{proposition}
\label{prop:K-is-H^s_h-isomorphism}
Under the hypothesis of Proposition~\ref{prop:regularity-H^s}, for~$ h $ and~$ \varepsilon_0 $ sufficiently small, $ \K|_{\Hdot{\mu}_h} : \Hdot{\mu}_h(\Td) \to \Hdot{\mu}_h(\Td) $ defines an isomorphism.
\end{proposition}

Some preliminary results will be proven before the proof of this proposition.
\begin{lemma}
\label{lem:equivalent-semiclassical-sobolev-norm}
Under the hypothesis of Proposition~\ref{prop:regularity-H^s}, define the symbol $ \hat{\gamma}^{(\mu)} = (\gamma^{(3/2)})^{2\mu/3}$ for $ \mu \ge 0 $, and the para\-differential operator
\begin{equation}
\label{eq:definition-Lambda^s_h}
\Lambda^\mu_h = 1 + h^\mu \T{\hat{\gamma}^{(\mu)}}.
\end{equation}
Then for $ \varepsilon_0 $ sufficiently small, $ \Lambda^\mu_h : \Hdot{\mu}_h(\Td) \to \Ldottwo(\Td) $ is invertible (whose inverse will be denoted by $ \Lambda^{-\mu}_h \bydef (\Lambda^\mu_h)^{-1} $). Moreover, we have the following estimate, uniformly in~$ h $,
\begin{equation}
\label{eq:compare-H^s_h-norm}
\Lambda^\mu_h - \langle hD_x \rangle^\mu = O(\varepsilon_0)_{\Linf([0,T],\L(\Hdot{\mu}_h,\Ldottwo))},
\end{equation}
which in particular implies the norm equivalence, uniformly for $ t_0 \in [0,T] $,
\begin{equation}
\label{eq:equivalence-two-H^s_h-norm}
\|\cdot\|_{\H{\mu}_h} \sim \|\Lambda_h^\mu|_{t=t_0} \cdot \|_{\Ltwo}.
\end{equation}
\end{lemma}
\begin{proof}
We omit the time variable in the proof, and write
\begin{equation}
\label{eq:difference-of-two-H^s_h-norm}
\Lambda^\mu_h = (1 + B_h) \langle hD_x \rangle^\mu
\end{equation}
with $ B_h = h^\mu T_{\hat{\gamma}^{(\mu)}-|\xi|^\mu} (1+h^\mu T_{|\xi|^\mu})^{-1} $. Then by Theorem~\ref{thm:Operator-Norm-Estimate-Paradiff}, $ B_h = O(\varepsilon_0)_{\L(\Ldottwo,\Ldottwo)} $, uniformly in~$ h $. For~$ \varepsilon_0 $ sufficiently small, $ \Id + B_h : \Ldottwo(\Td) \to \Ldottwo(\Td) $ is invertible. The norm equivalence follows as $ (1+h^\mu\T{|\xi|^\mu})\jp{hD_x}^{-\mu}  $ and $ (1+h^\mu\T{|\xi|^\mu})^{-1}\jp{hD_x}^{\mu}  $ are both bounded on $ \Ltwo $ as they both have Fourier multipliers which are bounded independently of~$ h $.
\end{proof}

The key point to the proof of Proposition~\ref{prop:K-is-H^s_h-isomorphism} is the following commutator estimate.

\begin{lemma}
\label{lem:commutator-estimate-K}
Under the hypothesis of Proposition~\ref{prop:regularity-H^s}, for~$ h $ and~$ \varepsilon_0 $ sufficiently small and $ \mu \ge 1 $, the following commutator estimate holds,
\begin{equation}
\label{eq:commutator-estimate-K-Lambda-s-h}
\big[\K,\Lambda^\mu_h|_{t=0}\big] \Lambda^{-\mu}_h|_{t=0} = O(\varepsilon_0 + h)_{\L(\Ldottwo,\Ldottwo)},
\end{equation}
\end{lemma}
\begin{proof}
By the definition of~$ \K $, write
\begin{align*}
-\big[\K,\Lambda^\mu_h|_{t=0}\big] \Lambda^{-\mu}_h|_{t=0} 
& = (\Range\T{\hat{\gamma}^{(\mu)}} - \T{\hat{\gamma}^{(\mu)}|_{t=0}}\Range) \B \B^* \Sol (h^\mu \Lambda^{-\mu}_h|_{t=0}) \\
& \qquad + h \Range (\B\B^*\T{\hat{\gamma}^{(\mu)}}  - \T{\hat{\gamma}^{(\mu)}} \B\B^*) \Sol (h^{\mu-1} \Lambda^{-\mu}_h|_{t=0}) \\
& \qquad \qquad + \Range\B\B^*(\Sol\T{\hat{\gamma}^{(\mu)}|_{t=0}} - \T{\hat{\gamma}^{(\mu)}}\Sol) (h^\mu\Lambda^{-\mu}_h|_{t=0}).
\end{align*}
By Lemma~\ref{lem:equivalent-semiclassical-sobolev-norm}, for $ 0 \le \sigma \le \mu $,
\begin{equation*}
\Lambda^{-\mu}_h|_{t=0} 
= O(1)_{\L(\Ldottwo, \Hdot{\mu}_h)} 
= O(1)_{\L(\Ldottwo, \Hdot{\sigma}_h)} 
= O(h^{-\sigma})_{\L(\Ldottwo, \Hdot{\sigma})},
\end{equation*}
from which $ \|h^\mu \Lambda^{-\mu}_h|_{t=0}\|_{\L(\Ldottwo,\Hdot{\mu})} \lesssim 1$ , $ \|h^{\mu-1} \Lambda^{-\mu}_h|_{t=0}\|_{\L(\Ldottwo,\Hdot{\mu-1})}  \lesssim 1 $.
Then~\eqref{eq:commutator-estimate-K-Lambda-s-h} results from the following estimates,
\begin{equation*}
\begin{split}
\|\B\B^*\T{\hat{\gamma}^{(\mu)}}  - \T{\hat{\gamma}^{(\mu)}} \B\B^*\|_{C([0,T],\L(\Hdot{\mu-1},\Ldottwo))} & \lesssim 1, \\
\|\Range\T{\hat{\gamma}^{(\mu)}} - \T{\hat{\gamma}^{(\mu)}|_{t=0}}\Range\|_{\L(\Ltwo([0,T],\Hdot{\mu}),\Ldottwo)} & \lesssim \varepsilon_0, \\
\|\Sol\T{\hat{\gamma}^{(\mu)}|_{t=0}} - \T{\hat{\gamma}^{(\mu)}}\Sol\|_{\L(\Hdot{\mu}, C([0,T],\Ldottwo))}  & \lesssim \varepsilon_0,
\end{split}
\end{equation*}
which will be treated separately in the following lemmas.
\end{proof}

\begin{lemma}
$\|\Range\T{\hat{\gamma}^{(\mu)}} - \T{\hat{\gamma}^{(\mu)}|_{t=0}}\Range\|_{\L(\Ltwo([0,T],\Hdot{\mu}),\Ldottwo)} \lesssim \varepsilon_0.$
\end{lemma}
\begin{proof}
Let $ G \in \Ltwo([0,T],\Hdot{\mu}(\Td)) $, and let $ u \in C([0,T],\Hdot{\mu}(\Td)) $, $ v \in C([0,T],\Ldottwo(\Td)) $ solve respectively the following equations
\begin{equation*}
(\pt + P) u  = G, \  u(T) = 0;  \qquad
(\pt + P) v  = \T{\hat{\gamma}^{(\mu)}} G, \  v(T) = 0.
\end{equation*}
Then $ w \bydef v - \T{\hat{\gamma}^{(\mu)}} u \in C([0,T],\Ldottwo(\Td)) $ satisfies
\begin{equation*}
(\pt + P) w = [\T{\hat{\gamma}^{(\mu)}},\pt + P] u, \quad w(T) = 0.
\end{equation*}
Observe that $ [\T{\hat{\gamma}^{(\mu)}},\pt + P] = -\T{\pt \hat{\gamma}^{(\mu)} } + [\T{\hat{\gamma}^{(\mu)}},P] $ where $ \T{\pt \hat{\gamma}^{(\mu)} } = O(\varepsilon_0)_{\L(\Hdot{\mu},\Ldottwo)} $ by Theorem~\ref{thm:Operator-Norm-Estimate-Paradiff}. The main estimate for $ [\T{\hat{\gamma}^{(\mu)} },P] $ is $ [\T{\hat{\gamma}^{(\mu)} },\T{\gamma^{(3/2)}}] $. By Remark~\ref{rmk:smallness-symbol},
\begin{equation*}
[\T{\hat{\gamma}^{(\mu)}},\T{\gamma^{(3/2)}}]
= \T{\{\hat{\gamma}^{(\mu)},\gamma^{(3/2)}\}/i} + O(\varepsilon_0)_{\L(\Hdot{\mu},\Ldottwo)} \\
= O(\varepsilon_0)_{\L(\Hdot{\mu},\Ldottwo)},
\end{equation*}
because $ \{\hat{\gamma}^{(\mu)},\gamma^{(3/2)}\} = 0 $ for $ \xi \ne 0 $. Combining the lower order terms, we then have $ [\T{\hat{\gamma}^{(\mu)}},\pt + P] = O(\varepsilon_0)_{\L(\Hdot{\mu},\Ldottwo)} $. Finally by an energy estimate,
\begin{align*}
\|(\Range\T{\hat{\gamma}^{(\mu)}} - \T{\hat{\gamma}^{(\mu)}|_{t=0}}\Range) G\|_{\Ltwo} 
& = \|w(0)\|_{\Ltwo}
 \lesssim \|[\T{\hat{\gamma}^{(\mu)}},\pt + P] u\|_{\Lone([0,T],\Ltwo)} \\
& \lesssim \varepsilon_0 \|u\|_{\Lone([0,T],\H{\mu})}
\lesssim \varepsilon_0 \|G\|_{\Ltwo([0,T],\H{\mu})}.
\end{align*}
\end{proof}

\begin{lemma}
$ \|\Sol\T{\hat{\gamma}^{(\mu)}|_{t=0}} - \T{\hat{\gamma}^{(\mu)}}\Sol\|_{\L(\Hdot{\mu}, C([0,T],\Ldottwo))} \lesssim \varepsilon_0. $
\end{lemma}
\begin{proof}
Let $ u_0 \in \Hdot{\mu}(\Td) $, and let $ u \in C([0,T],\Hdot{\mu}(\Td)) $, $ v \in C([0,T],\Ldottwo(\Td)) $ solve respectively the following equations,
\begin{equation*}
(\pt - P^*) u = 0, \  u(0) = u_0, \qquad
(\pt - P^*) v = 0, \  v(0) = \T{\hat{\gamma}^{(\mu)}|_{t=0}} u_0.
\end{equation*}
Then $ w = v - \T{\hat{\gamma}^{(\mu)}} u \in C([0,T],\Ldottwo(\Td)) $ satisfies
\begin{equation*}
(\pt - P^*) w = [\T{\hat{\gamma}^{(\mu)}},\pt - P^*] u, \quad w(0) = 0.
\end{equation*}
Recall that, for $ \sigma \ge 0 $, $ P - P^* = O(\varepsilon_0)_{\L(\Hdot{\sigma},\Hdot{\sigma})} $, we then have
$ [\T{\hat{\gamma}^{(\mu)}},\pt - P^*] 
= -\T{\pt\hat{\gamma}^{(\mu)}} - [\T{\hat{\gamma}^{(\mu)}},P] + [\T{\hat{\gamma}^{(\mu)}},P-P^*]
= O(\varepsilon_0)_{\L(\Ltwo,\Ltwo)} $. Again by an energy estimate,
\begin{align*}
\|(\Sol\T{\hat{\gamma}^{(\mu)}|_{t=0}} - \T{\hat{\gamma}^{(\mu)}}\Sol) u_0\|_{\Ltwo([0,T],\Ltwo)} 
& = \|w\|_{\Ltwo([0,T],\Ltwo)} 
\lesssim \|[\T{\hat{\gamma}^{(\mu)}},\pt - P^*] u\|_{\Lone([0,T],\Ltwo)} \\
& \lesssim \varepsilon_0 \|u\|_{\Lone([0,T],\H{\mu})}
\lesssim \varepsilon_0 \|u_0\|_{\H{\mu}}.
\end{align*}
\end{proof}

\begin{lemma}
$\|\B\B^*\T{\hat{\gamma}^{(\mu)}}  - \T{\hat{\gamma}^{(\mu)}} \B\B^*\|_{C([0,T],\L(\Hdot{\mu-1},\Ldottwo))} \lesssim 1$.
\end{lemma}
\begin{proof}
By symbolic calculus $ [\T{q},\T{\hat{\gamma}^{(\mu)}}] $ and $ [\T{q}^*,\T{\hat{\gamma}^{(\mu)}}] $ are of order $ \mu-1 $, it suffices to show that $ [\phiw,\T{\hat{\gamma}^{(\mu)}}] $ is also of order $ \mu-1 $. Indeed, write $ \phiw = \T{\phiw} + (\phiw - \T{\phiw}) $, then $ [\T{\phiw},\T{\hat{\gamma}^{(\mu)}}] $ is of order $ \mu-1 $, while $ \phiw - \T{\phiw} $ is of order $ -\infty $ since $ \phiw $ is smooth.
\end{proof}

\begin{proof}[Proof of Proposition~\ref{prop:K-is-H^s_h-isomorphism}]
Consider the $ \R $-bilinear form on $ \Hdot{\mu}_h(\Td) $,
\begin{equation*}
\varpi^\mu_h(f_0,g_0) = \Re (\Lambda^\mu_h|_{t=0} \K f_0, \Lambda^\mu_h|_{t=0} g_0)_{\Ltwo}.
\end{equation*}
Then by Lemma~\ref{lem:commutator-estimate-K}, and the $ \Ltwo $-coercivity of~$ \K $, for~$ h $ and~$ \varepsilon_0 $ sufficiently small,
\begin{align*}
\varpi^\mu_h(f_0,f_0)
& = \Re (\K \Lambda^\mu_h|_{t=0}  f_0, \Lambda^\mu_h|_{t=0} f_0)_{\Ltwo}
+ \Re ([\Lambda^\mu_h|_{t=0},\K]\Lambda^{-\mu}_h|_{t=0} \Lambda^\mu_h|_{t=0}f_0, \Lambda^\mu_h|_{t=0} f_0) \\
& \gtrsim \|\Lambda^\mu_h|_{t=0}  f_0\|_{\Ltwo}^2 - \|[\Lambda^\mu_h|_{t=0},\K]\Lambda^{-\mu}_h|_{t=0}\|_{\L(\Ldottwo,\Ldottwo)} \|\Lambda^\mu_h|_{t=0}  f_0\|_{\Ltwo}^2 \\
& \gtrsim \|f_0\|_{\H{\mu}_h}^2 - (\varepsilon_0 + h) \|f_0\|_{\H{\mu}_h}^2 \\
& \gtrsim \|f_0\|_{\H{\mu}_h}^2.
\end{align*}
Therefore, $ \varpi^\mu_h $ is coercive on $ \Hdot{\mu}_h(\Td) $, and we conclude by Lax-Milgram's theorem.
\end{proof}

\subsection{$ \H{s} $-Controllability}
\label{sec:H^s-controllability-perturbed}

The HUM control operator~$ \Theta $ solves the linear control problem~\eqref{eq:equation-paradiff-linear-non-perturbed} without the perturbation terms $ R u $ and $ \beta F $. This section constructs a control operator~$ \Phi $ for the linear control problem
\begin{equation}
\label{eq:linear-equation-perturbed}
(\pt + P + R) u = (\B + \beta) F,
\end{equation}
where $ P = P(\udlu) $, $ R = R(\udlu) $, $ \B = \B(\udlu) $, $ \beta = \beta(\udlu) $ with $ \udlu \in \Cls^{1,s}(T,\varepsilon_0) $.

\begin{proposition}
\label{prop:definition-of-Phi}
Suppose that~$ \omega $ satisfies the geometric control condition, $ s $ is sufficiently large, $ T > 0 $, $ \varepsilon_0 > 0 $, and $ \udlu \in \Cls^{1,s}(T,\varepsilon_0) $. Then for~$ \varepsilon_0 > 0 $ sufficiently small, there exists an operator
\begin{equation*}
\Phi = \Phi(\underline{u}) : \Hdot{s}(\Td) \to C([0,T],\Hdot{s}(\Td))
\end{equation*}
satisfying
\begin{equation*}
\|\Phi\|_{\L(\Hdot{s},C([0,T],\Hdot{s}))} \lesssim 1,
\end{equation*}
such that, for $ u_0 \in \Hdot{s}(\Td) $, setting $ F = \Phi u_0 $, the solution~$ u $ to~\eqref{eq:linear-equation-perturbed} with initial data $ u(0) = u_0 $ vanishes at time~$ T $, that is, $ u(T) = 0 $.
\end{proposition}
\begin{proof}
First we define a new solution operator $ \Sol_{\Theta} = \Sol_{\Theta}(\udlu) $. 

\begin{lemma}
For $ v_0 \in \Ldottwo(\Td) $, set $ F = \Theta v_0 \in C([0,T],\Ldottwo(\Td)) $. Let $ v \in C([0,T],\Ldottwo(\Td)) $ be the solution to
\begin{equation*}
(\pt + P) v   = \B F, \quad v(0) = v_0, \  v(T) = 0.
\end{equation*} 
We set $ \Sol_{\Theta} v_0 = v $, then for any $ \mu \ge 0 $, $ \Sol_{\Theta}|_{\Hdot{\mu}} $ is a bounded $ \R $-linear operator from $ \Hdot{\mu}(\Td) $ to $ C([0,T],\Hdot{\mu}(\Td)) $.
\end{lemma}
\begin{proof}
For $ v_{0,i} \in \Ldottwo(\Td) $ and $ \lambda_i \in \R$ ($ i = 1,2 $), set $ v_i = \Sol_\Theta v_{0,i} $ and $ v_\lambda = \lambda_1 v_1 + \lambda_2 v_2 $. Then they satisfy the equations
\begin{align*}
(\pt + P) v_i  &= \B \Theta v_{0,i}, & v_i(0) &= v_{0,i}, &  v_i(T) & = 0. \\
(\pt + P) v_\lambda  &= \B (\lambda_1\Theta v_{0,1} + \lambda_2\Theta v_{0,2}), & v_\lambda(0) & = \lambda_1 v_{0,1} + \lambda_2 v_{0,2}, &  v_\lambda(T) &= 0.
\end{align*}
By the $ \R $-linearity of $ \Theta $, $ \lambda_1\Theta v_{0,1} + \lambda_2\Theta v_{0,2} = \Theta (\lambda_1 v_{0,1} + \lambda_2 v_{0,2}) $, whence $ v_\lambda = \Sol_\Theta(\lambda_1 v_{0,1} + \lambda_2 v_{0,2}) $, that is, the $ \R $-linearity of $ \Sol_\Theta $. By Proposition~\ref{prop:regularity-H^s} and Proposition~\ref{prop:app-wellposedness-paradiff-eq}, we have
\begin{equation}
\label{eq:S-theta-estimate}
\|\Sol_\Theta v_0\|_{C([0,T],\H{\mu})}
\lesssim \|\B \Theta v_0\|_{\Lone([0,T],\H{\mu})}
\lesssim \|v_0\|_{\H{\mu}}.
\end{equation}
\end{proof}

Then we define a new range operator~$ \tilde{\Range} = \tilde{\Range}(\udlu) $. 
\begin{lemma}
For $ G \in \Ltwo([0,T],\Hdot{s}(\Td)) $, let $ w \in C([0,T],\Hdot{s}(\Td)) $ be the solution to the equation
\begin{equation*}
(\pt + P + R) w = G, \quad w(T) = 0.
\end{equation*}
We set $ \tilde{\Range} G = w(0) $, then $ \tilde{\Range} $ defines a bounded $ \R $-linear operator from $ \Ltwo([0,T],\Hdot{s}(\Td)) $ to $ \Hdot{s}(\Td) $.
\end{lemma}
\begin{proof}
This is a consequence of Proposition~\ref{prop:app-wellposedness-paradiff-eq} for the backwards equation.
\end{proof}

Now we start constructing~$ \Phi $ by perturbing~$ \Theta $. For $ v_0 \in \Hdot{s}(\Td) $, set $ v = \Sol_{\Theta} v_0 $, $ F = \Theta v_0 $, and let~$ w $ be the solution to the equation
\begin{equation*}
(\pt + P + R) w  = -Rv + \beta F, \quad w(T) = 0,
\end{equation*}
then $ u = v + w $ satisfies the equation
\begin{equation*}
(\pt + P + R) u = (\B + \beta) F, \qquad u(0) = v_0 + w_0, \  u(T) = 0,
\end{equation*}
where~$ w_0 \bydef w(0) = \tilde{\Range} (-Rv + \beta F) = \tilde{\Range}(-R\Sol_\Theta + \beta\Theta) v_0 $. Define the perturbation operator,
\begin{equation*}
\E = \tilde{\Range}(-R\Sol_\Theta + \beta\Theta|_{\Hdot{s}}) : \Hdot{s}(\Td) \mapsto \Hdot{s}(\Td),
\end{equation*}
then $ u_0 = (1 + \E) v_0 $. By~\eqref{eq:estimate-various-operators},~\eqref{eq:S-theta-estimate}, Proposition~\ref{prop:Ltwo-linear-control-HUM} and Corollary~\ref{cor:Range-Sol-regularity}, $ \E $ is small in the sense that,
\begin{equation*}
\|\E\|_{\L(\Hdot{s},\Hdot{s})} \lesssim \varepsilon_0^\vartheta, \quad \vartheta > 0.
\end{equation*}
Therefore, for $ \varepsilon_0 $ sufficiently small, $ 1 + \E : \Hdot{s}(\Td) \to \Hdot{s}(\Td) $ is invertible, and
\begin{equation}
\label{eq:Psi-as-perturbation-of-Theta}
\Phi \bydef \Theta|_{\Hdot{s}} (1 + \E)^{-1}
\end{equation}
is the desired control operator.
\end{proof}

\subsection{Contraction Estimate of Control Operator}
\label{sec:contraction-estimate}

In the next section, we use an iterative scheme to solve the nonlinear control problem~\eqref{eq:equation-control-of-u}. Contraction estimates of some operators, especially of the control operator, will be of great importance.

To do this, let $ \underline{u}_i \in \Cls^{1,s}(T,\varepsilon_0) $, $ (i = 1,2) $. And let $ \psi_i = \psi(\udlu_i) \in \Hdot{s}(\Td) $ and $ \eta_i = \eta(\udlu_i) \in \Hdot{s+1/2}(\Td) $ be determined by Proposition~\ref{prop:reversibility-u-eta-psi}. For any symbol depending on $ \eta $, $ a = a(\eta) $, set $ a_i = a(\eta_i) $. For any operator $ \mathcal{L} = \mathcal{L}(\underline{u})$, set  $ \mathcal{L}_i = \mathcal{L}(\udlu_i) $.

\begin{lemma}
\label{lem:stability-u-eta}
Suppose that $ s  > 3/2 + d/2 $, then for $ \varepsilon_0 $  sufficiently small,
\begin{equation*}
\|\eta_i\|_{\H{s+1/2}}  \lesssim  \varepsilon_0, \quad
\|\eta_1 - \eta_2\|_{\H{s+1/2}}  \lesssim  \|\udlu_1-\udlu_2\|_{\H{s}}.
\end{equation*}
\end{lemma}
\begin{proof}
By Proposition~\ref{prop:reversibility-u-eta-psi}, $ \|\eta_i\|_{\H{s+1/2}} \le 2 \varepsilon_0 $. Then write 
\begin{equation*}
\eta_1 - \eta_2 = - \T{p_1}^{-1} \Im\, (\udlu_1 - \udlu_2) + \T{p_2}^{-1}(\T{p_1} - \T{p_2})\T{p_1}^{-1} \Im\, \udlu_2,
\end{equation*}
from which 
\begin{equation*}
\|\eta_1-\eta_2\|_{\H{s+1/2}} \lesssim  \|\udlu_1-\udlu_2\|_{\H{s}} +  \varepsilon_0 \|\eta_1 - \eta_2\|_{\H{s+1/2}}.
\end{equation*}
We conclude for $ \varepsilon_0 $ sufficiently small.
\end{proof}

\begin{lemma}
\label{lem:stability-B-B^*}
Suppose that $ s > 3/2 + d/2 $, $ \sigma \in \R $, then for $ \varepsilon_0 $ sufficiently small,
\begin{equation*}
\|\B^*_1 - \B^*_2\|_{\L(\H{\sigma},\H{\sigma})}
=\|\B_1 - \B_2\|_{\L(\H{\sigma},\H{\sigma})} 
\lesssim \|\udlu_1 - \udlu_2\|_{\H{s}}.
\end{equation*}
\end{lemma}
\begin{proof}
By the definition of~$ \B $, Theorem~\ref{thm:Operator-Norm-Estimate-Paradiff} and Lemma~\ref{lem:stability-u-eta}, this comes from the estimate
\begin{equation*}
\|\T{q_1} - \T{q_2}\|_{\L(\H{\sigma},\H{\sigma})}
\lesssim \|\eta_1-\eta_2\|_{\H{s+1/2}}
\lesssim \|\udlu_1-\udlu_2\|_{\H{s}}.
\end{equation*}
\end{proof}

\begin{lemma}
\label{lem:stability-P-P^*}
Suppose that $ s > 3/2 + d/2 $, $ \sigma \in \R $, then for $ \varepsilon_0 $ sufficiently small, 
\begin{equation*}
\|P_1 - P_2\|_{\L(\Hdot{\sigma+3/2},\Hdot{\sigma})} + \|P^*_1 - P^*_2\|_{\L(\Hdot{\sigma+3/2},\Hdot{\sigma})} \lesssim \|\udlu_1 - \udlu_2\|_{\H{s}}.
\end{equation*}
\end{lemma}
\begin{proof}
The main estimate is, by Theorem~\ref{thm:Operator-Norm-Estimate-Paradiff} and Lemma~\ref{lem:stability-u-eta},
\begin{equation*}
\|\T{\gamma_1^{(3/2)}} - \T{\gamma_2^{(3/2)}}\|_{\L(\Hdot{\sigma+3/2},\Hdot{\sigma})} 
\lesssim \|\eta_1-\eta_2\|_{\H{s+1/2}}
\lesssim \|\udlu_1 - \udlu_2\|_{\H{s}}.
\end{equation*}
\end{proof}

\begin{lemma}
\label{lem:stability-of-solution-operator}
Suppose that $ s > 3/2 + d/2 $, $ \sigma \ge 0 $, then for $ \varepsilon_0 $ sufficiently small,
\begin{equation*}
\|\Sol_1 - \Sol_2\|_{\L(\Hdot{\sigma+3/2},C([0,T],\Hdot{\sigma}))} \lesssim \|\udlu_1 - \udlu_2\|_{\Linf([0,T],\H{s})}.
\end{equation*}
\end{lemma}
\begin{proof}
For $ v_0 \in \Hdot{\sigma+3/2}(\Td) $, set $ v^i = \Sol_i v_0 \in C([0,T],\Hdot{\sigma+3/2}(\Td)) $ for $ i = 1,2 $. Then $ w = v^1 - v^2 $ satisfies the equation
\begin{equation*}
(\pt - P_1^*) w = (P_1^* - P_2^*) v^2, \quad w(0) = 0,
\end{equation*}
whence the energy estimate
\begin{equation*}
\|w\|_{C([0,T],\H{\sigma})} 
\lesssim \|(P_1^* - P_2^*) v^2\|_{\Lone([0,T],\H{\sigma})}
\lesssim \|\udlu_1-\udlu_2\|_{\Linf([0,T],\H{s})} \|v_0\|_{\H{\sigma+3/2}}.
\end{equation*}
\end{proof}

\begin{lemma}
\label{lem:stability-range-op}
Suppose that $ s > 3/2 + d/2 $, $ \sigma \ge 0 $, then for~$ \varepsilon_0 $ sufficiently small,
\begin{equation*}
\|\Range_1 - \Range_2\|_{\L(\Ltwo([0,T],\Hdot{\sigma+3/2}),\Hdot{\sigma})} \lesssim \|\udlu_1 - \udlu_2\|_{\Linf([0,T],\H{s})}.
\end{equation*}
\end{lemma}
\begin{proof}
Let $ G \in \Ltwo([0,T],\Hdot{\sigma+3/2}(\Td)) $, and let $ v^i \in C([0,T],\Hdot{\sigma+3/2}(\Td))$ for $ i=1,2 $ be solutions to the equations
\begin{equation*}
 (\pt+P_i)v^i = G, \quad v^i(T) = 0. 
\end{equation*}
Then $ w = v^1-v^2 $ satisfies
\begin{equation*}
 (\pt+P_1) w = (P_2-P_1) v^2, \quad w(T) = 0, 
\end{equation*}
whence the estimate
\begin{align*}
\|(\Range_1 - \Range_2) G\|_{\H{\sigma}}
& =\|w(0)\|_{\H{\sigma}} 
 \lesssim \|(P_2-P_1) v^2\|_{\Lone([0,T],\H{\sigma})}\\
& \lesssim \|\udlu_1-\udlu_2\|_{\Linf([0,T],\H{s})} \|v^2\|_{\Lone([0,T],\H{\sigma+3/2})} \\
& \lesssim \|\udlu_1-\udlu_2\|_{\Linf([0,T],\H{s})} \|G\|_{\Ltwo([0,T],\H{\sigma+3/2})}.
\end{align*}
\end{proof}

\begin{lemma}
\label{lem:stability-K}
Suppose that $ s > 3/2 + d/2 $, $ \sigma \ge 0 $, then for~$ \varepsilon_0 $ sufficiently small,
\begin{equation*}
\|\K_1 - \K_2\|_{\L(\Hdot{\sigma+3/2},\Hdot{\sigma})} \lesssim \|\udlu_1 - \udlu_2\|_{\Linf([0,T],\H{s})}.
\end{equation*}
\end{lemma}
\begin{proof}
By the definition~\eqref{eq:definition-K} of~$ \K $, we have the identity
\begin{align*}
\K_2-\K_1
& = (\Range_1 - \Range_2) \B_1 \B_1^* \Sol_1
+ \Range_2 (\B_1-\B_2) \B_1^* \Sol_1 \\
& \qquad + \Range_2 \B_2 (\B_1^* - \B_2^*) \Sol_1
+ \Range_2 \B_2 \B_2^* (\Sol_1 - \Sol_2).
\end{align*}
And we conclude by Lemma~\ref{lem:stability-range-op}, Lemma~\ref{lem:stability-B-B^*} and Lemma~\ref{lem:stability-of-solution-operator}.
\end{proof}

\begin{lemma}
\label{lem:stability-Theta}
Suppose that~$ s $ is sufficiently large, $ \sigma \ge 0 $, then for~$ \varepsilon_0 $ sufficiently small,
\begin{equation*}
\|\Theta_1 - \Theta_2\|_{\L(\Hdot{\sigma+3/2}, C([0,T],\Hdot{\sigma}))} \lesssim \|\udlu_1 - \udlu_2\|_{\Linf([0,T],\H{s})}.
\end{equation*}
\end{lemma}
\begin{proof}
By the invertibility of~$ \K $ (Proposition~\ref{prop:K-is-H^s_h-isomorphism}) and the definition~\eqref{eq:definition-Theta} of~$ \Theta $, we have the following identity
\begin{equation*}
\Theta_1 - \Theta_2
= (\B^*_2 - \B^*_1) \Sol_2 \K_2^{-1}
+ \B^*_1 (\Sol_2 - \Sol_1) \K_2^{-1}
+ \B^*_1 \Sol_1 \K_2^{-1} (\K_1 - \K_2) \K_1^{-1}.
\end{equation*}
And we conclude by Lemma~\ref{lem:stability-B-B^*}, Lemma~\ref{lem:stability-of-solution-operator} and Lemma~\ref{lem:stability-K}.
\end{proof}

\begin{lemma}
\label{lem:stability-sol-theta}
Suppose that $ s $ is sufficiently large, $ \sigma \ge 0 $, then for~$ \varepsilon_0 $ sufficiently small,
\begin{equation*}
\|(\Sol_\Theta)_1 - (\Sol_\Theta)_2 \|_{\L(\Hdot{\sigma+3/2},C([0,T],\Hdot{\sigma}))} \lesssim \|\udlu_1 - \udlu_2\|_{\Linf([0,T],\H{s})}.
\end{equation*}
\end{lemma}
\begin{proof}
For $ v_0 \in \Hdot{\sigma+3/2}(\Td) $, let $ v^i = (\Sol_\Theta)_i v_0 \in C([0,T],\Hdot{\sigma+3/2}(\Td)), (i=1,2) $, then $ w=v^1-v^2 $ satisfies the equation
\begin{equation*}
(\pt + P_1) w = (\B_1\Theta_1-\B_2\Theta_2)v_0 + (P_2-P_1)v^2, \quad w(0)=0.
\end{equation*}
Therefore, by an energy estimate,
\begin{align*}
\| v^1 - v^2 \|_{\Ltwo([0,T],\H{\sigma})}
& \lesssim  \|(\B_1\Theta_1-\B_2\Theta_2)v_0\|_{\Lone([0,T],\H{\sigma})} 
 + \|(P_2-P_1)v_2\|_{\Lone([0,T],\H{\sigma})} \\
& \lesssim \|\udlu_1 - \udlu_2\|_{\Linf([0,T],\H{s})} \|v_0\|_{\H{\sigma+3/2}},
\end{align*}
where we have used Lemma~\ref{lem:stability-P-P^*}, Lemma~\ref{lem:stability-B-B^*} and Lemma~\ref{lem:stability-Theta}.
\end{proof}

\begin{lemma}
\label{lem:stability-R}
Suppose $ s > 2 + d/2 $, then for $ \varepsilon_0 $ sufficiently small,
\begin{equation*}
\|R_1 - R_2\|_{\L(\Hdot{s},\Hdot{s-3/2})} \lesssim \|\udlu_1 - \udlu_2\|_{\H{s-3/2}}.
\end{equation*}
\end{lemma}
\begin{proof}
The main difficulty is to estimate the $ R_G(\eta_1) - R_G(\eta_2) $, as by the definition of $ R $, the other terms in~$ R $ can easily be estimated by the estimates of paradifferential calculus. By Lemma~6.8 of~\cite{ABZ-Surface-Tension},
\begin{equation*}
\|\drv_\eta (R_G(\eta) \psi)\cdot \delta\eta\|_{\H{s-1}}
\lesssim \|\delta\eta\|_{\H{s-1}}
\end{equation*}
for $ (\psi,\eta) \in \H{s}(\Td) \times \H{s+1/2}(\Td) $ and $ \delta\eta \in \H{s+1/2}(\Td) $, from which
\begin{equation*}
\|R_G(\eta_1) \psi - R_G(\eta_2) \psi\|_{\H{s-1}}
\lesssim \|\eta_1 - \eta_2\|_{\H{s-1}} \|\nabla_x \psi\|_{\H{s-1}}
\lesssim \|\udlu_1 - \udlu_2\|_{\H{s-3/2}} \|\nabla_x \psi\|_{\H{s-1}}.
\end{equation*}
\end{proof}

\begin{lemma}
\label{lem:stability-Range-tilde}
Suppose that $ s > 2 + d/2 $, then for~$ \varepsilon_0 $ sufficiently small,
\begin{equation*}
\|\tilde{\Range}_1 - \tilde{\Range}_2\|_{\L(\Ltwo([0,T],\Hdot{s}),\Hdot{s-3/2})} \lesssim \|\udlu_1 - \udlu_2\|_{\Linf([0,T],\H{s-3/2})}.
\end{equation*}
\end{lemma}
\begin{proof}
Let $ G \in \Ltwo([0,T],\Hdot{s}(\Td)) $, and let $ v^i \in C([0,T],\Hdot{s}(\Td))$ for $ i=1,2 $ be solutions to the equations
\begin{equation*}
 (\pt+P_i+R_i)v^i = G, \quad v^i(T) = 0. 
\end{equation*}
Then $ w = v^1-v^2 $ satisfies
\begin{equation*}
 (\pt+P_1+R_1) w = (P_2-P_1) v^2 + (R_2-R_1) v^2, \quad w(T) = 0, 
\end{equation*}
whence the estimate, by Lemma~\ref{lem:stability-P-P^*}, Lemma~\ref{lem:stability-R} and Proposition~\ref{prop:app-wellposedness-paradiff-eq},
\begin{align*}
\|(\tilde{\Range}_1 - \tilde{\Range}_2) G\|_{\H{s-3/2}}
& =\|w(0)\|_{\H{s-3/2}}  \\
& \lesssim \|(P_2-P_1) v^2\|_{\Lone([0,T],\H{s-3/2})} + \|(R_2-R_1) v^2\|_{\Lone([0,T],\H{s-3/2})} \\
& \lesssim \|\udlu_1-\udlu_2\|_{\Linf([0,T],\H{s-3/2})} \|v^2\|_{\Lone([0,T],\H{s})} \\
& \lesssim \|\udlu_1-\udlu_2\|_{\Linf([0,T],\H{s-3/2})} \|G\|_{\Ltwo([0,T],\H{s})}.
\end{align*}
\end{proof}

\begin{lemma}
\label{lem:stability-B(eta)}
Suppose $ s > 2 + d/2 $, then for $ \varepsilon_0 $ sufficiently small,
\begin{equation*}
\|B(\eta_1) - B(\eta_2)\|_{\L(\Hdot{s},\Linf)} \lesssim \|\udlu_1-\udlu_2\|_{\H{s-3/2}}. 
\end{equation*}
\end{lemma}
\begin{proof}
This is merely a consequence of the contraction property for the Dirichlet-Neumann operator, see Theorem 5.2 of~\cite{ABZ-Without-Surface-Tension},
\begin{equation*}
\| (G(\eta_1) - G(\eta_2)) \psi \|_{\H{s-2}} 
\lesssim  \|\eta_1-\eta_2\|_{\H{s-1}} \|\nabla_x\psi\|_{\H{s-1}}.
\end{equation*}
\end{proof}

\begin{lemma}
\label{lem:stability-beta}
Suppose $ s > 2 + d/2 $, then for $ \varepsilon_0 $ sufficiently small,
\begin{equation*}
\|\beta_1 - \beta_2\|_{\L(\Hdot{s},\Hdot{s-1})} \lesssim \|\udlu_1 - \udlu_2\|_{\H{s-3/2}}.
\end{equation*}
\end{lemma}
\begin{proof}
It suffices to write
\begin{align*}
\beta_1 F - \beta_2 F
& = \T{q_1-q_2}\T{B(\eta_1) F}\T{p_1}^{-1}\Im \udlu_1
+ \T{q_2}\T{(B(\eta_1)-B(\eta_2)) F}\T{p_1}^{-1}\Im \udlu_1 \\
& \qquad + \T{q_2}\T{B(\eta_2) F}(\T{p_1}^{-1}-\T{p_2}^{-1})\Im \udlu_1
+ \T{q_2}\T{B(\eta_2) F}\T{p_2}^{-1}\Im (\udlu_1-\udlu_2),
\end{align*}
and conclude by Theorem~\ref{thm:Operator-Norm-Estimate-Paradiff}, Lemma~\ref{lem:stability-u-eta} and Lemma~\ref{lem:stability-B(eta)}.
\end{proof}

\begin{lemma}
\label{lem:stability-E}
Suppose that~$ s $ is sufficiently large, then for $ \varepsilon_0 $ sufficiently small,
\begin{equation*}
\|\E_1 - \E_2\|_{\L(\Hdot{s},\Hdot{s-3/2})} 
\lesssim \|\udlu_1 - \udlu_2\|_{\Linf([0,T],\H{s-3/2})}.
\end{equation*}
\end{lemma}
\begin{proof}
We use the identity,
\begin{align*}
\E_1 - \E_2
& = (\tilde{\Range}_1 - \tilde{\Range}_2)(-R_1 (\Sol_{\Theta})_1 + \beta_1 \Theta_1|_{\Hdot{s}}) \\ 
& \qquad + \tilde{\Range}_2 (R_2-R_1) (\Sol_{\Theta})_1 + \tilde{\Range}_2 R_2 \big((\Sol_{\Theta})_2 - (\Sol_{\Theta})_1  \big) \\
& \qquad \qquad + \tilde{\Range}_2 (\beta_1 - \beta_2) \Theta_1|_{\Hdot{s}} + \tilde{\Range}_2 \beta_2 (\Theta_1|_{\Hdot{s}} - \Theta_2|_{\Hdot{s}}).
\end{align*}
And conclude by Lemma~\ref{lem:stability-Range-tilde}, Lemma~\ref{lem:stability-R}, Lemma~\ref{lem:stability-sol-theta}, Lemma~\ref{lem:stability-beta} and Lemma~\ref{lem:stability-Theta}.
\end{proof}

\begin{proposition}
\label{prop:Stability-of-Phi}
Suppose that~$ s $ is sufficiently large, then for~$ \varepsilon_0 $ sufficiently small,
\begin{equation*}
\|\Phi_1 - \Phi_2\|_{\L(\Hdot{s}, C([0,T],\Hdot{s-3/2}))} 
\lesssim \|\udlu_1 - \udlu_2\|_{\Linf([0,T],\H{s-3/2})}.
\end{equation*}
\end{proposition}
\begin{proof}
We use the identity,
\begin{equation*}
\begin{split}
\Phi_1 - \Phi_2
= (\Theta_1|_{\Hdot{s}} - \Theta_2|_{\Hdot{s}}) (\Id + \E_1)^{-1} + \Theta_2|_{\Hdot{s}} (\Id + \E_1)^{-1} (\E_2 - \E_1) (\Id + \E_2)^{-1}.
\end{split}
\end{equation*}
And conclude by Lemma~\ref{lem:stability-Theta} and Lemma~\ref{lem:stability-E}.
\end{proof}

\section{Iterative Scheme}
\label{sec:iteration}

We adapt the iterative scheme of~\cite{ABH-Control} to construct a solution for the nonlinear control problem~\eqref{eq:equation-control-of-u} and thus close the proof of Theorem~\ref{thm:main-paradiff}. We have to be careful about the constants and will no longer use the notation~$ \lesssim $ in this section.

Suppose that~$ s $ is sufficiently large, fix $ 0 < \varepsilon_0 < 1 $ sufficiently small, such that the results of previous sections apply. Let $ C > 10 $ be a constant such that the following conditions are satisfied.
\begin{itemize}[noitemsep]
\item For all $ \udlu \in \Cls^{1,s}(T,\varepsilon_0) $, and all~$ \sigma \in \R $ with $ s \ge \sigma \ge 0 $,
\begin{equation*}
\begin{split}
& \|\Phi(\udlu)\|_{\L(\Hdot{s},C([0,T],\Hdot{s}))}  + \|P(\udlu)\|_{\Linf([0,T],\L(\Hdot{\sigma}, \Hdot{\sigma-3/2}))} \\
& \qquad + \|\B(\udlu)\|_{\Linf([0,T],\L(\Hdot{\sigma}, \Hdot{\sigma}))} + \|\beta(\udlu)\|_{\Linf([0,T],\L(\Hdot{s}, \Hdot{s+1/2}))}
 \le C.
\end{split}
\end{equation*}
\item For $ s \ge \mu \ge s-3/2 $, $ G \in \Linf([0,T],\Hdot{\mu}(\Td)) $, if~$ u $ satisfies the equation
\begin{equation*}
(\pt + P(\udlu) + R(\udlu)) u = G, \quad u(0) = 0 \mathrm{\ or\ } u(T) = 0,
\end{equation*}
then we have the energy estimate (by Proposition~\ref{prop:app-wellposedness-paradiff-eq}),
\begin{equation}
\label{eq:energy-estimate-iterative-scheme}
\|u\|_{C([0,T],\H{\mu})} + \|\pt u\|_{\Linf([0,T],\H{\mu-3/2})}  \le C \|G\|_{\Linf([0,T],\H{\mu})}.
\end{equation}
\item In all the lemmas and propositions of Section~\ref{sec:contraction-estimate}, the statements remain valid after replacing the relations $ \mathrm{l.h.s} \lesssim \mathrm{r.h.s} $ with $ \mathrm{l.h.s} \le C \times \mathrm{r.h.s} $.
\end{itemize}

Now fix $ K > C^{10} $, and let $ u_0 \in \Hdot{s}(\Td) $ be such that
\begin{equation*}
\|u_0\|_{\H{s}} < \varepsilon_0 / K.
\end{equation*}
We will define a sequence of functions $ \{u^n, F^n\}_{n \ge 0} $ by inductively solving a sequence a linear control problems. Once~$ u^n $ is defined, for any operator $ \mathcal{L}(u^n) $ that depends on~$ u^n $, we denote for simplicity $ \mathcal{L}_n = \mathcal{L}(u^n) $. The induction proceeds as follows. Let $ u^0 = F^0 \equiv 0 $, and for $ n \ge 0 $, let
\begin{equation*}
F^{n+1} = \Phi_{n}(u_0) \in C([0,T],\Hdot{s}(\Td)),
\end{equation*}
and let $ u^{n+1} \in C([0,T],\Hdot{s}(\Td)) $ be the solution to the equation
\begin{equation}
\label{eq:equation-u^n-iteration}
(\pt + P_{n} + R_{n}) u^{n+1} = (\B_{n} + \beta_n) F^{n+1}, \quad u^{n+1}(0) = u_0, \quad u^{n+1}(T) = 0.
\end{equation}
In order for $ F^{n+1} $ to be well defined, we have to verify that $ u^n \in \Cls^{1,s}(T,\varepsilon_0) $, this is justified by the following lemma.

\begin{lemma}
The sequence of functions $ \{u^n,F^n\}_{n\ge 0} $, formally defined as above, is well defined, and satisfies furthermore the following estimates, for all $ n \ge 0 $,
\begin{align*}
u^n  &\in \Cls^{1,s}(T,\varepsilon_0), &
u^{n+1}-u^n  &\in \Cls^{1,s-3/2}\big(T,\varepsilon_0^{n+1}\big), \\
F^{n}  &\in \Cls^{0,s}(T,\varepsilon_0), &
F^{n+1}-F^n  &\in \Cls^{0,s-3/2}\big(T,\varepsilon_0^{n+1}\big).
\end{align*}
\end{lemma}
\begin{proof}
In the following estimates, we keep in mind that $ C^{10} K^{-1} < 1 $. For $ n = 0 $, the estimate for~$ u^0 $ and~$ F^0 $ are clearly satisfied. As for the differences, we use $ F^1 = \Theta_0(u_0) $, Proposition~\ref{prop:definition-of-Phi}, and the energy estimate~\eqref{eq:energy-estimate-iterative-scheme}
\begin{align*}
\|F^1-F^0\|_{C([0,T],\H{s-3/2})}
& = \|F^1\|_{C([0,T],\H{s-3/2})} 
\le C \|u_0\|_{\H{s}}
\le C K^{-1} \varepsilon_0. \\
\|u^1-u^0\|_{C([0,T],\H{s-3/2})}
& = \|u^1\|_{C([0,T],\H{s-3/2})}
\le C \|F^1\|_{C([0,T],\H{s-3/2})}
\le C^2 K^{-1} \varepsilon_0.
\end{align*}
Then by the equation and~\eqref{eq:energy-estimate-iterative-scheme},
\begin{equation*}
\|\pt(u^1-u^0)\|_{\Linf([0,T],\H{s-3})}
= \|\pt u^1\|_{\Linf([0,T],\H{s-3})} 
\le C^2 \|F^1\|_{C([0,T],\H{s-3/2})} 
\le C^3 K^{-1} \varepsilon_0.
\end{equation*}
Suppose by now the estimates are proven for~$ n $. By Proposition~\ref{prop:definition-of-Phi},
\begin{equation*}
\|F^{n+1}\|_{C([0,T],\H{s})} \le C \|u_0\|_{\H{s}} \le C K^{-1} \varepsilon_0,
\end{equation*}
and then by~\eqref{eq:energy-estimate-iterative-scheme},
\begin{equation}
\label{eq:finer-estimate-for-u^n}
\|u^{n+1}\|_{C([0,T],\H{s})} + \|\pt u^{n+1}\|_{\Linf([0,T],\H{s-3/2})} 
 \le C^2 \|F^{n+1}\|_{C([0,T],\H{s})} \le C^3 K^{-1} \varepsilon_0. \\
\end{equation}
For the difference, we use Proposition~\ref{prop:Stability-of-Phi},
\begin{equation*}
\|F^{n+1} - F^n\|_{C([0,T],\H{s-3/2})} 
\le C \|u^n-u^{n-1}\|_{C([0,T],\H{s-3/2})} \|u_0\|_{\H{s}} 
\le C K^{-1} \varepsilon_0^{n+1}.
\end{equation*}
Observe that $ (\delta u)_n \bydef u^{n+1}-u^n $ satisfies the equation
\begin{equation*}
\begin{split}
(\pt + P_{n} + R_{n}) (\delta u)_n & = - (P_{n} - P_{n-1}) u^n - (R_n - R_{n-1}) u^n  \\
& \qquad + (\B_n F^{n+1} - \B_{n-1} F^n) + (\beta_n F^{n+1} - \beta_{n-1} F^n),  
\end{split}
\end{equation*}
with $ (\delta u)_n(0) = (\delta u)_n(T) = 0 $. By~\eqref{eq:energy-estimate-iterative-scheme},
\begin{align*}
\|& u^{n+1} - u^n\|_{C([0,T],\H{s-3/2})} + \|\pt (u^{n+1}-u^n)\|_{\Linf([0,T],\H{s-3})} \\
& \le C \|(P_n-P_{n-1}) u^n\|_{\Linf([0,T],\H{s-3/2})}
+ C \|(R_n-R_{n-1}) u^n\|_{\Linf([0,T],\H{s-3/2})} \\
& \quad + C \|\B_n F^{n+1} - \B_{n-1} F^n\|_{\Linf([0,T],\H{s-3/2})} 
+ C \|\beta_n F^{n+1} - \beta_{n-1} F^n\|_{\Linf([0,T],\H{s-3/2})},
\end{align*}
where, by Lemma~\ref{lem:stability-P-P^*}, and~\eqref{eq:finer-estimate-for-u^n},
\begin{align*}
\|(P_n-P_{n-1}) u^n\|_{\Linf([0,T],\H{s-3/2})}
& \le C \|u^n-u^{n-1}\|_{\Linf([0,T],\H{s-3/2})} \|u^n\|_{\Linf([0,T],\H{s})} \\
& \le C \times \varepsilon_0^{n} \times C^3 K^{-1} \varepsilon_0
\le C^4 K^{-1} \varepsilon_0^{n+1}.
\end{align*}
The same estimate, using Lemma~\ref{lem:stability-R}, gives
\begin{align*}
\|(R_n-R_{n-1}) u^n\|_{\Linf([0,T],\H{s-3/2})}
& \le C \|u^n-u^{n-1}\|_{\Linf([0,T],\H{s-3/2})} \|u^n\|_{\Linf([0,T],\H{s})} \\
& \le C \times \varepsilon_0^{n} \times C^3 K^{-1} \varepsilon_0
\le C^4 K^{-1} \varepsilon_0^{n+1}.
\end{align*}
Similarly, by Lemma~\ref{lem:stability-B-B^*} and Lemma~\ref{lem:stability-beta}, and the triangular inequality, we show that
\begin{equation*}
\|\B_n F^{n+1} - \B_{n-1} F^n\|_{\Linf([0,T],\H{s-3/2})} 
+ \|\beta_n F^{n+1} - \beta_{n-1} F^n\|_{\Linf([0,T],\H{s-3/2})}
\le C^3 K^{-1} \varepsilon_0^{n+1}.
\end{equation*}
In summary, we have
\begin{equation*}
\|u^{n+1} - u^n\|_{C([0,T],\H{s-3/2})} + \|\pt (u^{n+1}-u^n)\|_{\Linf([0,T],\H{s-3})} 
\le C^6 K^{-1} \varepsilon_0^{n+1},
\end{equation*}
which closes the proof.
\end{proof}

\begin{corollary}
\label{cor:Cauchy-(u,F)}
$ \{u^n,F^n\}_{n \ge 0} $ is a Cauchy sequence in 
\begin{equation*}
C([0,T],\Hdot{s-3/2}(\Td)) \cap \Holder{1}([0,T],\Hdot{s-3}(\Td)) \times C([0,T],\Hdot{s-3/2}(\Td)),
\end{equation*}
whose limit will be denoted by
\begin{equation*}
(u,F) = \lim_{n \to \infty} (u^n,F^n) 
\in \Cls^{1,s-3/2}(T,\varepsilon_0) \times \Cls^{0,s-3/2}(T,\varepsilon_0).
\end{equation*}
\end{corollary}

\begin{corollary}
The following convergence holds,
\begin{enumerate}[nosep]
\item $ P(u^n)u^{n+1} \to P(u)u $, strongly in $ C([0,T],\Hdot{s-3}(\Td)) $;
\item $ R(u^n)u^{n+1} \to R(u)u $, strongly in $ C([0,T],\Hdot{s-3}(\Td)) $, weakly in $ \Ltwo([0,T],\H{s}(\Td)) $;
\item $ \B(u^n) F^{n+1} \to \B(u) F $, strongly in $ C([0,T],\Hdot{s-3/2}(\Td)) $, weakly in $ \Ltwo([0,T],\H{s}(\Td)) $;
\item $ \beta(u^n) F^{n+1} \to \beta(u) F $, strongly in $ C([0,T],\Hdot{s-3/2}(\Td)) $, weakly in $ \Ltwo([0,T],\H{s}(\Td)) $.
\end{enumerate}
In particular, $ (u,F) $ satisfies the equation
\begin{equation}
\label{eq:equation-u-limit}
(\pt + P(u) + R(u)) u = ( \B(u) + \beta(u) )F, \quad u(0) = u_0, \quad u(T) = 0,
\end{equation}
in sense of distribution. Moreover,
\begin{equation*}
\|R(u)u\|_{\Ltwo([0,T],\H{s})} + \|\B(u) F\|_{\Ltwo([0,T],\H{s})} + \|\beta(u) F\|_{\Ltwo([0,T],\H{s})} \lesssim \varepsilon_0.
\end{equation*}
\end{corollary}
\begin{proof}
The convergences follow directly from our construction, Corollary~\ref{cor:Cauchy-(u,F)}, the triangular inequality, and Lemma~\ref{lem:stability-P-P^*}, Lemma~\ref{lem:stability-R}, Lemma~\ref{lem:stability-B-B^*}, Lemma~\ref{lem:stability-beta} and~\eqref{eq:estimate-various-operators}.
Therefore we pass to the limit $ n\to\infty $ in~\eqref{eq:equation-u^n-iteration}, and obtain~\eqref{eq:equation-u-limit}. The last estimate comes from the weak convergence.
\end{proof}

Now we rewrite~\eqref{eq:equation-u-limit} as 
\begin{equation}
\label{eq:nonlinear-linear-rewrite}
(\pt + P(\udlu)) u = G, \quad u(0) = u_0,\ u(T) = 0,
\end{equation}
where $ \udlu = u $, and $ G = - R(u)u + \B(u) F + \beta(u) F $ with $ \|G\|_{\Ltwo([0,T],\H{s})} \lesssim \varepsilon_0 $. By Proposition~\ref{prop:app-wellposedness-paradiff-eq}, the equation~\eqref{eq:nonlinear-linear-rewrite} admits a unique solution in $ C([0,T],\Hdot{s}(\Td)) $. Therefore, $ u \in C([0,T],\Hdot{s}(\Td)) $ with the energy estimate
\begin{equation*}
\|u\|_{C([0,T],\H{s})} \lesssim \|u_0\|_{\H{s}} + \|G\|_{\Ltwo([0,T],\H{s})} \lesssim \varepsilon_0.
\end{equation*}
That is, $ u \in \Cls^{0,s}(T,C'\varepsilon_0) $ for some $ C' > 0 $. Therefore, for~$ \varepsilon_0 $ sufficiently small, by~\eqref{eq:estimate-various-operators}, $ G \in \Linf([0,T],\Hdot{s-3/2}(\Td)) $, with $ \|G\|_{\Linf([0,T],\H{s-3/2})} \lesssim \varepsilon_0 $. This implies that
\begin{equation*}
\|\pt u\|_{\Linf([0,T],\H{s-3/2})}
\lesssim \|u_0\|_{\H{s}} + \|G\|_{\Linf([0,T],\H{s-3/2})} \lesssim \varepsilon_0.
\end{equation*}
Consequently $ u \in \Cls^{1,s}(T,C''\varepsilon_0) $ for some $ C'' > 0 $. When $ \varepsilon_0 $ is sufficiently small, $ \Phi(u) $ is well defined, and as $ n \to \infty $,
\begin{equation*}
\|\Phi(u) - \Phi(u^n)\|_{\L(\Hdot{s}, C([0,T],\Hdot{s-3/2}))}
\lesssim \|u - u^n\|_{C([0,T],\H{s-3/2})} \to 0.
\end{equation*}
Therefore in the sense of distribution,
\begin{equation*}
F = \lim_{n\to\infty} F^n = \lim_{n\to\infty} \Phi(u^n) u_0 = \Phi(u) u_0 \in C([0,T],\Hdot{s}(\Td)).
\end{equation*}
This finishes the proof of Theorem~\ref{thm:main-paradiff}.

\section{Theorem~\ref{thm:main-paradiff} Implies Theorem~\ref{thm:main}}

\label{sec:theorem-paradiff-implies-theorem}

Now we deduce Theorem~\ref{thm:main} from Theorem~\ref{thm:main-paradiff}. Observe that the null controllability holds for the time reversed equation (that is, the equation obtained by the change of variable $ t \mapsto -t $) of~\eqref{eq:intro-equation-paradiff-nonlinear} as well, with the same proof. Let $ (\eta_i,\psi_i)\  (i=0,1) $ satisfy the hypotheses of Theorem~\ref{thm:main}, and let $ u_i = u(\eta_i,\psi_i) \in \Hdot{s}(\Td) $ be defined by~\eqref{eq:def-u}. Then $ \|u_i\|_{\H{s}} \lesssim \varepsilon_0 $, and for $ \varepsilon_0 $ sufficiently small, there exist $ \dot{F}^i \in C([0,T/2],\Hdot{s}(\Td)) $ by Theorem~\ref{thm:main-paradiff}, such that $ \dot{F}^0 $ sends initial data~$ u_0 $ at time $ t=0 $ to final data~$ 0 $ at time~$ t=T/2 $ by~\eqref{eq:intro-equation-paradiff-nonlinear}, while $ \dot{F}^1 $ sends initial data~$ u_1 $ at time~$ t=0 $ to final data~$ 0 $ at time~$ t=T/2 $ by the time reversed equation of~\eqref{eq:intro-equation-paradiff-nonlinear}. Moreover, the estimates~\eqref{eq:soft-estimate-u-F} are verified by~$ \dot{F}_i $ and the corresponding solutions $ u^i \in C([0,T/2],\Hdot{s}(\Td)) $. Now define $ u \in C([0,T],\Hdot{s}(\Td)) $ and $ \dot{F} \in C([0,T],\Hdot{s}(\Td)) $ by
\begin{equation*}
u(t)  = \begin{cases}
u^0(t), & t \in [0,T/2], \\
u^1(T-t), & t \in (T/2,T].
\end{cases} 
\quad
\dot{F}(t)  = \begin{cases}
\chi_{T/2}^{}(t) \dot{F}^0(t), & t \in [0,T/2], \\
\chi_{T/2}^{}(T-t) \dot{F}^1(T-t), & t \in (T/2,T].
\end{cases}
\end{equation*}
Indeed~$ \dot{F} $ is continue in time, for the truncation function vanishes near $ t = T/2 $. Then~$ u $ satisfies~\eqref{eq:intro-equation-paradiff-nonlinear} (with~$ F $ replaced by~$ \dot{F} $), and $ u(0) = u_0 $, $ u(T) = u_1 $. By Proposition~\ref{prop:reversibility-u-eta-psi}, we can find
\begin{align*}
\dot{\eta} & = \dot{\eta}(u) \in C([0,T],\Hdot{s+1/2}(\Td)) \cap \Holder{1}([0,T],\Hdot{s-1}(\Td)), \\
\dot{\psi} & = \dot{\psi}(u) \in C([0,T],\Hdot{s}(\Td)) \cap \Holder{1}([0,T],\Hdot{s-3/2}(\Td)), 
\end{align*}
such that $ u = \T{q(\dot{\eta})} (\dot{\psi}-\T{B(\dot{\eta})\dot{\psi}} \dot{\eta}) -i \T{p(\dot{\eta})} \dot{\eta}$ by~\eqref{eq:identity-u-eta-psi}, and
\begin{equation*}
\dot{\eta}(0) = \eta_0,  \quad \dot{\psi}(0) = \pi(D_x) \psi_0, \quad
\dot{\eta}(T) = \eta_1,  \quad \dot{\psi}(T) = \pi(D_x) \psi_1.
\end{equation*} 
Next we look for $ (\eta,\psi,F) $ by adding zero frequencies to $ (\dot{\eta},\dot{\psi},\dot{F}) $ in such way that $ (\eta,\psi) $ satisfies~\eqref{eq:equation-water-wave}, as well as the initial and final conditions, with the exterior pressure being
\begin{equation*}
\Pext(t,x) = \phiw(x) \Re\, F(t,x).
\end{equation*}
More precisely, we look for $ (\eta, \psi, F) $ of the following form,
\begin{equation}
\label{eq:definition-eta-psi-F}
\eta(t,x) = \dot{\eta}(t,x), 
\quad \psi(t,x) = \dot{\psi}(t,x) + \frac{1}{(2\pi)^d} \alpha(t),
\quad F(t,x) = \dot{F}(t,x) + c_0,
\end{equation}
where $ c_0 \in \R $ is a constant, and $ \alpha $ is a $ C^1 $ function of $ t $. By reversing the paralinearization process, $ \pt \dot{\eta} = G(\dot{\eta}) \dot{\psi} $, hence $ \pt \eta = G(\eta) \psi $. In order for~$ \psi $ to meet the initial and final data, $ \alpha $ should satisfy the boundary conditions,
\begin{equation*}
\alpha(0) = \intTorus \psi_0(x) \dx, \quad \alpha(T) = \intTorus \psi_1(x) \dx.
\end{equation*}
Plugging~\eqref{eq:definition-eta-psi-F} into~\eqref{eq:equation-water-wave}, and integrating it over~$ \Td $, we obtain an ordinary differential equation for $ \alpha $,
\begin{equation*}
\ddt \alpha(t) + \intTorus \Big( \frac{1}{2} |\nabla_x \dot{\psi}|^2  - \frac{1}{2} \frac{(\nabla_x \dot{\eta} \cdot \nabla_x \dot{\psi} + G(\dot{\eta})\dot{\psi})^2}{1+|\nabla_x \dot{\eta}|^2} \Big) \dx  
= \intTorus \phiw \Re (\dot{F}(t,x) + c_0) \dx
\end{equation*}
Solving this equation by integrating it over $ [0,t] $ and using the initial condition for $ \alpha(0) $,
\begin{equation}
\label{eq:solution-alpha}
\begin{split}
\alpha(t) 
& = \intTorus \psi_0 \dx + \int_0^t \intTorus \phiw \Re (\dot{F} + c_0) \dx \dt   \\
& \qquad \qquad + \int_0^t \intTorus \Big( \frac{1}{2} \frac{(\nabla_x \dot{\eta} \cdot \nabla_x \dot{\psi} + G(\dot{\eta})\dot{\psi})^2}{1+|\nabla_x \dot{\eta}|^2} - \frac{1}{2} |\nabla_x \dot{\psi}|^2 \Big) \dx \dt.
\end{split}
\end{equation}
Observe that $ c_0 \mapsto \alpha(T) $ is an affine function of $ c_0 $, with the coefficient of $ c_0 $ being $ T \times \intTorus \phiw(x) \dx \ne 0 $, so there exists a unique $ c_0 \in \R $ such that $ \alpha(T) = \intTorus \psi_1(x) \dx $.

\begin{remark}
In the solution~\eqref{eq:solution-alpha} of $ \alpha $ the integrals for $ g \eta $ and $ H(\eta) $ do not appear because they have no zero frequencies. The former is due to our assumption, while the latter is by the divergence theorem. If we have infinite depth, that is, $ \b = \infty $, then the integral over $ \Td $ of $ \frac{1}{2} \frac{(\nabla_x\dot{\eta}\cdot\nabla_x\dot{\psi} + G(\dot{\eta})\dot{\psi})^2}{1+|\nabla_x\dot{\eta}|^2} - \frac{1}{2} |\nabla_x \dot{\psi}|^2 $ also vanishes. This can be proven by a direct computation using Green's identity, and is related to a conserved quantity of~\eqref{eq:equation-water-wave} by its Hamiltonian structure. For more on this subject, see Benjamin-Olver~\cite{BO-Conservation-Laws}.
\end{remark}

\appendix

\section{Necessity of Geometric Control Condition}

\label{app:necessity-of-GCC}

We prove that the geometric control condition is necessary for the controllability of the three dimensional water wave equation (that is $ d = 2 $) with infinite depth (that is $ b = \infty $) linearized around the flat surface (that is $ \eta = 0 $). We believe that similar arguments are suffice to prove the same results in arbitrary dimensions and for finite depth, however, we do not attempt to generalize the result in this direction for the sake of simplicity.

Now this linearized equation is a fractional Schr\"{o}dinger equation $ \pt u + i |D|^{3/2} u = \T{1}\Re \phiw F $.
We will consider more generally the following control problem in $ \Ltwo(\Torus^2) $,
\begin{equation}
\label{eq:equation-schrodinger-fractional}
\pt u + i |D|^{\alpha} u = B \varphi F,
\quad 1 \le \alpha \le 2.
\end{equation}
Here $ \varphi \in \Cinf(\Torus^2) $ and~$ B $ is a bounded operator on $ \Ltwo(\Torus^2) $. If $ \alpha = 1 $, we have the (half) wave equation. When $ B = \Id $, it is exactly controllable, if and only if the geometric control condition is satisfied, see~\cite{BLR,BG-control-wave}. If $ \alpha = 2 $, we have the Schr\"{o}dinger equation. When $ B = \Id $, it is always exactly controllable (on tori) whether or not under the geometric control condition, see~\cite{Haraux,Jaffard-Control-Plaque,Anantharaman-Macia-SM-Schrodinger,BZ-Control-Schrodinger,BBZ-control-schrodinger}. We are now in the middle of the two typical cases $ 1 < \alpha < 2 $ where we show that the geometric control condition is necessary to exactly control~\eqref{eq:equation-schrodinger-fractional} on~$ \Torus^2 $.

\begin{definition}
We say that~\eqref{eq:equation-schrodinger-fractional} is exactly controllable, if there exists $ T > 0 $, such that for all $ u_0, u_1 \in \Ltwo(\Torus^2) $, there exists $ F \in C([0,T],\Ltwo(\Torus^2)) $, and a solution $ u \in C([0,T],\Ltwo(\Torus^2)) $ to~\eqref{eq:equation-schrodinger-fractional}, satisfying $ u(0) = u_0 $, $ u(T) = u_1 $.
\end{definition}

\begin{proposition}
Suppose that $ 1 < \alpha < 2 $, and~\eqref{eq:equation-schrodinger-fractional} is exactly controllable, then
\begin{equation*}
\omega \bydef \{ z = (x,y) \in \Torus^2 : \varphi(z) \ne 0\}
\end{equation*}
satisfies the geometric control condition.
\end{proposition}
\begin{proof}
The idea is to prove by contradiction by using the following lemma due to~Burq--Zworski~\cite{BZ-black-box}, and Miller~\cite{Miller}.
\begin{lemma}
The equation~\eqref{eq:equation-schrodinger-fractional} is exactly controllable, if and only if, for some $ C > 0 $, for all $ \lambda \in \R $, and for all $ u \in \Cinf(\Torus^2) $,
\begin{equation}
\label{eq:resolvent-estimate-app}
C \|u\|_{\Ltwo} \le \|(|D_z|^\alpha - \lambda) u\|_{\Ltwo} + \|B\varphi u\|_{\Ltwo}.
\end{equation}
\end{lemma} 
We may assume that $ \varphi \not\equiv 0 $, so that $ \omega \ne \emptyset $, for the case will be trivial otherwise. Suppose that~$ \omega $ does not satisfy the geometric control condition, we will show that~\eqref{eq:resolvent-estimate-app} does not hold for any fixed $ C > 0 $. By hypothesis, modulo some necessary translation, there exists some $ \gamma \in \R^2\backslash\{0\} $ such that the geodesic $ \Gamma_\gamma = \{\gamma t : t \in \R \} $ does not enter~$ \omega $. Now that~$ \omega \ne \emptyset $, $ \Gamma_\gamma $ cannot be dense. So we may further more assume that
$ \gamma \in  \Z^2 $.

Consider $ \Gamma_\gamma $ as a Lie group acting on $ \Torus^2 $, $ \Gamma_\gamma \ni \gamma t : z \mapsto z + \gamma t $, which defines a quotient manifold $ \kappa : \Torus^2 \to \Torus^2 / \Gamma_\gamma $, $ z \mapsto z + \Gamma_\gamma $. Choose $ \delta > 0 $ sufficiently small such that 
\begin{equation*}
\|\varphi\|_{\Linf(\mathcal{N}_\delta)} < \frac{C}{2}(1+\|B\|_{\L(\Ltwo,\Ltwo)})^{-1},
\end{equation*} 
where $ \mathcal{N}_\delta = \{z \in \Torus^2 : \mathrm{dist}(z,\Gamma_\gamma) < \delta \} $. Observe that $ \kappa(\mathcal{N}_\delta) $ is open (with respect to the canonical quotient topology) for $ \kappa^{-1} (\kappa(\mathcal{N}_\delta)) = \mathcal{N}_\delta $ is open. Fix $ 0 \ne \psi \in \Ccinf(\kappa(\mathcal{N}_\delta)) \subset \Cinf(\Torus^2/\Gamma_\gamma) $, and set $ \chi = \psi\comp\kappa \in \Ccinf(\mathcal{N}_\delta) \subset \Cinf(\Torus^2) $, $ u^n(z) = e^{in\gamma\cdot z} \chi(z) $ for $ n \in \N $.

Expending~$ \chi $ in Fourier series, we write $ \chi(z) = \sum_{k\in\Z^2} c_k e^{ik\cdot z} $, and claim that $ c_k = 0 $ unless $ k \in \gamma^\perp \bydef \{\ell \in \Z^d : \ell \cdot \gamma = 0 \}$. Indeed, if $ k \notin \gamma^\perp $, then there exists $ w = \gamma t \in \Gamma_\gamma $ such that, $ k \cdot w \not\equiv 0 $ modulo~$ 2\pi $. Observe that $ \chi(z+w) = \psi(\kappa(z+w)) = \psi(\kappa(z)) = \chi(z) $, we have
\begin{equation*}
c_k = \frac{1}{4\pi^2} \int_{\Torus^2} \chi(z+w) e^{-ik\cdot z} \d z
= e^{ik\cdot w} \frac{1}{4\pi^2} \int_{\Torus^2} \chi(z) e^{-ik\cdot z} \d z
= e^{ik\cdot w} c_k,
\end{equation*}
which implies that $ c_k = 0 $. Therefore, $ u^n(z) = \sum_{k \in \gamma^\perp} c_k e^{i(n\gamma + k) \cdot z} $, and
\begin{equation*}
|D_z|^\alpha u^n 
= \sum_{k \in \gamma^\perp} c_k |n\gamma+k|^{\alpha} e^{i(n\gamma + k) \cdot z}
= \sum_{k \in \gamma^\perp} c_k (n^2|\gamma|^2+|k|^2)^{\alpha/2} e^{i(n\gamma + k) \cdot z}.
\end{equation*}
Let $ \lambda_n = n^\alpha |\gamma|^\alpha $, then
\begin{equation*}
(|D_z|^\alpha - \lambda_n) u^n
= \sum_{0 \ne k \in \gamma^\perp} c_k |k|^\alpha f\big(\frac{|k|}{n|\gamma|}\big) e^{i(n\gamma + k) \cdot z},
\end{equation*}
where $ f(t) = \frac{(1+t^2)^{\alpha/2}-1}{t^\alpha}. $ By an integration by part, for any $ N \ge 1 $, $ |k|^\alpha |c_k| \lesssim |k|^{-N}  $. Observe that~$ f $ is continuous on $ ]0,+\infty[ $, with $ \lim_{t\to+\infty} f(t) = 1 $, and $ f(t) =  O(t^{2-\alpha}) $ as $ t \to 0^+ $. We have therefore the estimate,
\begin{equation*}
f\big(\frac{|k|}{n|\gamma|}\big) \lesssim
\begin{cases}
\frac{1}{n^{(2-\alpha)/2}}, & 0 < |k| \le \sqrt{n}; \\
1, & |k| > \sqrt{n}.
\end{cases}
\end{equation*}
To conclude, we show that the sequence $ (u^n,\lambda_n) $ violates~\eqref{eq:resolvent-estimate-app}. Indeed $ \|u^n\|_{\Ltwo} = \|\chi\|_{\Ltwo} $, and $ \|B \varphi u^n\|_{\Ltwo} \le \|B\|_{\L(\Ltwo,\Ltwo)} \|\varphi\|_{\Linf(\mathcal{N}_\delta)} \|\chi\|_{\Ltwo} \le \frac{C}{2} \|\chi\|_{\Ltwo} $, and for any $ N \ge 2 $,
\begin{align*}
\|(|D_z|^\alpha - \lambda_n) u^n\|_{\Ltwo}^2
& \lesssim \frac{1}{n^{(2-\alpha)/2}} \sum_{|k| \le \sqrt{n}} \frac{1}{|k|^{N}} + \sum_{|k| > \sqrt{n}} \frac{1}{|k|^{N}} = o(1), \quad
\mathrm{as} \quad n \to \infty.
\end{align*}
\end{proof}

\section{Paradifferential Calculus}
\label{sec:paradifferential-calculus}

For results of this section, we refer to~\cite{Metivier-Paradifferential-Calculus,ABZ-non-local}.

\subsection{Paradifferential Operators}

For $ \infty \ge \rho \ge 0 $, denote by $ \Holder{\rho}(\Td) $ the space of H\"{o}lderian functions of regularity~$ \rho $ on~$ \Td $.

\begin{definition}
For $ m \in \R$, $ \rho \ge 0 $, let $ \Symbol{m}{\rho}(\Td) $ denote the space of locally bounded functions $ a(x,\xi) $ on $ \Td_x \times (\Rdxi \backslash 0) $, which are $ \Cinf $ with respect to $ \xi \in \Rd \backslash 0 $, such that for all $ \alpha \in \N^d $ and $ \xi \ne 0 $, the function $ x \mapsto \pxi^\alpha a(x,\xi) $ belongs to $ \Holder{\rho}(\Td) $, and that for some constant $ C_\alpha $,
\begin{equation*}
\|\pxi^\alpha a(\cdot,\xi)\|_{\Holder{\rho}} \le C_\alpha \langle\xi\rangle^{m - |\alpha|}, 
\quad \forall |\xi| \ge \demi,
\end{equation*}
where $ \jp{\xi} = (1 + |\xi|^2)^{1/2} $.
\end{definition}
Define on $ \Symbol{m}{\rho}(\Td) $ the semi-norms,
\begin{equation*}
\Mdot{m}{\rho}{n}{a} = \sup_{|\alpha| \le n} \sup_{|\xi| \ge 1/2} \|\langle\xi\rangle^{|\alpha| - m} \pxi^{\alpha} a(\cdot,\xi)\|_{\Holder{\rho}}.
\end{equation*}

\begin{definition}
The function $ \chi = \chi(\theta,\eta) $ is called an admissible cutoff function, if 
\begin{enumerate}[nosep]
\item $ \chi \in \Cinf(\Rd_\theta \times \Rd_\eta) $ is an even function, that is, $ \chi(-\theta,-\eta) = \chi(\theta,\eta) $;
\item it satisfies the following spectral condition: for some $ 0 < \epsilon_1 < \epsilon_2 < 1/2 $,
\begin{equation}
\label{eq:Spectral-Condition-for-Admissible-CutOff}
\begin{cases}
\chi(\theta,\eta) = 1, & |\theta| \le \epsilon_1 \langle\eta\rangle, \\
\chi(\theta,\eta) = 0, & |\theta| \ge \epsilon_2 \langle\eta\rangle;
\end{cases}
\end{equation}
\item for all $ (\alpha,\beta) \in \N^d\times\N^d $, and some $ C_{\alpha\beta} > 0 $,
\begin{equation*}
\big|\partial_\theta^\alpha\partial_\eta^\beta \chi(\theta,\eta)\big| 
\le C_{\alpha\beta} \langle\eta\rangle^{-|\alpha|-|\beta|}.
\end{equation*}
\end{enumerate}
\end{definition}

Let~$ \chi $ be an admissible cutoff function, and let $ \pi \in \Cinf(\Rd) $ be an even function such that $ 0 \le \pi \le 1 $, $ \pi(\xi) = 0 $ for $ |\xi| \le 1/4 $, and $ \pi(\xi) = 1 $ for $ |\xi| \ge 3/4 $. Now given a symbol $ a \in \Symbol{m}{\rho}(\Td) $, the para\-differential operator $ \T{a} $ is formally defined by
\begin{equation}
\label{eq:definition-paradifferential-operator}
\widehat{\T{a}{u}}(\xi) = (2\pi)^{-d} \sum_{\eta \in \Z^d} \chi(\xi - \eta,\eta) \hat{a}(\xi - \eta,\eta) \pi(\eta) \hat{u}(\eta),
\end{equation}
where $ \hat{a}(\theta,\eta) = \big(\Fourier_{x \to \theta} a\big)(\theta,\eta) = \intTorus e^{-i x \cdot \theta} a(x,\eta) \dx $. Alternatively, set 
\begin{equation*}
a^\chi(\cdot,\xi) = \chi(D_x,\xi) a(\cdot,\xi),
\end{equation*}
then by definition, $ \T{a}{u} = a^\chi(x,D_x) \pi(D_x) u $.

\begin{theorem}
\label{thm:Operator-Norm-Estimate-Paradiff}
Let $ a \in \Symbol{m}{0} $, then $ \T{a} $ is of order $ m $ such that for all $ s \in \R $, $ \T{a} $ defines a bounded operator from $ \H{s+m}(\Td) $ to $ \Hdot{s}(\Td) $, such that
\begin{equation}
\|\T{a}\|_{\L(\H{s + m}, \Hdot{s})} \lesssim \M{m}{0}{d/2+1}{a}.
\end{equation}
In particular, $ \T{a} = \T{a} \pi(D_x) = \pi(D_x) \T{a} $.
\end{theorem}
\begin{proof}
For the estimate we refer to~\cite{Metivier-Paradifferential-Calculus}. It remains to show that for $ u \in \H{s+m}(\Td) $, $ \T{a} u $ has no zero frequency, or equivalently, $ \widehat{\T{a}u}(0) = 0 $.
Indeed, by definition
\begin{equation*}
\widehat{\T{a}{u}}(0) = (2\pi)^{-d} \sum_{0 \ne \eta \in \Z^d} \chi(- \eta,\eta) \hat{a}(- \eta,\eta) \hat{u}(\eta) = 0,
\end{equation*}
since $ \chi(- \eta,\eta) = 0 $ for all $ \eta \ne 0 $ by~\eqref{eq:Spectral-Condition-for-Admissible-CutOff}.
\end{proof}

\begin{proposition}
\label{prop:T_1=pi}
For all $ s \in \R $,
$ \T{1} = \pi(D_x) = \Id_{\Hdot{s}} $.
\end{proposition}
\begin{proof}
Observe that $ \hat{a}(\theta,\eta) = 1_{\theta = 0}(\theta,\eta) $ if $ a \equiv 1 $, therefore by definition,
\begin{equation*}
\widehat{\T{1} u}(\xi) 
= \sum_{\eta \in \Z^d} \chi(\xi-\eta,\eta) 1_{\xi=\eta} \pi(\eta) \hat{u}(\eta) \\
= \chi(0,\xi) \pi(\xi) \hat{u}(\xi)
= \pi(\xi) \hat{u}(\xi),
\end{equation*}
since $ \chi(0,\xi) = 1 $ for all $ \xi \in \Z^d $.
\end{proof}

\begin{lemma}
\label{lem:Estimate-pseudo-para-difference}
Let $ a \in \Symbol{m}{\rho}(\Td) $, and $ \alpha \in \N^d $, with $ |\alpha| \le \rho $, then $ \px^\alpha (a - a^{\chi}) \in \Symbol{m - \rho + |\alpha|}{0}(\Td) $ with estimates that for all $ n \in \N $, $ \M{m - \rho + |\alpha|}{0}{n}{\px^\alpha (a - a^{\chi})} \lesssim \M{m}{\rho}{n}{a} $.
\end{lemma}
\begin{proof}
See~\cite{Metivier-Paradifferential-Calculus}.
\end{proof}

\begin{proposition}
\label{prop:unparadifferentialization}
Let $ a \in \Symbol{m}{\rho}(\Td) $ with $ m \ge 0 $ and $ \rho > m + 1 + d/2 $, then
\begin{equation*}
\| \op(a\pi) - \T{a} \|_{\L(\Ltwo,\Ltwo)} \lesssim \M{m}{\rho}{d/2+1}{a}.
\end{equation*}
\end{proposition}
\begin{proof}
By the Calder\'{o}n--Vaillancourt Theorem, it suffices to show that 
\begin{equation*}
\M{0}{d/2+1}{d/2+1}{a-a^\chi} \lesssim \M{m}{\rho}{d/2+1}{a}. 
\end{equation*}
Indeed, for $ |\alpha| \le d/2 + 1 $, by the previous lemma,
\begin{equation*}
\M{0}{0}{d/2+1}{\px^\alpha (a-a^\chi)}
\lesssim \M{m-\rho+|\alpha|}{0}{d/2+1}{\px^\alpha (a-a^\chi)}
\lesssim \M{m}{\rho}{d/2+1}{a}.
\end{equation*}
\end{proof}

\begin{lemma}
\label{lem:real-part-preserving-property}
Let $ a \in S^m(\Td) $ be in the H\"{o}rmander class. If it is either a real valued even function of~$ \xi $, or a pure imaginary valued odd function of~$ \xi $, then for $ u \in \Cinf(\Td,\C) $, 
\begin{equation*}
\op(a) \Re\, u = \Re\, \op(a) u, \qquad \T{a} \Re\, u = \Re\, \T{a} u.
\end{equation*}
\end{lemma}
\begin{proof}
We first prove the case of pseudo\-differential operators $ \op(a) $. Let $ \tilde{a}(x,\xi) = \overline{a(x,-\xi)} $, then by our hypothesis $ a = \tilde{a} $. Therefore, for any real valued function~$ u $,
\begin{equation*}
\op(a) u = \overline{\op(\tilde{a}) \bar{u}} = \overline{\op(a) u},
\end{equation*}
which implies that $ \op(a) u $ is real valued. Then for a complex function~$ u $,
\begin{equation*}
\op(a) u = \op(a) (\Re\, u + i \Im\, u) = \op(a) \Re\, u + i \op(a) \Im\, u.
\end{equation*}
We conclude that $ \Re\, \op(a) u = \op(a) \Re\, u $, $ \Im\,\op(a) u = \op(a) \Im\, u $.

As for the para\-differential case, $ \T{a} = \op(a^\chi\pi) $ with $ (a^\chi \pi)(x,\xi) = \chi(D_x,\xi) a(\cdot,\xi) $. Now that~$ \chi $ is even, by the pseudo\-differential case, $ \chi(D_x,\xi) $ commute with $ \Re $. This implies that $ a^\chi\pi $ remains to be a real symbol, and an even function of~$ \xi $, or a pure imaginary symbol and an odd function of~$ \xi $. So the case of para\-differential operators follows.
\end{proof}

\subsection{Symbolic Calculus}

\begin{theorem}
\label{thm:Symbolic-Calculus-Paradiff-Composition}
Let $ a \in \Symbol{m}{\rho}(\Td) $, $ b \in \Symbol{m'}{\rho}(\Td) $, with $ m, m' \in \R $, $ 0 \le \rho <\infty $. Set
\begin{equation*}
a \sharp b = \sum_{|\alpha| < \rho} \frac{1}{\alpha!} \pxi^\alpha a D_x^\alpha b.
\end{equation*}
Then $ \T{a} \T{b} - \T{a \sharp b} $ is of order $ m + m' - \rho $, and
\begin{equation*}
\| \T{a} \T{b} - \T{a \sharp b}{} \|_{ \L(\H{s + m + m' - \rho}, \H{s})} 
\lesssim \M{m}{0}{d/2+1+\rho}{a} \M{m'}{\rho}{d/2+1}{b}
+ \M{m'}{0}{d/2+1+\rho}{b} \M{m}{\rho}{d/2+1}{a}.
\end{equation*}
\end{theorem}

\begin{theorem}
\label{thm:Symbolic-Calculus-Paradiff-Adjoint}
Let $ a \in \Symbol{m}{\rho}(\Td) $, with $ m \in \R $ and $ 0 \le \rho < \infty $. Set
\begin{equation*}
a^* = \sum_{|\alpha|<\rho} \frac{1}{\alpha!} \pxi^\alpha D_x^\alpha \bar{a}.
\end{equation*}
Denote by $ \T{a}^* $ the formal adjoint of $ \T{a} $, then $ \T{a}^* - \T{a^*} $ is of order $ m - \rho $, and
\begin{equation*}
\| \T{a}^* - \T{a^*} \|_{\L(\H{s+m-\rho},\H{s})}
\lesssim \M{m}{\rho}{d/2+1+\rho}{a}.
\end{equation*}
\end{theorem}

\subsection{Paraproducts and Paralinearization}

\begin{theorem}
\label{thm:Estimate-Paraproduct-Paralinearization}
Let $ a \in \H{\alpha}(\Td) $ and $ b \in \H{\beta}(\Td) $ with $ \alpha>d/2 $ , $\beta>d/2 $. Then
\begin{enumerate}[nosep]
\item $ T_aT_b-T_{ab} $ is of order $ -\rho $ with $ \rho = \min\{\alpha,\beta\} - d/2 $, that is, for $ s \in \R $,
\begin{equation}
\label{eq:composition-for-paraproduct}
\|\T{a}\T{b}-\T{ab}\|_{\L(\H{s-\rho},\H{s})}
\lesssim \|a\|_{\H{\alpha}} \|b\|_{\H{\beta}};
\end{equation}
\item $ \T{a}^* - \T{\bar{a}} $ is of order $ -\rho $ with $ \rho = \alpha - d/2 $, that is, for $ s \in \R $,
\begin{equation*}
\| T_a^* - T_{\bar{a}} \|_{\L(\H{s-\rho},\H{s})}
\lesssim \| a \|_{\H{\alpha}}; 
\end{equation*}
\item Define the bilinear form,
\begin{equation}
\label{eq:remainder-of-paraproduct}
R(a,b) = ab - \T{a}{b} - \T{b}{a},
\end{equation}
then $ R(a,b) \in \H{\alpha + \beta - d/2}(\Td) $,
\begin{equation*}
\| R(a,b) \|_{\H{\alpha+\beta-d/2}} 
\lesssim \| a \|_{\H{\alpha}} \| b \|_{\H{\beta}}; 
\end{equation*}
\item Let $ F \in \Cinf $ with $ F(0) = 0 $, then
$ F(a) = \T{F'(a)}{a} + R_F(a) $ with
\begin{equation}
\label{eq:paralinearization}
\| R_F(a) \|_{\H{2\alpha-d/2}} \lesssim C(\|a\|_{\H{\alpha}})\|a\|_{\H{\alpha}}.
\end{equation}
In particular,
\begin{equation}
\label{eq:paradiff-estimate-composition-by-function}
\|F(a)\|_{\H{\alpha}} \le C(\|a\|_{\H{\alpha}}) \|a\|_{\H{\alpha}}.
\end{equation}
\end{enumerate}
\end{theorem}

\section{Some Linear Equations}
\label{app:linear-equations}

\begin{proposition}
\label{prop:app-eq-pseudo-well-posedness}
Let $ \udlu \in \Cls^{0,s}(T,\varepsilon_0) $ for~$ s $ sufficiently large, $ T > 0 $, and $ \varepsilon_0 $ sufficiently small. Let $ Q = Q(\udlu) $ be defined by~\eqref{eq:definition-Q}, and suppose 
\begin{equation*}
R \in \Linf([0,T],\L(\Ldottwo,\Ldottwo)) \cap C([0,T],\L(\Ldottwo,\Hdot{-\mu}))
\end{equation*}
for some $ \mu \ge 0 $, $ f \in \Lone([0,T],\Ldottwo(\Td)) $. Then the Cauchy problem
\begin{equation}
\label{eq:app-eq-pseudo}
(\pt + i Q + R) u = f, \quad u(0) = u_0 \in \Ldottwo(\Td),
\end{equation}
admits a unique solution $ u \in C([0,T],\Ldottwo(\Td)) $, which moreover satisfies the estimate
\begin{equation*}
\|u\|_{C([0,T],\H{s})} \lesssim \|u_0\|_{\Ltwo} + \|f\|_{\Lone([0,T],\Ltwo)}
\end{equation*}
\end{proposition}
\begin{proof}
Let $ j_\varepsilon = \exp(-\varepsilon \gamma^{(3/2)}) + \frac{1}{2} \pxi \cdot D_x \exp(-\varepsilon \gamma^{(3/2)}) $, and set $ J_\varepsilon = \pi(D_x) \op(j_\varepsilon \pi) $. Consider the regularized Cauchy problem
\begin{equation*}
(\pt + i Q J_\varepsilon + R J_\varepsilon) u^\varepsilon = f, \quad u^\varepsilon(0) = u_0,
\end{equation*}
which admits a unique solution $ u^\varepsilon \in C([0,T_\varepsilon],\Ldottwo(\Td)) $ for some $ T_\varepsilon > 0 $, by Cauchy-Lipschitz theorem, as $ Q J_\varepsilon, R J_\varepsilon \in C([0,T],\L(\Ldottwo,\Ldottwo)) $. Following a routine method, see for example~\cite{Metivier-Paradifferential-Calculus}, to prove the existence of a solution, on the whole interval $ [0,T] $, it suffice to prove a uniform \textit{a priori} bound for $ u^\varepsilon $ and its time derivative in the energy space over the time interval $ [0,T] $.

By the choice of the symbol $ j_\varepsilon $, we have
\begin{equation*}
\|[Q,J_\varepsilon]\|_{\L(\Ldottwo,\Ldottwo)} \lesssim 1, \quad
\|Q-Q^*\|_{\L(\Ldottwo,\Ldottwo)} \lesssim 1, \quad
\|J_\varepsilon - J_\varepsilon^*\|_{\L(\Ldottwo,\Hdot{2})} \lesssim 1,
\end{equation*}
from which the a priori estimate, that for almost every $ t \in [0,T] $,
\begin{equation}
\label{eq:app-pseudo-derivative-estimate-pre}
\ddt \|u^\varepsilon(t)\|_{\Ltwo}^2 \lesssim \|u^\varepsilon(t)\|_{\Ltwo}^2 + (u^\varepsilon(t),f(t))_{\Ltwo}.
\end{equation}
Moreover, by Gronwall's inequality, we have
\begin{equation}
\label{eq:app-pseudo-Gronwall}
\|u^\varepsilon\|_{C([0,T],\Ltwo)} \lesssim \|u_0\|_{\Ltwo} + \|f\|_{\Lone([0,T],\Ltwo)}.
\end{equation}
Plugging it into~\eqref{eq:app-pseudo-derivative-estimate-pre}, we obtain
\begin{equation}
\label{eq:app-pseudo-derivative-estimate}
\big\| \ddt \|u^\varepsilon(t)\|_{\Ltwo}^2 \big\|_{\Lone([0,T])}
\lesssim \|u_0\|_{\Ltwo}^2 + \|f\|_{\Lone([0,T],\Ltwo)}^2.
\end{equation}
The energy estimate~\eqref{eq:app-pseudo-Gronwall}, the hypothesis on~$ R $, and Arzela-Ascoli's theorem imply the existence of a weak solution 
\begin{equation*}
u \in \Ltwo([0,T],\Ldottwo(\Td)) \cap C([0,T],\Hdot{-\mu}(\Td))
\end{equation*}
to~\eqref{eq:app-eq-pseudo}. Then~\eqref{eq:app-pseudo-derivative-estimate} and Arzela-Ascoli's theorem again implies that $ t \mapsto \|u\|_{\Ltwo}^2 $ is continuous in time. Therefore $ u \in C([0,T],\Ldottwo(\Td)) $. The energy estimate for~$ u $ follows by Gronwall's inequality as in~\eqref{eq:app-pseudo-Gronwall}, and the uniqueness follows from the energy estimate.
\end{proof}

Using the same method, we have the following corollary.
\begin{corollary}
For $ \vec{w}_0 \in \Ldottwo(\Td) $, there exists a unique solution to the Cauchy problem~\eqref{eq:equation-pseudo-diff-sys}, $ \vec{w} \in C([0,T],\Ldottwo(\Td)) $ with $ \vec{w}(0) = \vec{w}_0 $.
\end{corollary}

\begin{corollary}
For $ \vec{w}_{h,0} \in \Ldottwo(\Td) $, $ f \in \Lone([0,T]_s,\Ldottwo(\Td)) $, there exists a unique solution to the Cauchy problem~\eqref{eq:equation-semiclassical-w-f}, $ \vec{w}_h \in C([0,T]_s,\Ldottwo(\Td)) $ with $ \vec{w}_h(0) = \vec{w}_{h,0} $.
\end{corollary}

Similar results hold for the para\-differential equations.
\begin{proposition}
\label{prop:app-wellposedness-paradiff-eq}
Suppose $ s \ge s' \ge 0 $, $ \mu \ge 3 + d/2 $, $ T > 0 $, $ \udlu \in \Cls^{0,\mu}(T,\varepsilon_0) $ for some $ \varepsilon_0 > 0 $ sufficiently small. Let $ P = P(\udlu) $ be defined by~\eqref{eq:definition-P}, 
\begin{equation*}
R \in \Linf([0,T],\L(\Hdot{s},\Hdot{s})) \cap C([0,T],\L(\Hdot{s},\Hdot{s'})),
\end{equation*}
and let $ F \in \Lone([0,T],\Hdot{s}(\Td)) $. Then for $ u_0 \in \Hdot{s}(\Td) $, the following Cauchy problem
\begin{equation}
(\pt + P + R) u = F, \quad u(0) = u_0,
\end{equation}
admits a unique solution $ u \in C([0,T],\Hdot{s}(\Td)) $, which moreover satisfies the estimate
\begin{equation}
\|u\|_{C([0,T],\H{s})} \lesssim \|u_0\|_{\H{s}} + \|F\|_{\Lone([0,T],\H{s})}.
\end{equation}
\end{proposition}
\begin{proof}
The proof is almost the same as above, but here we choose $ J_\varepsilon = T_{j_\varepsilon} $ with $ j_\varepsilon $ defined as above, and use the following estimates
\begin{equation*}
\|[P,J_\varepsilon]\|_{\L(\Hdot{s},\Hdot{s})} \lesssim 1, \quad
\|P-P^*\|_{\L(\Hdot{s},\Hdot{s})} \lesssim 1, \quad
\|J_\varepsilon - J_\varepsilon^*\|_{\L(\Hdot{s},\Hdot{s+2})} \lesssim 1.
\end{equation*}
\end{proof}

\begin{corollary}
\label{cor:Range-Sol-regularity}
Let the Range operator $ \Range $ and the solution operator $ \Sol $ be formally defined in Section~\ref{sec:HUM}, then for all $ \mu \ge 0 $,
\begin{equation*}
\Range : \Ltwo([0,T],\Hdot{\mu}(\Td)) \to \Hdot{\mu}(\Td),  
\qquad
\Sol : \Hdot{\mu}(\Td) \to C([0,T],\Hdot{\mu}(\Td)),
\end{equation*}
and satisfies the estimates
\begin{equation*}
\|\Range\|_{\L(\Ltwo([0,T],\Hdot{\mu}), \Hdot{\mu})} \lesssim 1,
\qquad
\|\Sol\|_{\L(\Hdot{\mu}, C([0,T],\Hdot{\mu}))} \lesssim 1.
\end{equation*}
\end{corollary}

\vspace{\baselineskip}
\parbox{5in}{
HUI ZHU

\textsc{Laboratoire de Mathématiques d’Orsay, Univ.~Paris-Sud, CNRS,}

\textsc{Université Paris-Saclay, 91405 Orsay, France}

\textit{E-mail}: \texttt{hzhu.pde@gmail.com}
}

\end{document}